%% file: main.tex
\documentclass[leqno, preprint]{imsart}
\usepackage[margin=1.4in]{geometry} 
\newcommand\NoCaseChange[1]{ #1 }

\usepackage{pdflscape}

\usepackage[utf8]{inputenc} 
\usepackage{hyperref}       
\usepackage{url}            
\usepackage{booktabs}       
\usepackage{amsfonts}       
\usepackage{nicefrac}       
\usepackage{microtype}      
\usepackage{xcolor}

\usepackage{amssymb}
\usepackage{amsmath}
\usepackage{amsthm}
\usepackage{thmtools}
\usepackage{mathtools}
\usepackage{cleveref}
\input{macros}

\usepackage{constants}
\usepackage[textwidth=2.5cm]{todonotes}

\usepackage{thmtools}
\usepackage{thm-restate}

\usepackage[frozencache=true,cachedir=minted-cache]{minted}

\usepackage{caption}
\usepackage{subcaption}
\usepackage{graphicx}
\usepackage{float}
\graphicspath{{Figures/}}

\declaretheorem[name=Theorem,numberwithin=section]{theorem}
\declaretheorem[name=Proposition,sibling=theorem]{proposition}
\declaretheorem[name=Lemma,sibling=theorem]{lemma}

\declaretheorem[name=Assumption,numberwithin=section]{assumption}
\declaretheorem[name=Definition,style=definition]{definition}

\crefname{assumption}{assumption}{assumptions}

\numberwithin{equation}{section}

\usepackage[inline]{enumitem}

\renewcommand{\theenumi}{(\roman{enumi})}

\newcommand\definitionOfLambdaZero{
    \lambda_0 = \Big(\max_{j=1,...,p}\bSigma_{jj}^{1/2}\Big)
    \frac{\sigma(1+\eta_1)}{\sqrt{nT}}\Big(1 + \sqrt{(2/T)\log(p/s) }\Big)
}

\begin{document}
\title{Chi-square and normal inference in high-dimensional multi-task regression}
\runtitle{Chi-square inference in multi-task regression}
\author{Pierre C. Bellec and Gabriel Romon}
\runauthor{Bellec and Romon}

\begin{abstract}
  The paper proposes chi-square and normal inference methodologies
  for the unknown coefficient matrix $\bB^*$
  of size $p\times T$ in a multi-task
  linear model with
  $p$ covariates, $T$ tasks and $n$ observations
  under a row-sparse assumption on $\bB^*$.
  The row-sparsity $s$, dimension $p$ and number of tasks $T$
  are allowed to grow with $n$.
  In the high-dimensional regime $p\ggg n$, in order to leverage
  the row-sparsity \cite{lounici2011oracle,obozinski2011support},
  the multi-task Lasso is considered.

  We build upon the multi-task Lasso with a de-biasing scheme
  to correct for the bias induced by the penalty.
  The de-biasing scheme requires the introduction of a new data-driven
  object, coined the interaction matrix,
  that captures the effective correlations between 
  noise vector and residuals on different tasks.
  The interaction matrix is symmetric positive semi-definite,
  of size $T\times T$ and can be computed efficiently.

  The interaction matrix lets us derive asymptotic normal
  and asymptotic $\chi^2_T$
  results under general Gaussian design and the rate condition
  ${sT + s\log(p/s)}/{n} \to0$ which corresponds to consistency
  in Frobenius norm of the multi-task Lasso.
  These asymptotic distribution results yield valid
  confidence intervals for single entries of $\bB^*$
  and valid confidence ellipsoids for single rows of $\bB^*$.
  If the covariance of the design
  is unknown, a modification of the multi-task de-biasing scheme
  using the nodewise Lasso provides comparable
  confidence intervals and confidence ellipsoids
  for the $j$-th row of $\bB^*$,
  provided that the $j$-th column of the precision matrix $\bSigma^{-1}$ is
  sufficiently sparse.
  While previous proposals in grouped-variables regression
  require row-sparsity  $s\lesssim \sqrt n$ up to constants depending on $T$
  and logarithmic factors in $(n,p)$ for unknown $\bSigma$,
  the de-biasing scheme
  using the interaction matrix provides confidence intervals
  and $\chi^2_T$ confidence ellipsoids under the conditions
  ${\min(T^2,\log^8p)}/{n} \to 0$ and
  $$
  \frac{sT + s\log(p/s) + \|\bSigma^{-1}\be_j\|_0\log p}{n} \to0,
  \quad
  \frac{\min(s,\|\bSigma^{-1}\be_j\|_0)}{\sqrt n}
  \sqrt{[T+\log(p/s)]\log p} \to 0,
  $$
  allowing for row-sparsity $s\ggg \sqrt n$
  when $\|\bSigma^{-1}\be_j\|_0 \sqrt T \lll \sqrt{n}$
  up to logarithmic factors.
\end{abstract}

\maketitle

\section{Introduction}

\subsection{Model}

We consider a multi-task linear regression model with $T$ tasks,
with $n$ i.i.d. observations $(\bx_i,y^{{(1)}}_i,...,y^{{(T)}}_i)$,
where $\bx_i\in\R^{p}$ is a random feature vector and
$y^{{(1)}}_i,...,y^{{(T)}}_i$ are $T$ different scalar responses.
We assume that on each task $t=1,...,T$, the response $y_i^{{(t)}}$
satisfies a linear model
\begin{equation}
    \label{eq:linear-models-1...T}
    y_i^{{(t)}}
    = \bx_i^\top \bbeta^{{(t)}} + \eps_i^{{(t)}},
    \qquad
    t=1,...,T
\end{equation}
where $\bbeta^{{(t)}} \in \R^p$ is the unknown coefficient vector
on the task $t$.
Throughout, $\bX\in\R^{n\times p}$ is the design matrix with $n$ rows 
$\bx_1^\top,...,\bx_n^\top$. The linear models \eqref{eq:linear-models-1...T} may be
rewritten in vector and matrix form
\begin{equation}
    \label{model-matrix}
    \by^{{(t)}} = \bX\bbeta^{{(t)}} + \bep^{{(t)}},
    \qquad
    \qquad
    \bY = \bX\bB^* + \bE
\end{equation}
where
$\by^{{(t)}} = (y_1^{{(t)}},...,y_n^{{(t)}})^\top$
and
$\bep^{{(t)}} = (\eps_1^{{(t)}},...,\eps_n^{{(t)}})^\top$
are vectors in $\R^n$,
$\bY\in \R^{n\times T}$ is the response matrix with columns $\by^{{(1)}},...,\by^{{(T)}}$,
$\bE\in \R^{n\times T}$ is a noise matrix with columns $\bep^{{(1)}},...,\bep^{{(T)}}$,
and $\bB^*\in\R^{p\times T}$ is an unknown coefficient matrix with
columns $\bbeta^{{(1)}},...,\bbeta^{{(T)}}$.

Estimation of $\bB^*$ in the above multi-task model has been well studied
during the last decade in the high-dimensional regime where $p\ggg n$, see for instance
\cite{lounici2011oracle}. 
This literature on multi-task learning suggests to use a joint convex optimization
problem over the tasks in order to estimate $\bB^*$, namely
$$\hbB
=
\argmin_{\bB\in\R^{p\times T}}
\Big[
    \frac{1}{2nT}\|\bY - \bX\bB \|_F^2 + g(\bB)
\Big]
=
\argmin_{\bB\in\R^{p\times T}}
\Big[
    \frac{1}{2nT}
    \sum_{t=1}^T
    \sum_{i=1}^n
    (y_i^{{(t)}} - \bx_i^\top \bB \be_t)^2
 + g(\bB)
\Big]
$$
where $\be_t\in\R^T$ is the $t$-th canonical basis vector,
$\|\cdot\|_F$ is the Frobenius norm of matrices and $g:\R^{p\times T}\to\R$
is a convex penalty function. The role of the convex penalty $g$ is to
promote a shared structure on the coefficient vectors $\bbeta^{(1)},...,\bbeta^{(T)}$.
The most common shared structure is that of row-sparsity where one assumes that only a few features are relevant across all tasks:
there is a support set
$S\subset \{1,...,p\}$ of small cardinality (relatively to $n,p$) such that
for every task $t=1,...,T$, $\bbeta^{{(t)}}_j =0 \iff j\notin S$.
Equivalently, $\be_j{}^\top \bB^* = {\mathbf 0}_{1\times T}$ if and only if $j\notin S$,
i.e., only $|S|$ rows of $\bB^*$ are nonzero.
In this case, the sparsity pattern encoded by $S\subset\{1,...,p\}$
is shared on all tasks, and previous literature on estimation in this setting
uses a penalty proportional to the $\ell_{2,1}$ norm, $g(\bB)= \lambda \sum_{j=1}^p \|\bB^\top \be_j\|_2$, or alternatively its Elastic-Net version
$g(\bB)= \lambda \sum_{j=1}^p \| \bB^\top \be_j\|_2 + \mu \|\bB\|_F^2$
for non-negative tuning parameters $\lambda,\mu\ge 0$.
If the row-sparsity assumption holds and such $\ell_{2,1}$ penalty is used,
estimation of $\bB^*$ by $\hbB$ is improved compared to estimating
$\bbeta^{(1)},...,\bbeta^{(T)}$ separately \cite{lounici2011oracle}.

\subsection{Noise and residuals: non-trivial correlations for non-separable
penalties}
\label{eq:ls-ridge-non-trivial-correlation}
Classical multivariate statistics studies the least-squares estimate
$\hbB{}^{(ls)} = (\bX^\top\bX)^\dagger\bX^\top\bY$,
which corresponds to $g(\cdot) = 0$ in the above minimization problem.
Here, the estimation on two tasks is independent,
as on the $t$-th task for $t=1,...,T$ we have
$\hbB{}^{(ls)}\be_t = (\bX^\top\bX)^\dagger\bX^\top \by^{(t)}$
for the $t$-th canonical basis vector $\be_t\in\R^T$:
the estimator $\hbB{}^{(ls)}\be_t$ of the unknown regression vector $\bbeta^{(t)}$
on the $t$-th task only depends on the $t$-th response $\by^{(t)}$,
and is independent of the other responses $(\by^{(t')})_{t'\in\{1,...,T\}\setminus \{t\}}$. By independence, if the noise $\bE$ has i.i.d. mean-zero entries,
then
\begin{equation}
    \E[ \bep^{(t')} \be_t^\top(\bY-\bX\hbB{}^{(ls)})^\top]= \mathbf{0}_{n\times n}
\qquad
\forall t\ne t',
\label{eq:uncorrelation-ls}
\end{equation}
i.e., residual and noise on two different tasks are uncorrelated.
A similar story holds for multi-task Ridge regression, which
corresponds to $g(\bB) = \mu \|\bB\|_F^2$ in the above minimization
problem. The optimization problem is separable in the sense that
\begin{align*}
    \hbB{}^{(R)} &= \argmin_{\bB\in\R^{p\times T}}
    \frac{\|\bY-\bX\bB\|_F^2}{2nT}+ \mu\|\bB\|_F^2
    \quad
    \text{ and }
    \quad
    \hbB{}^{(R)}\be_t
    =\argmin_{\bb\in\R^p}
    \frac{ \|\by^{(t)} - \bX\bb\|_2^2 }{2nT} + \mu\|\bb\|_2^2
\end{align*}
equivalently define $\hbB{}^{(R)}$. It follows again that $\hbB{}^{(R)}\be_t$
only depends on the $t$-th response $\by^{(t)}$, and
if $\bE$ has i.i.d. mean-zero entries
then \eqref{eq:uncorrelation-ls}
holds also for $\hbB{}^{(R)}$ by independence.

The situation is more complex for non-separable penalty functions,
for instance if the penalty is proportional to the $\ell_{2,1}$ norm,
$g(\bB) = \lambda\sum_{j=1}^p \|\bB^\top\be_j\|_2$
where $\be_j\in\R^p$ is the $j$-th canonical basis vector.
The corresponding
estimator studied throughout the paper is the multi-task Lasso
\begin{equation}
    \label{eq:hbB-Lasso}
\hbB=\argmin_{\bB\in\R^{p\times T}}
\Big(
\frac{1}{2nT}\|\bY - \bX\bB \|_F^2 + \lambda \|\bB\|_{2,1}
\Big)
\quad\text{where}\quad
\|\bB\|_{2,1}
= \sum_{j=1}^p \|{\bB^{\top} \be_j}\|_2
.
\end{equation}
The estimate $\hbB\be_t$ of the unknown vector $\bbeta^{(t)}$
on the $t$-th task depends in an intricate way on all the responses
including $(\by^{(t')})_{t'\in\{1,...,T\}\setminus \{t\}}$.
Note that this dependence of $\hbB\be_t$ on all responses is
purposeful: we hope to leverage a shared pattern on all tasks
(e.g., if $\bB^*$ is row-sparse and a sparsity pattern
is shared by all $\bbeta^{(t)}, t=1,...,T$
) in order to improve estimation
compared to $\hbB{}^{(ls)}$ or $\hbB{}^{(R)}$.
In this case, however, \eqref{eq:uncorrelation-ls} does not hold
and the correlation between the residual on task $t$ and the noise
on task $t'$ is non-trivial. Our results below (specifically
\Cref{lemma:differentiation-E})
reveal that for $t,t'\in[T]$, 
$$\bigl((\bY-\bX\hbB)\be_t\bigr)^\top\bep^{(t')}
\approx 
\begin{cases}
    \sigma^2(n- \hbA_{tt'}) &\text{ if }  t=   t'\\
    - \sigma^2\hbA_{tt'}         & \text{ if } t\ne t'
\end{cases}
$$
when the noise $\bE$ has i.i.d. $\mathcal N(0,\sigma^2)$ entries
and $\hbA_{tt'}$ is the $(t,t')$ entry of a symmetric matrix
$\hbA\in\R^{T\times T}$ defined in \Cref{sec:A}. This matrix plays
a central role in the present paper to derive asymptotic normality
and asymptotic $\chi^2$ results.

\subsection{Confidence intervals for linear functionals of \texorpdfstring{$\bbeta^{{(1)}}$}{beta\textonesuperior}}
\label{subsec:goal-ci}
A first goal
of the present paper is to provide
confidence intervals for linear functionals of the
regression vector on the first task. Throughout the paper,
regarding asymptotic normality and confidence intervals,
$\ba\in\R^p$ is a fixed direction
of interest and we wish to construct confidence intervals for $\ba^\top
\bbeta^{{(1)}}$. For instance, the direction $\ba\in\R^p$ may be
of the following form.
\begin{enumerate}
    \item a canonical basis vector $\be_j\in\R^p$. For $\ba=\be_j$,
        the goal is to construct confidence intervals for
        $\ba^\top\bbeta^{(1)} = \beta_j^{{(1)}}$, the coefficient
        of the $j$-th feature on the first task.
        This is the classical goal in statistics where one wishes
        to provide inference on the effect of the $j$-th covariate.
    \item a new feature vector $\bx_{new}\in\R^p$, that may for instance
        correspond to the characteristics of a new subject whose responses
        $y_{new}^{{(1)}},...,y_{new}^{{(T)}}$
        are not known yet. The goal is to provide a confidence interval
        for $\ba^\top \bbeta^{(1)}$ which corresponds to the expected response
        of $Y_{new}$ conditionally on the feature vector $\bx_{new}$.
\end{enumerate}
We stress here that the first task ($t=1$) has a special role: the unknown
parameter $\ba^\top \bbeta^{(1)}$ only involves the first unknown coefficient vector $\bbeta^{(1)}$
and not the other coefficient vectors $\bbeta^{(t)}, t=2,...,T$.
If a single linear model $\by^{{(1)}} = \bX\bbeta^{(1)} + \bep^{{(1)}}$
is observed, the construction of confidence intervals for $\ba^\top\bbeta^{(1)}$
has been extensively studied. Most related to the present paper,
\cite{ZhangSteph14,GeerBR14,JavanmardM14a,JavanmardM14b} initially provided
methodologies 
for de-biasing (or de-sparsifying) the Lasso for construction of confidence intervals
in a canonical basis direction $\ba=\be_j$ for sparsity $s\lesssim \sqrt n/\log p$,
\cite{javanmard2018debiasing} extended the sparsity requirement to $s\lesssim n/(\log p)^2$,
\cite{zhu2018linear,bradic2018testability,cai2017confidence,cai2019individualized,zhu2018significance,bellec2019biasing} studied estimation and construction of confidence intervals
in dense direction $\ba\in\R^p$,
and \cite{bellec2019second} extended the de-biasing methodologies to arbitrary convex
penalties.

Of course, one could throw away the responses $\by^{{(2)}},...,\by^{{(T)}}$ and use only the response $\by^{{(1)}}$ with the aforementioned
methodologies, since our goal is to construct
confidence intervals for $\ba^\top\bbeta^{(1)}$. However, throwing away
the responses on tasks $2,...,T$ should intuitively lead to information loss
and is not desirable.

\subsection{Asymptotic \texorpdfstring{$\chi^2$}{chi\texttwosuperior} results and confidence ellipsoids
for rows of \texorpdfstring{$\bB^*$}{B*}}
\label{subsec:goal-chi2}
    The second goal of the paper is to develop
    confidence ellipsoids for whole rows of the unknown matrix
    $\bB^*$. The $j$-th row of $\bB^*$ is
    the vector $(\bB^*)^\top\be_j$ in $\R^T$
    where $\be_j\in\R^p$ is the $j$-th canonical vector.
    Given a confidence level $\alpha\in(0,1)$,
    a confidence ellipsoid for $(\bB^*)^\top\be_j$
    is a subset $\hat{\mathcal E}_{\alpha}$ of $\R^T$
    constructed from the data such that
    $$
    \P\bigl((\bB^*)^\top\be_j \in \hat{\mathcal E}_{\alpha}\bigr) \ge 1-\alpha - o(1)
    $$
    where $o(1)$ converges to 0 as $n\to+\infty$.
    Ideally, the confidence ellipsoid enjoys the exact
    nominal coverage probability $1-\alpha$ asymptotically in the sense 
    that
    \begin{equation}
        \label{eq:intro-exact-quantiles}
    \big|
    \P\bigl((\bB^*)^\top\be_j \in \hat{\mathcal E}_{\alpha}\bigr) - (1-\alpha)
    \big|
    \to 0
    \end{equation}
    as $n\to+\infty$.
    Note that one could also consider confidence sets $\hat{\mathcal E}_\alpha$
    that are not ellipsoids (e.g., hyperrectangles);
    we focus here on ellipsoids as they are the natural confidence sets
    stemming from $\chi^2$-distributed pivotal quantities.
    As in classical multivariate statistics,
    an advantage of confidence ellipsoids is that
    they provide simultaneous confidence intervals
    for every direction $\bb\in\R^T$, that is,
    $\P\bigl(
        \forall \bb\in\R^T, \be_j^\top\bB^*\bb
        \in \{\bb^\top\bu, \bu\in\hat{\mathcal E}_{\alpha} \} 
    \bigr)\to 1-\alpha$
    when \eqref{eq:intro-exact-quantiles} holds 
    and $\hat{\mathcal E}$ is closed and convex.

    Such a confidence ellipsoid allows to perform hypothesis tests of
    \begin{equation}
        H_0: (\bB^{*})^\top\be_j = \mathbf{0}_{T\times 1}
    \qquad
    \text{ against }
    \qquad
    H_1: \|(\bB^{*})^\top\be_j\|_2 \ge \rho,
    \label{H_1-H_0}
    \end{equation}
    where the null hypothesis corresponds to the signal $\bY$
    being independent of the $j$-th feature $\bX\be_j$,
    and $\rho>0$ is a separation radius.
    If a single task is observed ($T=1$), it is impossible 
    to distinguish between the null $\beta_{j}=0$ 
    and the alternative $\beta_{j}\ne 0$ with constant type I and type II
    errors unless $|\beta_j|\ge c \sigma n^{-1/2}$ for some constant $c>0$.
    This follows by noting that the total variation distance
    between $\by^{H_0} = \bX\bbeta^{H_0} + \bep$
    and $\by^{H_1} = \bX\bbeta^{H_1}+\bep$ converges to $0$ 
    if $\bbeta^{H_0},\bbeta^{H_1}$ are the same
    except on coordinate $j$ where $|\beta_j^{H_0}-\beta_j^{H_1}|= a_n$ with
    $a_n=o(\sigma n^{-1/2})$, $\|\bX\be_j\|^2/n \asymp 1$
    and $\bep\sim \mathcal N_n(\mathbf{0},\bI_{n\times n})$,
    for instance 
    by Pinsker's inequality and a standard bound on the Kullback Leibler
    divergence of two multivariate normals.
    If several tasks are observed as in the setting of interest here,
    we will see that it is possible to perform
    the hypothesis test \eqref{H_1-H_0} in situations where
    all nonzero coefficients of $(\bB^*)^\top\be_j$ are of order
    $o(\sigma n^{-1/2})$, i.e., of indistinguishable order 
    when a single task is observed.

    If asymptotic normality results are available for each
    of the $T$ individual
    coefficients of $(\bB^*)^\top\be_j$ (for instance
    such as those described in the previous
    subsection), a natural strategy
    to construct confidence ellipsoids is to sum the square
    of the $T$ asymptotically normal random variables
    and hope that the resulting sum has approximately the $\chi^2$
    distribution with $T$ degrees-of-freedom.
    However, throughout the paper the number of tasks $T$
    is allowed to grow to infinity with $n$ which
    results in some challenges regarding this strategy,
    as pointed out by \cite{mitra2016benefit}.
    For the sake of illustrating the resulting difficulty,
    assume that we have established the asymptotic normality of $T$
    pivotal random variables $U_1,...,U_T$ by
    proving decompositions of the form $U_t = (\hat\sigma/\sigma)Z_t + B_t$
    where $Z_t\sim \mathcal N(0,1)$ and the convergence in probability
    $\hat\sigma/\sigma\smash{\xrightarrow[]{\P}} 1$
    and $B_t\smash{\xrightarrow[]{\P}} 0$ hold,
    so that Slutsky's theorem ensures that the pivotal quantities
    are asymptotically normal
    with $U_t \smash{\xrightarrow[]{d}} \mathcal N(0,1)$.
    Denoting by $\chi^2_T = \sum_{t=1}^T Z_t^2$, summing the squares
    of the pivotal quantities
    and applying the triangle inequality for the Euclidean norm on $\mathbb R^T$
     yields
    \begin{equation}
        \label{eq:challenges-chi2-inference}
    \textstyle
    \big|
    \sqrt{\sum_{t=1}^TU_t^2}
    -
    \sqrt{\chi^2_T}
    \big|
    \le
    |\hat\sigma/\sigma - 1|
    \sqrt{\chi^2_T}
    +
    \sqrt{\sum_{t=1}^TB_t^2}.
    \end{equation}
    While $\E[(\chi^2_T)^{1/2}]$ is of order $\sqrt T$,
    the variance and quantiles of $( \chi^2_T)^{1/2}$ are of 
    constant order
    (specifically,
    $\P((\chi^2_T)^{1/2} - \sqrt{T} \le z_\alpha/\sqrt 2)
    \to 1-\alpha$
    holds by \eqref{eq:approximation-q_T-alpha} below,
    and $\text{Var}[(\chi^2_T)^{1/2}]\to 1/2$ by \cite{stackexchange3376610variance_chi}).
    This implies that a sufficient condition that ensures
    that $(\sum_{t=1}^TU_t^2)^{1/2}$ and $(\chi^2_T)^{1/2}$
    asymptotically share the same quantiles is that
    $\sum_{t=1}^TB_t^2\smash{\xrightarrow[]{\P}}0$ 
    and
    $\sqrt T |\hat\sigma/\sigma - 1|\smash{\xrightarrow[]{\P}} 0$.
    While $B_1\smash{\xrightarrow[]{\P}}0$ and $\hat\sigma/\sigma\smash{\xrightarrow[]{\P}} 1$ are sufficient
    to grant asymptotic normality
    for $U_1$  on the first task,
    the conditions
    $\sqrt T |\hat\sigma/\sigma - 1|\smash{\xrightarrow[]{\P}} 0$
    and
    $(\sum_{t=1}^TB_t^2)^{1/2}\smash{\xrightarrow[]{\P}} 0$
    are much more stringent
    as they involve the number of tasks $T$.

\subsection{Asymptotics and assumptions}
We will derive asymptotic normality and asymptotic $\chi^2_T$
results for a sequence of multi-task regression
problems of increasing dimensions. For each $n$,
we consider the multi-task linear model \eqref{model-matrix} and
and the multi-task Lasso estimate $\hbB$ in \eqref{eq:hbB-Lasso}
where $\bB^*$,
the number of tasks $T$, dimension $p$, tuning parameter
$\lambda$ and row-sparsity $s$ are all 
functions of $n$.
The dependence in $n$ is implicit and will be omitted to avoid notational burden.
We will assume that the sequence of regression problems satisfies the following.

\begin{assumption}
    \label{assum:main}
    \begin{enumerate} 
        \item 
$\bX\in \R^{n\times p}$ is a Gaussian design matrix with i.i.d. $\mathcal N_p(\mathbf 0,\bSigma)$ rows; 
\item 
    $\bB^*\in \R^{p\times T}$ is a row-sparse unknown matrix with at most $s$ nonzero rows;
\item $\bE$ is a Gaussian noise matrix with i.i.d. $\mathcal N(0,\sigma^2)$ entries;
\item $\{s,n,T,p\}$ are positive and satisfy $\frac s n (T+\log \frac p s ) \to 0$ and $n\le p$,
this implies $\frac s p \vee \frac T n\to 0$; 
\item The spectrum of $\bSigma$ is bounded: $C_{\min}\leq \phi_{\min}(\bSigma)\leq \phi_{\max}(\bSigma)\leq C_{\max}$ for some constants $0<C_{\min}\leq C_{\max}$ 
which are independent of $n,p,s,T$;
\item
    $\bSigma$ satisfies $\max_{j=1,...,p}\bSigma_{jj}\le 1$;
\item For two constants $\eta_1,\eta_2>0$,
    the tuning parameter $\lambda$ in \eqref{eq:hbB-Lasso}
    is given by
    \end{enumerate}
    \begin{equation}
        \label{eq:lambda-lambda-0}
        \lambda = (1+\eta_2)\lambda_0,
        \quad \text{where} \quad 
    \definitionOfLambdaZero 
.
    \end{equation}
\end{assumption}

\subsection{Related literature}

For integers $\bar n,\bar p\ge 1$, the multi-task setting
above bears resemblance with the single-response
linear model of the form
\begin{equation}
    \label{linear-model-bar}
    \bar\by
    = \bar \bX \bar\bbeta + \bar\bep
\end{equation}
where $\by\in\R^{\bar n}$, $\bar\bep\in\R^{\bar n}$,
$\bar\bX\in\R^{\bar n\times \bar p}$,
and the features $\{1,...,\bar p\}$ are partitioned into $p$ groups
with equal sizes. Indeed, with $\bar p = pT$, $\bar n = nT$
and by vectorizing the matrices in \eqref{eq:linear-models-1...T},
our multi-task setting is in one-to-one correspondence
with the single-response linear model \eqref{linear-model-bar}
with $\bar\by = \vec(\bY)$, $\bar\bep = \vec(\bE)$,
$\bar\bX$ block diagonal with $T$ blocks each equal to $\bX$,
and the partition $(G_1,...,G_p)$ of
$\{1,...,\bar p\}$ into $p$ groups is given by
$G_j = \{j + (t-1)p, t=1,...,T\}$.
With this correspondence, the estimator $\hbB$ is the group Lasso
$\hat{\bar\bbeta}=\argmin_{\bar\bb\in\R^{\bar p}}
\|\bar \by - \bar\bX \bar\bb\|^2/(2\bar n) + \lambda \|\bar\bb\|_{2,1}
$ where
$\|\bar\bb\|_{2,1} = \sum_{j=1}^p \|\bar\bb_{G_j}\|_2$.
Inference for grouped variables in a single-response linear model
\eqref{linear-model-bar} focuses on estimation, hypothesis tests
or confidence sets
for the vector $\bar\bbeta_{G_j}$ for a group $G_j\subset\{1,...,\bar p\}$
of interest.
In the single task setting \eqref{linear-model-bar} with grouped variables,
\cite{mitra2016benefit} extends the de-biasing methodology
in \cite{ZhangSteph14,GeerBR14}
to inference for grouped variables and provides $\chi^2$ asymptotic
distribution results. 
The paper \cite{mitra2016benefit} already describes some challenges
of chi-square inference in high-dimension (cf. the discussion after
\eqref{eq:challenges-chi2-inference}); the multi-task problem
of the present paper shares some of these challenges,
however our approach and proofs have no overlap with
that of \cite{mitra2016benefit}.
The papers \cite{stucky2018asymptotic,van2016chi} give a different
extension of the de-biasing methodology of \cite{ZhangSteph14,GeerBR14}
to the group setting, again based on the group Lasso, but
here by estimation of the inverse covariance matrix restricted
to the group of interest with a multi-task estimator 
penalized by the nuclear norm.
False Discovery Rate control in single-task linear models
with grouped variables has been studied in 
\cite{brzyski2019group} with a group SLOPE estimator.
Under weak assumptions (in particular, no assumption on $\bX$),
\cite{meinshausen2015group} provides an approach
to inference for grouped variables, although the resulting confidence
regions are conservative.
The papers \cite{mandozzi2016hierarchical,mandozzi2016sequential}
study group inference in a sequence rejection fashion
when the groups are hierarchically ordered.
Bootstrap methods based on the group Lasso are studied in
\cite{zhou2017uncertainty}, without trying to remove the bias.
The paper \cite{guo2019group} develops
conservative inference methods for
quantities of the form $(\bar\bbeta_{G_j})^\top\bA\bar\bbeta_{G_j}$
for a group $G_j\subset[p]$ of interest and a given positive definite matrix $
\bA\in\R^{|G_j|\times |G_j|}$, based on the quadratic program
de-biasing methodology given in \cite{ZhangSteph14,JavanmardM14a}.
Finally, \cite{bellec2019second} introduces a degrees-of-freedom adjustment
for the group Lasso to perform inference on a single coordinate or linear form
of the unknown regression vector in \eqref{linear-model-bar}.

Some papers focus on estimation and inference in the multi-task
model \eqref{model-matrix}.
The papers \cite{van2016chi,bertrand2019handling} study
multi-task models of the form \eqref{model-matrix} where
the noise $\bE\in\R^{n\times T}$ has i.i.d. rows, and the entries
within each row are correlated.
A multi-task extension of the square-root Lasso is developed to
concurrently estimate $\bB^*$ and the correlations in the noise $\bE$.
Such results on estimating the correlations of
the entries in $\bE$ are useful to de-bias the group Lasso
in the single-task model \cite{van2016chi}.
Support recovery through bounds on the group norm
$\|\bB\|_{2,\infty}=\sup_{j\in[p]}\|\bE^\top\be_j\|_2$
is studied in \cite{massias2020support} under a mutual incoherence
assumption on $\bX$.
The mutual incoherence assumption 
requires a row-sparsity level
$s\lesssim \sqrt n$ if $\bX$ has i.i.d. entries.
Closest to the setup and goals of the present paper,
\cite{chevalier2020multi_task_lasso_neurips} extends the de-biasing
methodology of \cite{ZhangSteph14,GeerBR14} to the multi-task setting,
using the nodewise Lasso to estimate a column of the precision matrix
of the design. This approach requires row-sparsity of $\bB^*$ of order
$s\lesssim \sqrt n$ up to logarithmic factors.
Although our approach also involves the nodewise Lasso to estimate
columns of the precision matrix, the de-biasing methodology significantly
differs from \cite{chevalier2020multi_task_lasso_neurips}
and cannot be seen as a straightforward extension of \cite{ZhangSteph14,GeerBR14}: our approach requires the introduction of a data-driven
symmetric matrix $\hbA$ of size $T\times T$ which captures
the interactions between the residuals on different tasks.
Introduction of this novel object lets us significantly relax
the requirement on the row-sparsity of $\bB^*$ while obtaining
normal and $\chi^2_T$ inference results, that are proved
to be non-conservative under some assumption on $T, s, n, p$.

\subsection{Adjustments in high-dimensional inference}
In single-task models, recent literature on high-dimensional
inference has highlighted the necessity to adjust
classical inference principles with scalar adjustments.
To describe such adjustments
consider a single-task linear model $\by=\bX\bbeta + \bep$ 
with $\bbeta\in\R^p$, Gaussian noise $\bep \sim \mathcal N_n(\mathbf{0},\sigma^2\bI_{n\times n})$ 
and $\bX$ with i.i.d. $\mathcal N_p(\mathbf{0},\bSigma)$ rows,
where an initial estimator $\hbbeta{}^{(init)}$
is available. If one is interested in confidence
intervals for the projection $\ba^\top\bbeta$
in some direction $\ba$ normalized with 
$\|\bSigma^{-1/2}\ba\|_2=1$,
a 1-step MLE correction
in direction $\bSigma^{-1}\ba$ \cite{zhang2011statistical},
i.e., maximizing the likelihood over the one-dimensional
model $\{\hbbeta^{(init)} + u\bSigma^{-1}\ba, u\in\R\}$
yields the corrected estimate
\begin{equation}
    \label{eq:one-step-MLE}
\ba^\top\bbeta^{(init)}  + \bz_0^\top (\by-\bX\hbbeta^{(init)})
\|\bz_0\|_2^{-2}
\end{equation}
where $\bz_0 = \bX \bSigma^{-1}\ba$
when $\|\bSigma^{-1/2}\ba\|_2=1$;
and the direction $\bSigma^{-1}\ba$ is the one that maximizes the Fisher information
\cite{zhang2011statistical}.
(Since $\|\bz_0\|_2^2\sim \chi^2_n$ concentrates around $n$,
we allow ourselves to replace $\|\bz_0\|_2^2$ by $n$ in \eqref{eq:one-step-MLE}
in this informal discussion).
In high dimensions, this general principle requires
a modification that accounts for the degrees-of-freedom
of $\hbbeta{}^{(init)}$:
\cite{JavanmardM14b,bellec2019biasing} for the Lasso
and \cite{bellec2019second} for general penalty suggest
to amplify the correction with the degrees-of-freedom adjustment $(1-\df/n)^{-1}$ and to use the estimate
\begin{equation}
    \label{eq:data-driven-adjustmenet-single-task}
    \ba^\top\bbeta^{(init)}  + (1-\df/n)^{-1}\bz_0^\top (\by-\bX\hbbeta^{(init)}) n^{-1}
\end{equation}
instead of \eqref{eq:one-step-MLE}. If $\hbbeta{}^{(init)}$ is the Lasso,
the adjustment $(1-\df/n)^{-1}$ is required for efficiency for large
sparsity levels \cite{bellec2019biasing}.
For the Lasso, the data-driven adjustment $(1-\df/n)^{-1}$ may be replaced
by a deterministic scalar adjustment, i.e.,
\begin{equation}
    \label{eq:deterministic-adjustmenet-single-task}
    \ba^\top\bbeta^{(init)}  + (1-\delta^{-1} s_*)^{-1}\bz_0^\top (\by-\bX\hbbeta^{(init)}) n^{-1}
\end{equation}
where $\delta=n/p$ and $s_*$ is the scalar parameter obtained
after solving the system of two equations with two unknowns
in \cite[Proposition 3.1]{miolane2018distribution}.
The correspondence between $\df/n$ and $s_*$ can be seen
in \cite[Theorem F.1]{miolane2018distribution}
or \cite[Section 3.3]{celentano2020lasso}.
This system of two nonlinear equations first appeared in
\cite{bayati2012lasso} for the Lasso and can be extended
to permutation invariant penalty functions (see \cite{celentano2019fundamental} and the references therein)
and robust M-estimators \cite{thrampoulidis2018precise}.

We are not aware of previous proposals to study
such high-dimensional adjustments in the multi-task setting,
e.g., by extending the data-driven adjustment in 
\eqref{eq:data-driven-adjustmenet-single-task} or the deterministic one in
\eqref{eq:deterministic-adjustmenet-single-task}. 
One goal of the paper is to fill this gap.

\subsection{Contributions}
To summarize \Cref{subsec:goal-ci,subsec:goal-chi2}, the inferential goals of
the paper are twofold:
\begin{enumerate}
    \item 
        To construct valid confidence intervals for a linear functional
        $\ba^\top\bbeta^{(1)}$ of the unknown coefficient on the first task,
        by leveraging responses on all tasks simultaneously.
    \item 
        To construct valid confidence ellipsoids for rows $\be_j^\top\bB^*\in\R^{1\times T}$
        of the unknown coefficient matrix $\bB^*$,
        for instance to provide hypothesis tests on the nullity of
        the $j$-th row of $\bB^*$, or equivalently testing
        that the signal does not depend on the $j$-th covariate.
\end{enumerate}

In order to achieve these statistical goals, we introduce
a new object, the data-driven symmetric matrix $\hbA\in\R^{T\times T}$.
Introduction of the matrix $\hbA$ is
key to equip the estimator $\hbB$ with the inference capabilities (i)
and (ii) above, as the theory and simulations of the next sections will show.
This data-driven matrix $\hbA$ generalizes, to the multi-task
setting, the effective degrees-of-freedom
and other scalar adjustments in single-task linear models discussed 
in the previous subsection. Since $\hbA$ is symmetric, $T(T+1)/2$ scalar
adjustments are necessary in the multi-task setting and that number
of adjustments 
is unbounded if $T\to+\infty$ as a function of $n$. The fact that 
a growing, unbounded number of scalar adjustments would be necessary
to achieve the above inference capabilities in the multi-task
setting was surprising---at least to us---, since existing works
on adjustments in high-dimensional statistics so far
only require a bounded number of scalar adjustments.

The paper also includes contributions related to the performance
of the multi-task estimator $\hbB$ in \eqref{eq:hbB-Lasso}.
We improve the logarithmic dependence in tuning parameter $\lambda$
and the known upper bounds on $\|\hbB-\bB^*\|_F$ and $\|\bX(\hbB-\bB^*)\|_F$
compared to \cite{lounici2011oracle}. We also develop tools to
show that the random matrix $\bX$ enjoys a multi-task
Restricted Eigenvalue (RE) condition from \cite{bickel2009simultaneous}.
Although the single-task
case follows in a straightforward manner from Gordon's
escape through a mesh theorem (e.g., \cite{raskutti2010restricted}),
the multi-task version of the RE condition for the random matrix $\bX$
requires different tools.

\subsection{Organization}
The rest of the paper is organized as follows.
The next section summarizes notation.
\Cref{sec:A} describes a new quantity, the interaction matrix $\hbA$
that plays a major role in our estimates and confidence intervals.
\Cref{sec:known-Sigma} constructs confidence intervals 
for $\ba^\top\bbeta^{(1)}$ when the covariance matrix $\bSigma$ of the design
is known.
\Cref{sec:unknown-Sigma} extends these results and methodologies 
when $\bSigma$ is unknown.
\Cref{sec:chi2} develops confidence ellipsoids for rows of $\bB^*$.
\Cref{sec:computing} provides an efficient way of computing the interaction matrix.
\Cref{sec:simu} presents numerical experiments that corroborate
our theoretical findings.
The proofs are deferred to appendices and some intuition
behind the main technical argument is given in
\Cref{sec:intuition}.

\subsection{Notation}
\label{sec:notation}

Throughout the paper, the linear model vector and matrix notation
\eqref{model-matrix} holds. 
$T$, $p$ and $s$ are all non-decreasing functions of $n$. In all the displays of convergence (e.g., $\to$, $\lim$, $o(\cdot)$, $O(\cdot)$), we implicitly mean that $n$ goes to $\infty$.
Convergence in distribution and in probability are denoted by $\smash{\xrightarrow[]{d}}$ and $\smash{\xrightarrow[]{\P}}$.

Estimators of the unknown $\bB^*$ are denoted by
$\hbB$.
For any real $a$, $a_+=\max(0,a)$ and $[k]= \{1,...,k\}$ for any integer $k$,
e.g., $[n], [p], [T]$. 
We use indices $i,i',i_1,i_2,...$ to sum or loop
over $[n]$ (i.e., over the $n$ observations),
indices $t,t',t_1,t_2,...$ to sum or loop over $[T]$ (i.e., over the $T$ tasks),
indices $j,j',j_1,j_2,...$ to sum or loop over $[p]$ (i.e., the $p$ covariates).
The vectors $\be_j\in\R^p$, $\be_t\in\R^T$, $\be_i\in\R^n$ denote the canonical basis
vector of the corresponding index; the size of such canonical vector
will be made explicit if it is not clear from context.
The identity matrices of sizes $p\times p$,
$n\times n$, $T\times T$ are $\bI_{p\times p}, \bI_{n\times n}$ and $\bI_{T\times T}$
respectively
and ${\mathbf 0}_{k\times q}$ is the zero matrix with $k$ rows and $q$ columns.

For any $q\ge 1$, $\|\cdot\|_q$ is the $\ell_q$-norm of vector,
e.g., $\|\cdot\|_2$ is the Euclidean norm.
For any matrix $\bM$, $\|\bM\|_F$ is the Frobenius norm
and $\|\bM\|_{op} = \sup_{\|\bu\|_2=1} \|\bM\bu\|_2$ the operator norm, also known as the spectral norm.
If $\bM$ is symmetric, $\phi_{\min}(\bM)$ (resp. $\phi_{\max}(\bM)$) denotes the smallest (resp. largest) eigenvalue of $\bM$.
The Moore-Penrose pseudoinverse of matrix $\bM$ is denoted by $\bM^\dagger$.
The Kronecker product between two matrices $\bU$, $\bV$ with
$\bU\in\R^{k\times q}$ is
\begin{equation*}
    \bU\otimes \bV
    \coloneqq
\begin{pmatrix}
    u_{11} \bV &  \dots & u_{1q} \bV \\
    \vdots & \vdots & \vdots \\
    u_{k1} \bV &  \dots & u_{kq} \bV \\
 \end{pmatrix}
 \text{so that }
    \bI_{T\times T}\otimes \bX =
\begin{pmatrix}
    \bX & {\mathbf 0}_{n\times p} & \dots & {\mathbf 0}_{n\times p}
    \\
    {\mathbf 0}_{n\times p} &\bX &  \dots & {\mathbf 0}_{n\times p} &
    \\
    \vdots & \vdots & \vdots & \vdots
    \\
    {\mathbf 0}_{n\times p} & \dots & \dots {\mathbf 0}_{n\times p} &\bX
 \end{pmatrix}
\end{equation*}
for $\bX\in\R^{n\times p}$. We will use the mixed product property
    of Kronecker products,
\begin{equation}
    \label{mixed-product-property}
    (\bU\otimes \bV)
    (\bP\otimes \bQ)
    = (\bU \bP) \otimes  (\bV \bQ),
    \qquad
    (\bU\otimes \bV)^\dagger
    =
    \bU^\dagger \otimes \bV^\dagger
\end{equation}
whenever the dimensions are such that the matrix products $\bU\bP$ and $\bV\bQ$ make sense.
The following trace property also holds
\begin{equation}
    \label{trace-property-kronecker}
    \trace[\bU\otimes \bV]
    =\trace[\bU] \trace[\bV].
\end{equation}
If $\|\cdot\|$ denotes a Schatten norm (e.g., Frobenius or spectral norm), then for any $\bU$, $\bV$ we have
\begin{equation}
    \label{norm-property-kronecker}
    \|\bU\otimes \bV\|
    =\|\bU\| \|\bV\|.
\end{equation}
We define the vectorization $\vec(\bU)$
of any matrix $\bU\in\R^{m\times q}$ by stacking vertically the columns of $\bU$
into a column vector in $\R^{qm\times 1}$, i.e.,
\setcounter{MaxMatrixCols}{20} 
$$\vec(\bA)^\top = 
    \begin{pmatrix}
         u_{11} & u_{21} & \dots & u_{m1}
       & u_{12} & u_{22} & \dots & u_{m2}
       & \dots 
       & u_{1q} & u_{2q} & \dots & u_{mq}
   \end{pmatrix}.$$
For any three matrices $\bA,\bB,\bC$ such that the matrix product $\bA\bB\bC$
makes sense, the above vectorization operator satisfies
\begin{equation}
\vec(\bA\bB\bC)
=
(\bC^\top\otimes \bA) \vec(\bB).
\label{kronecker-vectorization-relation}
\end{equation}
These many properties of Kronecker products are referenced in Section 4.2 of \cite{horn1991topics}.

We consider restrictions of vectors (respectively matrices)
by zeroing the corresponding entries (respectively columns). More precisely,
if $\bv\in\R^p$ and $B\subset [p]$ then $\bv_B\in\R^p$ is the vector
with $(\bv_B)_j = 0$ if $j\notin B$ and $(\bv_B)_j = v_j$ if $j\in B$.
If $\bX\in\R^{n\times p}$ and $B\subset[p]$, $\bX_B\in\R^{n\times p}$
is a matrix of the same
dimension as $\bX$ such that $(\bX_B)\be_j = {\mathbf 0}$ if $j\notin B$ and
$(\bX_B)\be_j = \bX\be_j$ if $j\in B$, i.e., $\bX_B$ is a copy of
$\bX$ after having zeroed the columns not indexed in $B$.
Finally, $I\{\Omega\}$ denotes the indicator function of an event $\Omega$,
and $I\{i\in B\} = 1$ if $i\in B$ and $I\{i\in B\} = 0$ if $i\notin B$
is the indicator that some index $i$ belongs to $B$.

\section{The interaction matrix \texorpdfstring{$\hbA$}{A} of the Multi-Task Lasso estimator}
\label{sec:A}

We consider the multi-task Lasso estimator, with $\ell_{2,1}$ penalty,
given \eqref{eq:hbB-Lasso} for some tuning parameter $\lambda>0$.
Let $\hat S = \{j\in [p]: {\hbB{}^{\top} \be_j} \neq 0\}$ denote the set of nonzero rows of $\hbB$. We will refer to $\hat S$ as the support of $\hbB$
and denote by $|\hat S|$ its cardinality.
The above estimator is the one commonly used in the multi-task learning literature
under a row-sparsity assumption on $\bB^*$, see, e.g., \cite{lounici2011oracle}.
Recall that $\bX_{\hat S} \in\R^{n\times p}$ is a copy of $\bX$
obtained after zeroing the columns not belonging to $\hat S$. Define
$\tbX
\coloneqq
 \bI_{T\times T}
 \otimes
 \bX_{\hat S}
$
where $\otimes$ denotes the Kronecker product defined in \Cref{sec:notation},
so that $\tbX\in\R^{nT\times pT}$ is block-diagonal with $T$ blocks,
each equal to $\bX_{\hat S}$.
Consequently $\tbX{}^\top\tbX = \bI_{T\times T} \otimes (\bX_{\hat S}^\top\bX_{\hat S})
\in\R^{(pT)\times (pT)}$ is also block-diagonal with $T$ blocks
equal to $\bX_{\hat S}^\top \bX_{\hat S}$.
For any $j\in\hat S$, define the matrix
\begin{equation}
    \label{Hj}
\bH^{(j)}
\coloneqq
\lambda
\|\hbB{}^\top \be_j\|_2^{-1}\left(\bI_{T\times T} - \hbB{}^\top\be_j \be_j^\top\hbB ~  \|\hbB{}^\top\be_j\|_2^{-2} \right) \quad \in\R^{T\times T}
\end{equation}
and note that $\bH^{(j)}$ is proportional to an orthogonal projection of rank $T-1$.
The matrix $\bH^{(j)}$ is the Hessian of $\bu\mapsto \lambda \|\bu\|_2$
at $\bu=\hbB{}^\top\be_j$.
Finally, let 
$\tbH \in\R^{(pT) \times (pT)}$ be the matrix defined by
$\tbH \coloneqq \sum_{j\in \hat S} \bH^{(j)} \otimes (\be_j\be_j^\top)$.

\begin{definition}
The interaction matrix $\hbA\in \R^{T\times T}$ of the estimator $\hbB$
in \eqref{eq:hbB-Lasso}
is defined entrywise by
\begin{equation}
\label{def-A}
\hbA_{tt'} \coloneqq
\trace\left(
    \left[
    \begin{array}{c|c|c}
        {\mathbf 0}_{n\times p(t-1)} 
        & \bX_{\hat S}
        & {\mathbf 0}_{n\times p(T-t)} 
    \end{array}
    \right]
    \left[\tbX{}^\top\tbX + nT \tbH \right]^{\dagger}
    \left[
    \begin{array}{c}
        {\mathbf 0}_{p(t'-1)\times n} \\[1mm] \hline \rule{0mm}{4mm}  
        (\bX_{\hat S})^\top
        \\[1mm]  \hline {\mathbf 0}_{p(T-t')\times n}
    \end{array}
    \right]
\right)
\end{equation}
for all $t,t'\in[T]$, where $\dagger$ denotes the Moore-Penrose inverse.
Equivalently, if $\bu,\bv\in\R^T$ then
\begin{equation*}
\bu^\top \hbA \bv
=
\trace\Big(
    \left[
    \begin{array}{c|c|c|c}
        u_1 \bX_{\hat S}
        &
        u_2 \bX_{\hat S}
        &
        \dots
        &
        u_T \bX_{\hat S}
    \end{array}
    \right]
    \left[\tbX{}^\top\tbX + nT \tbH \right]^{\dagger}
    \left[
    \begin{array}{c|c|c|c}
        v_1 \bX_{\hat S}
        &
        v_2 \bX_{\hat S}
        &
        \dots
        &
        v_T \bX_{\hat S}
    \end{array}
    \right]^\top
\Big),
\end{equation*}
or with Kronecker product notation, \begin{equation}
	\label{eq:quadra-A}
        \bu^\top\hbA\bv = \trace\big[ (\bu^{\top} \otimes \bX_{\hat S})
    [\tbX{}^\top\tbX + nT \tbH ]^{\dagger}
    (\bv \otimes (\bX_{\hat S})^{{\top}})
\big].
\end{equation}

\end{definition}
Observe that $\sum_{j\in \hat S}(\be_j\be_j^\top)\otimes \bH^{(j)}$ is a block-diagonal matrix
with $p$ diagonal blocks equal to $I\{j\in \hat S\}\bH^{(j)}$.
For $\bA,\bB$ any
square matrices, $\bA \otimes \bB = \bP (\bB\otimes \bA) \bP{}^\top$ holds
for a permutation matrix $\bP$ that only depends on the dimensions
of $\bA$ and $\bB$.
This permutation $\bP$ is particularly simple and known as a perfect shuffle.
It follows that $\bP \tbH \bP{}^\top$ is block diagonal with $p$ diagonal blocks
for some permutation matrix $\bP\in\R^{pT\times pT}$.
Thus the matrix
\begin{equation}
\tbX{}^\top \tbX + nT \tbH\qquad\in \R^{pT\times pT}
\label{eq:tbX-tbX-nT-tbH}
\end{equation}
appearing in \eqref{def-A}-\eqref{eq:quadra-A}
is the sum of two matrices of size $pT\times pT$, each summand
being block diagonal but in a different basis.
If $\lambda=0$ then $\tbH=\mathbf{0}$ and $\hbA$ is diagonal
as $\tbX{}^\top\tbX + nT\tbH$ can be inverted by block.
This corresponds to the unregularized least-squares estimate
$\hbB{}^{(ls)}$ discussed in \eqref{eq:ls-ridge-non-trivial-correlation}
with $\hbB{}^{(ls)}\be_t$ depending on the $t$-th response $\by^{(t)}$ only.
In the case $\lambda>0$ of interest here, the matrix $\tbH$ induces
nonzero entries outside of the $T$ diagonal blocks of $\tbX{}^\top\tbX$,
the matrix \eqref{eq:tbX-tbX-nT-tbH} is not diagonal by block
and the resulting matrix $\hbA$ is not diagonal.
Additional structure in \eqref{eq:tbX-tbX-nT-tbH} and $\hbA$
is studied  in \Cref{sec:computing}, which yields an efficient
and practical algorithm to compute $\hbA$.


The interaction matrix plays a major role in the construction
of our confidence intervals for $\ba^\top\bbeta^{(1)}$
as well as for chi-square inference regions for rows of $\bB^*$.
A high-level interpretation
of its role is that $\hbA$ captures the correlation between the residuals
on different tasks. The following proposition summarizes some useful
properties of $\hbA$.
Result (iii) is important as our confidence interval for $\ba^\top\bbeta^{(1)}$
defined in the next section will involve the inverse of $\bI_{T\times T} - \hbA/n$.
\Cref{propMatrixA} is proved in \Cref{sec:proof-matrix} of the supplement.

\begin{restatable}{proposition}{propMatrixA}
    \label{propMatrixA}
    Let $\hbA$ be defined by \eqref{def-A}. Then
    \begin{enumerate}[partopsep=0pt,topsep=0pt,itemsep=0pt,parsep=0pt]
        \item $\hbA$ is symmetric and positive semi-definite.
        \item If $\bX_{\hat S}$ is rank $|\hat S|$ 
            then the spectral norm of $\hbA$ is bounded from above as $\|\hbA\|_{op}\leq|\hat S|$.
        \item If $\bX_{\hat S}$ is rank $|\hat S|$ and ${|\hat S|}/{n}<1$ then $\bI_{T\times T} - \hbA/n$ is positive-definite
            and \\
            $\|\bI_{T\times T} - (\bI_{T\times T} - \hbA/n)^{-1}\|_{op} 
            \le
            (|\hat S|/n)/(1-|\hat S|/n)$.
    \end{enumerate}
\end{restatable}

\section{Asymptotic normality and confidence intervals in the multi-task setting}
\label{sec:normal-thms}

\subsection{Known \texorpdfstring{$\bSigma$}{Sigma}: Pivotal random variable, asymptotic normality
and confidence intervals}
\label{sec:known-Sigma}

We assume throughout this section that the direction $\ba$ of interest
is normalized with $\|\bSigma^{-1/2}\ba\|_2=1$. This normalization
assumption is relaxed in the next \Cref{sec:unknown-Sigma} where
we develop a methodology for unknown $\bSigma$.
If $\bSigma$ is known, our main result is the following
where $\hbA$ denotes the interaction matrix \eqref{def-A}.

\begin{restatable}{theorem}{ThmNormal}
    \label{thm-normal}
    Let \Cref{assum:main}
    be fulfilled. 
    Assume that $\|\bSigma^{-1/2}\ba\|_2^2=1$. If $\bz_0=\bX\bSigma^{-1}\ba$ then
    \begin{equation}
        \label{eq:asymptotic-normality-bb}
        \frac{
            n \ba^T (\hbB - \bB^*) \bb
            + \bz_0^T(\bY-\bX\hbB)(\bI_{T\times T}-\hbA/n)^{-1}
            \bb
        }
        {
        \|(\bY-\bX\hbB)
        (\bI_{T\times T}-\hbA/n)^{-1}
        \bb\|_2
        }
         \quad
         \smash{\xrightarrow[]{d}}
         \quad
         \mathcal N(0,1)
  \end{equation}
  for any $\bb\in\R^{T}$.
  Hence for $\bb=\be_1\in\R^T$, the parameter $\ba^\top\bbeta^{(1)}$
  of interest satisfies
    \begin{equation}
        \label{eq:asymptotic-normality-e_1}
        \frac{
            n(\ba^T \hbB\be_1 - \ba^\top\bbeta^{(1)})
            + \bz_0^T(\bY-\bX\hbB)(\bI_{T\times T}-\hbA/n)^{-1}
            \be_1
        }
        {
        \|(\bY-\bX\hbB)
        (\bI_{T\times T}-\hbA/n)^{-1}
        \be_1\|_2
        }
         \quad
         \smash{\xrightarrow[]{d}}
         \quad
        \mathcal N(0,1)
         .
  \end{equation}
\end{restatable}
\Cref{thm-normal} is proved in \Cref{sec:proof-thm-normal}.
The left-hand sides of both displays in \Cref{thm-normal} can be interpreted
as Z-scores that have asymptotically standard normal distribution.
In the second display, the only unknown quantity on the left hand side
is $\ba^\top\bbeta^{(1)}$, the parameter of interest
(while in the first display, the only unknown quantity is the scalar $\ba^\top\bB^*\bb$).
Consequently if $z_{\alpha/2}$ is the $1-\alpha/2$ quantile of the standard normal distribution such that $\P(|\mathcal N(0,1)|\le z_{\alpha/2})=1-\alpha$,
an asymptotic $1-\alpha$ confidence interval for $\ba^\top\bbeta^{(1)}$ is given by
$[L_-^\alpha,L_+^\alpha]$ where
$$
L_{\pm}^\alpha =
\underbrace{
\ba^T \hbB\be_1
}_{\substack{\text{initial} \\ \text{estimate}}}
+ 
\underbrace{
    \frac{\bz_0^\top(\bY-\bX\hbB)(\bI_{T\times T}-\frac{\hbA}{n})^{-1}\be_1}{n}
}_{\text{bias correction using the interaction matrix}}
\pm
\underbrace{
\frac{z_{\alpha/2}
\|(\bY-\bX\hbB)
(\bI_{T\times T}-\frac{ \hbA}{n})^{-1}
\be_1\|_2}{n}
}_{\text{confidence interval half-length}}
.
$$
\eqref{eq:asymptotic-normality-e_1} in \Cref{thm-normal}
states that $\P(\ba^\top\bbeta^{(1)} \in [L_-^\alpha,L_+^\alpha]) \to(1-\alpha)$ as $n,p\to+\infty$.

The confidence interval is centered at 
$\ba^T \hbB\be_1$ (which can be interpreted as the initial estimate of $\ba^\top\bbeta^{(1)}$
given by the estimator $\hbB$ in \eqref{eq:hbB-Lasso})
plus a de-biasing correction
$\bz_0^T(\bY-\bX\hbB)(\bI_{T\times T}-\hbA/n)^{-1}$ that involves the interaction
matrix $\hbA$ through the matrix inverse 
\begin{equation}
    \label{matrix-inverse}
    (\bI_{T\times T}-\hbA/n)^{-1}.
\end{equation}
The fact that penalized estimators such as \eqref{eq:hbB-Lasso} require
a de-biasing correction should be expected since it is already the case
for $T=1$ for the Lasso
\cite{ZhangSteph14,GeerBR14,JavanmardM14a,JavanmardM14b,javanmard2018debiasing,bellec2019biasing} and any regularized least-squares
\cite{bellec2019second}.
However, the apparition in the de-biasing correction of the interaction matrix
through the matrix inverse \eqref{matrix-inverse} is surprising
at least to us: we did not expect the multi-task de-biasing correction
to require a matrix inversion such as \eqref{matrix-inverse} when initially
tackling this problem.
The length of the confidence interval above is
$2z_{\alpha/2} n^{-1 }\|(\bY-\bX\hbB) (\bI_{T\times T}-\hbA/n)^{-1} \be_1\|_2$ when $\bb=\be_1$,
and an estimate of this norm is given by the following theorem.

\begin{theorem}
    \label{thm:variance}
    Let the assumptions and setting  of \Cref{thm-normal} be fulfilled.
    Then \\
    $\|(\bY-\bX\hbB)
    (\bI_{T\times T}-\hbA/n)^{-1}
    \bb\|_2^2/n \; \smash{\xrightarrow[]{\P}} \; \sigma^2$
    when $\|\bb\|_2=1$.
\end{theorem}

Consequently the length of the confidence interval is approximately
$2z_{\alpha/2}\sigma n^{-1/2}$ which is the typical length for
two-sided confidence intervals for an unknown mean $\mu$ when observing i.i.d. $Y_1,...,Y_n$
with $\E[Y_i]=\mu, \mathrm{Var}[Y_i]=\sigma^2$.
\Cref{thm-normal,thm:variance} are proved together in \Cref{sec:proof-thm-normal}.

\paragraph*{Comparison with single-task Lasso on the first task.}
\label{para:width}
It is instructive to compare the above confidence interval with the confidence interval
induced by a single-task Lasso estimator computed on $(\bX,\by^{(1)})$,
i.e., when throwing away the responses $\by^{(2)},...,\by^{(T)}$ on tasks $2,...,T$.
This is also a good opportunity to analyse the form of
$\hbA$ and the matrix inversion \eqref{matrix-inverse} in the degenerate case
where a single task is observed.

For $T=1$, a response vector $\by^{(1)}=\bX\bbeta^{(1)} + \bep^{(1)}$ in $\R^n$
is observed and
the estimator \eqref{eq:hbB-Lasso} reduces to the usual Lasso
with response vector $\by^{(1)}$,
$$\hbbeta^L = \argmin_{\bb\in\R^p} \|\by^{(1)} - \bX\bb\|^2/(2n) + \lambda \|\bb\|_1.$$
The asymptotic normality result in \Cref{thm-normal} for $\bb=1$ asserts that
\begin{equation}
    \frac{ n\ba^\top(\hbbeta^L - \bbeta^{(1)})
    + (1-\hbA_{11}/n)^{-1} \bz_0^\top(\by^{(1)}-\bX\hbbeta^L
    )
}{
(1-\hbA_{11}/n)^{-1}\|\by^{(1)} - \bX\hbbeta^L\|_2
} 
\quad\smash{\xrightarrow[]{d}}\quad
\mathcal N(0,1).
\label{eq:normality-lasso-T=1}
\end{equation}
In the degenerate case $T=1$, the matrices in \eqref{Hj} are all zeros
and the matrix $\hbA$ reduces to a scalar $\hbA_{11}$ equal to
$\trace[\bX(\bX_{\hat S^L}^\top\bX_{\hat S^L})^\dagger \bX^\top] = |\hat S^L|$
where $\hat S^L$ is the support of the Lasso $\hbbeta^L$.
Here $\hbA_{11}$ is the usual effective degrees-of-freedom for the Lasso.
The factor $(1-\hbA_{11}/n) = (1-|\hat S^L|/n)^{-1}$ in \eqref{eq:normality-lasso-T=1} is the
degrees-of-freedom adjustment for the Lasso studied in \cite{bellec2019biasing},
which is required for the asymptotic normality result \eqref{eq:normality-lasso-T=1}
when
$s\gtrsim n^{2/3}$ \cite{bellec2019biasing}.
So \Cref{thm-normal} reduces to the asymptotic normality result of \cite{bellec2019biasing} in the degenerate case $T=1$, and in this case the matrix inversion
\eqref{matrix-inverse} reduces to a degrees-of-freedom adjustment
through the scalar multiplication by $(1-|\hat S^L|/n)^{-1}$.
The length of the resulting confidence interval for $\ba^\top\bbeta^{(1)}$
when $T=1$ (or when the tasks $2,...,T$) are thrown away) is then
\begin{equation}
    \label{eq:length-only-one-task}
    2
    z_{\alpha/2}n^{-1}\|\by^{(1)}-\bX\hbbeta^L\|_2(1-|\hat S^L|/n)^{-1}.
\end{equation}
We may compare the lengths of the two confidence intervals:
\begin{itemize}
    \item The confidence interval $[L_-^\alpha,L_+^\alpha]$ based on \eqref{eq:asymptotic-normality-e_1} using the responses on all tasks $1,...,T$ with
        length 
        $2n^{-1}z_{\alpha/2}
\|(\bY-\bX\hbB)
(\bI_{T\times T}-\hbA/n)^{-1}
\be_1\|_2$, and
    \item
        The confidence interval based on \eqref{eq:normality-lasso-T=1}  obtained by throwing away the responses on tasks $2,...,T$ with length \eqref{eq:length-only-one-task}.
\end{itemize}
The length of the confidence interval based on $\hbB$ and the responses on all tasks
$1,...,T$ is smaller than the length \eqref{eq:length-only-one-task} only when
\begin{equation}
    \label{eq:length-comparison}
    \|\by^{(1)}-\bX\hbbeta^L\|_2(1-|\hat S^L|/n)^{-1}
    >
\|(\bY-\bX\hbB)
(\bI_{T\times T}-\hbA/n)^{-1}
\be_1\|_2.
\end{equation}
Our simulations in \Cref{sec:simu} (see \Cref{fig:Width}) reveal
that \eqref{eq:length-comparison} holds,
in some situations with significant margins,
when $s$ is not too large.
Since the comparison \eqref{eq:length-comparison} can be performed
by looking at the data, the practitioner should choose the multi-task confidence interval
based on \eqref{eq:asymptotic-normality-e_1}
over the single-task confidence interval based on \eqref{eq:normality-lasso-T=1}
when \eqref{eq:length-comparison} holds.
When performing this comparison, two tests are constructed which calls for
a Bonferroni correction to avoid invalid coverage due to multiple testing.

\subsection{Unknown \texorpdfstring{$\bSigma$}{Sigma}: Pivotal random variable, asymptotic normality
and confidence intervals}
\label{sec:unknown-Sigma}
The knowledge of $\bSigma$ is not available in most practical situations
and the methodology of the previous subsection cannot be applied.
Indeed the left hand sides in \Cref{thm-normal}
involve $\bz_0=\bX\bSigma^{-1}\ba$ which cannot be directly constructed from
the data when $\bSigma$ unknown.
Another issue that arises when $\bSigma$ is unknown is that one cannot
verify the normalization $\|\bSigma^{-1/2}\ba\|_2=1$ required in \Cref{thm-normal}.
Intuitively, though, if it was possible to estimate both
$\bz_0=\bX\bSigma^{-1}\ba$ and $\|\bSigma^{-1/2}\ba\|_2$ fast enough, replacing
these quantities by their estimates in \eqref{eq:asymptotic-normality-e_1} should
not break asymptotic normality.
Following ideas from the early de-biasing literature \cite{ZhangSteph14,JavanmardM14b,GeerBR14}, we consider a direction
\begin{equation}
    \ba=\be_j
\end{equation}
for some fixed covariate $j\in\{1,...,p\}$ and compute the nodewise Lasso
\begin{equation}
    \label{lasso-gamma-j}
    \hbgamma^{(j)} = \argmin_{\bgamma\in\R^{p}}
    \|\bX\be_j - \bX_{-j}\bgamma\|_2^2/(2n)
    + \hat\tau_j(1+\eta)\sqrt{(2/n)\log p}\|\bgamma\|_1
\end{equation}
for regressing $\bX\be_j$ on $\bX_{-j}$,
where $\bX_{-j}\in\R^{n\times p}$ is the matrix $\bX$ with $j$-th column replaced
by a column of zeros,
$\hat\tau_j$ is a consistent estimate of $\|\bSigma^{-1/2}\be_j\|_2^{-1}$
and $\eta>0$ is a small constant.
Alternatively, one may use the scale invariant version of \eqref{lasso-gamma-j}
again for regressing $\bX\be_j$ on $\bX_{-j}$,
\begin{equation}
    \label{scaled-lasso-gamma-j}
    \hbgamma^{(j)} = \argmin_{\bgamma\in\R^{p}:\gamma_j=0}
    \big(\|\bX\be_j - \bX_{-j}\bgamma\|_2^2/(2n)\big)^{1/2}
    + (1+\eta)\sqrt{(2/n)\log p}\|\bgamma\|_1,
\end{equation}
known as Scaled lasso \cite{sun2012scaled}
or square-root Lasso \cite{belloni2011square},
and \eqref{scaled-lasso-gamma-j} is equal to \eqref{lasso-gamma-j}
with $\hat\tau_j = \|\bX\be_j - \bX_{-j}\hbgamma^{(j)}\|_2/\sqrt n$.
We finally set 
\begin{equation}
    \hbz_j = \bX\be_j - \bX_{-j} \hbgamma^{(j)}.
    \label{eq:def-hbz-j}
\end{equation}
    This corresponds to the residuals of the estimator $\hbgamma^{(j)}$
    in the linear model
    \begin{equation}
        \label{eq:linear-model-nodewise}
        \bX \be_j
        =
        \bX_{-j} \bgamma^{(j)}
        + \bep^{(j)}
    \end{equation}
    with response vector $\bX\be_j\in\R^n$,
    design matrix $\bX_{-j}$,
    true regression vector 
    $\bgamma^{(j)}\coloneqq
    - \|\bSigma^{-1/2}\be_j\|_2^{-2} (\bI_p - \be_j\be_j^\top)\bSigma^{-1}\be_j
    $
    (so that $\be_j^\top\bgamma^{(j)} = 0$
    and $
    \be_k^\top\bgamma^{(j)}= - (\bSigma^{-1})_{jj}^{-1} (\bSigma^{-1})_{jk}$ for $k\in[p]\setminus\{j\}$),
    and Gaussian noise vector
    $\bep^{(j)}\coloneqq \|\bSigma^{-1/2}\be_j\|_2^{-2} \bX\bSigma^{-1}\be_j$
    independent of $\bX_{-j}$
    with distribution
    $\bep^{(j)} \sim \mathcal N_n(\mathbf{0}, \tau_j^2 \bI_{n\times n})$ where $\tau_j^2 \coloneqq \|\bSigma^{-1/2}\be_j\|_2^{-2} = (\bSigma^{-1})_{jj}^{-1}$. 
    The relationship between $\bSigma^{-1}$ and $(\bgamma^{(j)},\tau_j)$
    is the well known connection between precision matrix
    and linear regression for multivariate normal random vectors
    (see, e.g.,
    \cite{meinshausen2006high,sun2013sparse}).

    The estimators $\hbgamma^{(j)}$ in \eqref{lasso-gamma-j}
    and \eqref{scaled-lasso-gamma-j} both satisfy inequalities
    \begin{align}
        \label{eq:KKT-hbgamma-j}
        \|\bX_{-j}^\top(\bX\be_j - \bX_{-j} \hbgamma^{(j)})\|_\infty
        =
        \|\bX_{-j}^\top \hbz_j\|_\infty
        &\le O_\P(1) \tau_j \sqrt{n \log p},
        \\
        \|\hbgamma^{(j)}-\bgamma^{(j)}\|_1
        &\le
        O_\P(1) 
        \|\bSigma^{-1}\|_{op}
        \|\bgamma^{(j)}\|_0
        \tau_j
        \sqrt{\log(p)/n}
        \label{eq:ell_1_rate_hbgamma-j}
    \end{align}
    provided that $\|\bSigma^{-1}\|_{op}\|\bgamma^{(j)}\|_0\log(p)/n\to 0$.
    Inequality \eqref{eq:ell_1_rate_hbgamma-j} is the usual
    $\ell_1$ estimation rate for the Lasso \cite{bickel2009simultaneous}
    or the Scaled Lasso \cite{sun2013sparse,belloni2011square}, and 
    $\|\bSigma^{-1}\|_{op}^{-1}$ represents a high-probability
    lower bound on the restricted eigenvalue
    in the linear model \eqref{eq:linear-model-nodewise} \cite{raskutti2010restricted}.
    Inequality \eqref{eq:KKT-hbgamma-j} follows from the
    KKT conditions of \eqref{lasso-gamma-j} for the Lasso,
    and from the KKT conditions of \eqref{scaled-lasso-gamma-j}
    combined with $\hat \tau_j/\tau_j\smash{\xrightarrow[]{\P}} 1$ which holds thanks to
    properties of the Scaled or square root Lasso \cite{sun2013sparse,belloni2011square}.
    Inequalities \eqref{eq:KKT-hbgamma-j}-\eqref{eq:ell_1_rate_hbgamma-j}
    are the only properties of $\hbgamma^{(j)}$ that we will use
    in the proof of the following result. Other estimators
    $\hbgamma^{(j)}$ could be used, for instance ones based on
    the Dantzig selector, as long as 
    \eqref{eq:KKT-hbgamma-j}-\eqref{eq:ell_1_rate_hbgamma-j}
    are satisfied.

\begin{theorem}
    \label{thm:unknown-Sigma-normal}
    Consider a canonical basis direction
    $\be_j\in\R^p$ for some $j\in[p]$ and let \Cref{assum:main} be fulfilled.
    Additionally assume that the sparsity of $\bSigma^{-1}\be_j$ satisfies
    either
    \begin{equation}
        \label{eq:assum:T-log-p-sqrt-n-unknown-Sigma}
        n^{-1/2} \|\bSigma^{-1}\be_j\|_0\sqrt{[T+\log(p/s)]\log p}\to 0.
    \end{equation}
    or
    \begin{equation}
        \label{eq:assum-unknown-Sigma-second-version}
        \|\bSigma^{-1}\be_j\|_0\log(p)/n\to 0
        \qquad
        \text{ and }
        \qquad
        s \sqrt{\log(p)[T + \log(p/s)]/n}\to 0
        .
    \end{equation}
    Then for any estimator $\hbgamma^{(j)}$
    satisfying \eqref{eq:KKT-hbgamma-j}-\eqref{eq:ell_1_rate_hbgamma-j} 
    and every fixed $\bb\in\R^{T}$ we have
    \begin{equation}
        \frac{
            n \be_j^\top (\hbB - \bB^*) \bb
            + n (\hbz_j^\top \bX\be_j)^{-1}
            \hbz_j^\top(\bY-\bX\hbB)(\bI_{T\times T}-\hbA/n)^{-1}
            \bb
        }
        {
            (\tau_j)^{-1}  ~
        \|(\bY-\bX\hbB)
        (\bI_{T\times T}-\hbA/n)^{-1}
        \bb\|_2
        }
        ~{\xrightarrow[]{d}}~
        \mathcal N(0,1).
        \label{eq:asymptotic-normality-conclusion-unknown-Sigma}
  \end{equation}
  Asymptotic normality \eqref{eq:asymptotic-normality-conclusion-unknown-Sigma}  still holds if $\tau_j$ in the denominator
  is replaced by either 
  $(\hbz_j^\top\bX\be_j/n)^{1/2}$
  or $\hat\tau_j=(\|\hbz_j\|_2/\sqrt n)$.   
\end{theorem}
\Cref{thm:unknown-Sigma-normal}
is proved in \Cref{subsec:proof-unknown-Sigma-normality}.

\section{Confidence ellipsoids for rows of \texorpdfstring{$\bB^*$}{B*}}
\label{sec:chi2}

\subsection{Known \texorpdfstring{$\bSigma$}{Sigma}}
\label{sec:chi2-known-Sigma}
We first construct confidence ellipsoids with the knowledge of $\bSigma$.
\begin{restatable}{theorem}{thmChiSquareQuantiles}
    \label{thm:chi2z0}
    Define the observable positive semi-definite matrix
    $\hbGamma = (\bY-\bX\hbB)^\top(\bY-\bX\hbB)\in\R^{T\times T}$ as well as
    \begin{equation}
        \label{eq:def-bxi}
        \bxi =  (\bY-\bX\hbB)^\top \bz_0 + (n\bI_{T\times T}-\hbA)(\hbB-\bB^*)^\top\ba
        .
    \end{equation}
    Then
    under \Cref{assum:main}, there exists
    a random variable $\chi^2_T$ with chi-square distribution with $T$
degrees of freedom such that
    \begin{align*}
    \sqrt{1-\tfrac T n}
        \Big\|\hbGamma^{-1/2}\bxi\Big\|_2
    -\sqrt{\chi^2_T}
    &\le
    o_\P(1) +
    O_\P\Big(\min\Bigl\{\frac{T}{\sqrt n}, \frac{s^2\log^2(p/s)}{n\sqrt T}\Bigr\}\Big)
    \end{align*}
    as well as
    $$- o_\P(1) - O_\P\Big(\frac{T}{\sqrt n} + \frac{sT + s\log(p/s)}{n}\sqrt T \Big)
        \le
        \sqrt{1-\tfrac T n}
        \Big\|\hbGamma^{-1/2}\bxi\Big\|_2
    -\sqrt{\chi^2_T}.$$
    Consequently,
    \begin{enumerate}
        \item 
    $
    (1-\frac{T}{n})^{\frac12}
    \|\hbGamma{}^{-1/2}\bxi\|_2
    -(\chi^2_T)^{1/2}
    \le o_\P(1)$ holds if additionally 
    $\min\{\frac{T^2}{n},\frac{\log^8p}{n}\}\to0$, and
\item
    $
    (1-\frac{T}{n})^{\frac12}
    \|\hbGamma{}^{-1/2}\bxi\|_2
    -(\chi^2_T)^{1/2}
    \ge o_\P(1)
    $
    holds if additionally $\frac{T^2}{n} + \frac{sT + s\log(p/s)}{n}\sqrt T  \to 0$.
    \end{enumerate}
\end{restatable}
\Cref{thm:chi2z0}
is proved in \Cref{section:proofs-chi2-quantiles}.
The following proposition with $W_n = (1-\frac Tn)^{1/2} \|\hbGamma{}^{-1/2}\bxi\|_2$
relates the $(1-\alpha)$-quantile of $\|\hbGamma{}^{-1/2}\bxi\|_2$
to that of  $(\chi^2_T)^{1/2}$ when either (i) or (ii) above holds. 

\begin{restatable}{proposition}{chisquareQuantileProp}
    \label{prop:4.2}
    Let $(W_{n})_{n\ge 1}$ be a sequence of random random variables
    and $\chi^2_T$ a sequence of random variables
    with chi-square distribution with $T$ degrees-of-freedom,
    where $T=T_n$ is function of $n$ (in particular, $T\to +\infty$
    as $n\to+\infty$ is allowed).
    If $\alpha\in(0,1)$ is a fixed constant not depending on $n,T$
    and $q_{T,\alpha}> 0$ is the quantile defined by
    $\P((\chi_T^2)^{1/2} \le q_{T,\alpha}) = 1-\alpha$ then
    \begin{enumerate}
        \item 
    $W_n
    -(\chi^2_T)^{1/2}
    \le o_\P(1)$ implies that
    $\P(
        W_n
        \le q_{T,\alpha}
    ) \ge 1-\alpha - o(1)$ and
\item
    $
    W_n
    -(\chi^2_T)^{1/2}
    \ge -o_\P(1)
    $ implies that
    $\P(
    W_n
        \le q_{T,\alpha}
    ) \le 1-\alpha + o(1).$
    \end{enumerate}
\end{restatable}

    \Cref{prop:4.2} is proved in
    \Cref{section:proofs-chi2-quantiles}.
If $T\to+\infty$, the order of $q_{T,\alpha}$ is given
by 
\begin{equation}
    \label{eq:order-q_T-alpha}
q_{T,\alpha}
- \sqrt T \to z_{\alpha}/{\sqrt 2}
\end{equation}
where $z_\alpha$ is the standard normal quantile defined
by $\int_{-\infty}^{z_\alpha}(\sqrt{2\pi})^{-1} e^{-u^2/2}du=1-\alpha$.
A short proof of \eqref{eq:order-q_T-alpha}
is given around \eqref{eq:approximation-q_T-alpha};
see \cite{mitra2016benefit} for related discussions.
However, using $q_{T,\alpha}$ itself to construct confidence sets
should be preferred in practice to avoid the approximation error
in \eqref{eq:order-q_T-alpha}.

Combining the above two results provides confidence ellipsoids for the rows of $\bB^*$,
or more generally for the unknown vector $(\bB^*){}^\top \ba\in\R^T$
for a fixed direction $\ba\in\R^p$ of interest.
Let $\hat{\mathcal E}_\alpha$ be the subset of $\R^T$ defined by
$$
\hat{\mathcal E}_{\alpha} \coloneqq \Big\{
    \btheta\in\R^T:
    (1-\tfrac Tn)^{1/2}
    \|\hbGamma{}^{-1/2}
    \big[
    (\bY-\bX\hbB)^\top\bz_0
    + (n\bI_{T\times T}-\hbA) (\hbB^\top\ba- \btheta)
    \big]
    \|_2 
    \le q_{T,\alpha}
\Big\}.
$$
Since 
$\hat{\mathcal E}_{\alpha} = \{
    \btheta\in\R^T:
    (\btheta - \bu)^\top \bC (\btheta - \bu) \leq 1
\}$
where 
$  \bC = (q_{T,\alpha})^{-2}(1-\tfrac Tn)
        (n\bI_{T\times T}-\hbA) \hbGamma{}^{-1} (n\bI_{T\times T}-\hbA)
     \text{ and }
    \bu = \hbB{}^\top\ba + (n\bI_{T\times T}-\hbA)^{-1} (\bY-\bX\hbB)^\top\bz_0,$
this set is an ellipsoidal region with center $\bu$.
If 
\begin{equation}
    \label{eq:extra-condition}
    \min\Bigl\{\frac{T^2}{n},\frac{\log^8 p}{n}\Bigr\}\to 0
\end{equation}
additionally to \Cref{assum:main} as required
in case (i) of \Cref{thm:chi2z0},
then $\P[(\bB^*)^\top\ba \in \hat{\mathcal E}_{\alpha} ] \ge 1-\alpha - o(1)$.
If additionally 
\begin{equation}
    \label{eq:lower-bound-extra-condition}
T^2/n + \sqrt T\bigl(sT + s\log(p/s) \bigr) / n\to 0
\end{equation}
as required in case (ii) for the lower bound,
then $\P[(\bB^*)^\top\ba \in \hat{\mathcal E}_{\alpha} ] \to 1-\alpha$
and the above confidence ellipsoid provides the exact nominal coverage
(i.e., it is provably non-conservative).
Note that the upper bound (i) is more important than the lower bound (ii)
since the upper bound (i) guarantees that the type I error 
in the hypothesis test \eqref{H_1-H_0}
is at most
$\alpha$, i.e., $\P[(\bB^*)^\top\ba \in \hat{\mathcal E}_{\alpha} ] \ge 1-\alpha - o(1)$.
It is thus fortuitous that only the weak additional condition
\eqref{eq:extra-condition} is required for the upper bound (i)
to guarantee the desired type I error, while the more stringent
condition \eqref{eq:lower-bound-extra-condition} is only required
to prove non-conservativeness.

The additional assumption \eqref{eq:extra-condition}
is satisfied for a large class of growths of $(T,n,p)$.
For instance it holds under polynomial growth $p\asymp n^\gamma$ 
or exponential growth of the form
$p\lesssim \exp(n^{1/8-\gamma'})$ for constants $\gamma,\gamma'>0$,
as $\frac{\log^8 p}{n}\to 0$ is then satisfied.
Although we believe that the mild condition
\eqref{eq:extra-condition} is an artefact of the proof,
it is unclear at this point how to relax \eqref{eq:extra-condition}
unless a different ellipsoid is considered.
In \Cref{eq:drop-extra-conditions}, we will construct a different ellipsoid
that does not require the extra conditions
\eqref{eq:extra-condition} or \eqref{eq:lower-bound-extra-condition}
but that has worse performance in simulations.

The radius of $\hat{\mathcal E}_\alpha$
i.e., the half-length of its largest axis is given by
\begin{equation}
    \phi_{\min}(\bC)^{-1/2} = (1-T/n)^{-1/2} q_{T,\alpha} \|\hbGamma{}^{1/2}(n\bI_{T\times T} - \hbA)^{-1}\|_{op}.
 \label{eq:radius-ellipsoid}
\end{equation}
Since $\|\bI_{T\times T} - (\bI_{T\times T} - \hbA/n)^{-1} \|_{op} = o_\P(1)$
by \Cref{propMatrixA} and \Cref{lemma:sparsity} on the one hand,
and all eigenvalues of $\hbGamma$ are of order $\sigma^2 n(1+o_\P(1))$
by the arguments in the proof of \Cref{lemma:bxi-bz} on the other hand,
the radius \eqref{eq:radius-ellipsoid} is $q_{T,\alpha} \sigma n^{-1/2} (1+o_\P(1))$ which is of order $\sigma \sqrt{T/n}$ by \eqref{eq:order-q_T-alpha}.

The random vector \eqref{eq:def-bxi} involves multiplication
by $(n\bI_{T\times T} - \hbA)$ which differs from the pivotal quantity
in the asymptotic normality result \eqref{eq:asymptotic-normality-bb}.
However, \Cref{thm:chi2z0} still holds with $\bxi$ in \eqref{eq:def-bxi}
replaced by
\begin{equation}
    \label{eq:def-check-bxi}
    \check\bxi =
    (\bI_{T\times T} - \hbA/n)^{-1}(\bY-\bX\hbB)^\top \bz_0 + n(\hbB-\bB^*)^\top\ba
    .
\end{equation}
Indeed, with
$$
\big|\|\hbGamma^{-1/2}\check\bxi\|_2
-
\|\hbGamma^{-1/2}\bxi\|_2
\big|
\le
\|\hbGamma^{-1/2}(\check\bxi - \bxi)\|_2
=
\|\hbGamma^{-1/2}\bigl((\bI_{T\times T} - \hbA/n)^{-1} - \bI_{T\times T}\bigr) \bxi\|_2.
$$
Since the eigenvalues of $\hbGamma$ are all of order $\sigma^2n(1+o_\P(1))$
,
since $\|(\bI_{T\times T} - \hbA/n)^{-1} - \bI_{T\times T}\|_{op} \le (1+o_\P(1)) |\hat S|/n$
by \Cref{propMatrixA}
and since $\|\bxi\|_2 = O_\P(\sqrt{\sigma^2 n T})$ by \Cref{thm:chi2initial},
the previous display is $O_\P(\sqrt T s/n)$ and converges to 0 in probability
by \Cref{assum:main}. Under \Cref{assum:main}, \Cref{thm:chi2z0}
thus holds for $\bxi$ in \eqref{eq:def-bxi} if and only if it holds
for $\check\bxi$. Furthermore the corresponding ellipsoid,
$$
\check{\mathcal E}_{\alpha} = \Big\{
    \btheta\in\R^T:
    (1-\tfrac Tn)^{1/2}
    \|\hbGamma{}^{-1/2}
    \big[
    (\bI_{T\times T}-\hbA/n)^{-1}(\bY-\bX\hbB)^\top\bz_0
    + n (\hbB^\top\ba- \btheta)
    \big]
    \|_2 
    \le q_{T,\alpha}
\Big\}
$$
enjoys the same properties as $\hat{\mathcal E}_\alpha$:
Type I error guarantees
$\P[(\bB^*)^\top\ba \in \check{\mathcal E}_{\alpha} ] \ge 1-\alpha - o(1)$
under \eqref{eq:extra-condition},
and non-conservativeness
$\P[(\bB^*)^\top\ba \in \check{\mathcal E}_{\alpha} ]\to 1 - \alpha$
under \eqref{eq:lower-bound-extra-condition}.

\subsection{Unknown \texorpdfstring{$\bSigma$}{Sigma}}
\label{sec:chi2-unknown-Sigma}
A similar result is available if $\bSigma$ is unknown.
Consider the notation \eqref{eq:def-hbz-j} from
\Cref{sec:unknown-Sigma}.

\begin{theorem}
    \label{thm:unknown-Sigma-chi2}
    Consider a canonical basis direction
    $\be_j\in\R^p$ for some $j\in[p]$ and let \Cref{assum:main} be fulfilled.
    Additionally assume that either
    \eqref{eq:assum:T-log-p-sqrt-n-unknown-Sigma}
    or \eqref{eq:assum-unknown-Sigma-second-version} holds.
    Then for any estimator $\hbgamma^{(j)}$
    satisfying \eqref{eq:KKT-hbgamma-j}-\eqref{eq:ell_1_rate_hbgamma-j},
    \begin{multline}
        \label{eq:conclusion-chi2}
        \frac{\sqrt{n-T}}{\|\hbz_j\|_2}
    \Bigl\|
    \hbGamma^{-1/2}
        \Big(
            (\bY-\bX\hbB)^\top \hbz_j  
            + \frac{
                (n\bI_{T\times T}-\hbA)(\hbB-\bB^*)^\top\be_j
            }{
                n (\hbz_j^\top \bX\be_j)^{-1} 
            }
        \Big)
    \Bigr\|_2
    \\\le
    \sqrt{\chi^2_T} 
    + o_\P(1)
    + O_\P(\min\{\tfrac{T}{\sqrt n}, ~ \tfrac{\log^2 p}{n^{1/4}} \})
    \end{multline}
    where $\chi^2_T$ is a random variable with chi-square distribution
    with $T$ degrees-of-freedom.
\end{theorem}
\Cref{thm:unknown-Sigma-chi2}
is proved in \Cref{subsec:proof-unknown-Sigma-chi2}.
The corresponding confidence ellipsoid for the $j$-th row
$(\bB^*)^\top\be_j$ of $\bB^*$ is
$$
\hat{\mathcal E}_\alpha^j =
\Bigl\{\btheta_j\in\R^T:
        \frac{{\sqrt{n-T}}}{\|\hbz_j\|_2}
    \Bigl\|
    \hbGamma^{-1/2}
        \Big[
            (\bY-\bX\hbB)^\top \hbz_j  
            + \frac{
                (n\bI_{T\times T}-\hbA)(\hbB^\top \be_j-\btheta_j)
            }{
                n (\hbz_j^\top \bX\be_j)^{-1} 
            }
        \Big]
    \Bigr\|_2
    \le q_{T,\alpha}
\Bigr\}.
$$
If either one of the condition
\eqref{eq:assum:T-log-p-sqrt-n-unknown-Sigma}
or \eqref{eq:assum-unknown-Sigma-second-version} holds
on the growth of the sparsity of $\bSigma^{-1}\be_j$,
this confidence ellipsoid does not require the knowledge of $\bSigma$
and has the same guarantees as those of the previous section. 

\subsection{Relaxing the additional assumptions
\eqref{eq:extra-condition} and \eqref{eq:lower-bound-extra-condition}}
\label{eq:drop-extra-conditions}

Instead of normalizing using $\hbGamma{}^{-1/2}$ as in the previous
sections, a simple estimate of $\sigma^2$ lets us relax
the conditions
\eqref{eq:extra-condition} and \eqref{eq:lower-bound-extra-condition}
that are required in the previous section to ensure
$\|\hbGamma{}^{-1/2}\bxi\|_2 = (\chi^2_T)^{1/2} + o_\P(1)$.

\begin{theorem}
    Let $\bxi,\check\bxi$ be defined in \eqref{eq:def-bxi} and \eqref{eq:def-check-bxi} respectively, 
    and let $\hat\sigma^2 = \|\bY-\bX\hbB\|_F^2/(nT)$.
    Then under \Cref{assum:main},
    there exists a random variable $\chi^2_T$ with chi-square
    distribution with $T$ degrees of freedom such that
    \begin{equation}
        \label{eq:asymptotic-chi2-hat-sigma-bxi-check-bxi}
        (\hat\sigma^2 n )^{-1/2}\| \bxi \|_2 = (\chi^2_T)^{1/2} + o_\P(1),
        \qquad
        (\hat\sigma^2 n )^{-1/2}\| \check\bxi \|_2 = (\chi^2_T)^{1/2} + o_\P(1)
        .
    \end{equation}
\end{theorem}
\begin{theorem}
    \label{thm:unknown-Sigma-chi2-hatsigma}
    Consider a canonical basis direction
    $\be_j\in\R^p$ for some $j\in[p]$ and let \Cref{assum:main} be fulfilled.
    Additionally assume that either
    \eqref{eq:assum:T-log-p-sqrt-n-unknown-Sigma}
    or \eqref{eq:assum-unknown-Sigma-second-version} holds.
    Then for any estimator $\hbgamma^{(j)}$
    satisfying \eqref{eq:KKT-hbgamma-j}-\eqref{eq:ell_1_rate_hbgamma-j},
    \begin{equation}
        \frac{1}{\|\hbz_j\|_2\hat\sigma}
    \Bigl\|
            (\bY-\bX\hbB)^\top \hbz_j  
            + \frac{
                (n\bI_{T\times T}-\hbA)(\hbB-\bB^*)^\top\be_j
            }{
                n (\hbz_j^\top \bX\be_j)^{-1} 
            }
    \Bigr\|_2
    =
    \sqrt{\chi^2_T} 
    + o_\P(1)
        \label{eq:asymptotic-chi2-hat-sigma-unknown-Sigma}
    \end{equation}
    where $\chi^2_T$ is a random variable with chi-square distribution
    with $T$ degrees-of-freedom.
\end{theorem}

The above asymptotic chi-square results hold under the same 
assumptions as \Cref{thm-normal} and \Cref{thm:unknown-Sigma-normal}.
The reason for the success of these estimates is that 
$\hat \sigma$ estimates $\sigma$ at a rate faster than $T^{-1/2}$:
we have $|\hat \sigma/\sigma - 1| = o_\P(T^{-1/2})$ 
by \Cref{thm:chi2initial}.
However, simulations in \Cref{sec:simu} reveal that the
asymptotic $(\chi^2_T)^{1/2}$ estimates
of the previous subsections involving the matrix $\hbGamma{}^{-1/2}$
are more robust to larger sparsity levels, although
\Cref{assum:main} is oblivious to this phenomenon.

The corresponding $1-\alpha$ confidence ellipsoid for $(\bB^*)^\top\ba$
based on \eqref{eq:asymptotic-chi2-hat-sigma-bxi-check-bxi} and
$\check\bxi$ is
\begin{equation}
    \label{eq:check-hat-sigma-ellipsoid}
    \check{\mathcal E}_{\hat\sigma, \alpha}
    =
    \Bigl\{
        \btheta \in \R^T:
        \tfrac{1}{\hat\sigma\sqrt n}
        \big\|
        \bigl(\bI_{T\times T}-\tfrac{\hbA}{n}\bigr)^{-1}
        (\bY -\bX\hbB)^\top\bz_0
        + n (\hbB^\top\ba - \btheta)
        \big\|_2\leq q_{T,\alpha}
    \Bigr\}
\end{equation}
and satisfies $\P[(\bB^*)^\top\ba \in \check{\mathcal E}_{\hat\sigma,\alpha}]
\to 1-\alpha$ under \Cref{assum:main}.
Similar confidence ellipsoids based on 
\eqref{eq:asymptotic-chi2-hat-sigma-unknown-Sigma}
can be readily constructed.

\subsection{Hypothesis testing}

We now turn to type II error for the testing problem
\begin{equation}
    H_0: (\bB^{*})^\top\ba = \mathbf{0}_{T\times 1}
\qquad
\text{ against }
\qquad
H_1: \|(\bB^{*})^\top\ba\|_2 \ge \rho_n
\label{H_1-H_0-ba-section3}
\end{equation}
where $\rho_n>0$ is a separation radius.
The hypothesis test
\eqref{H_1-H_0-ba-section3}
at level $1-\alpha$ is naturally achieved by rejecting
$H_0$ if and only if $\mathbf{0}_{T\times 1}\not\in \check{\mathcal E}_{\hat\sigma ,  \alpha}$ for the ellipsoid in \eqref{eq:check-hat-sigma-ellipsoid}.
Similar rejection procedures can be obtained
with $\check{\mathcal E}_\alpha$
or $\hat{\mathcal E}_\alpha^j$ for the confidence ellipsoids
defined in \Cref{sec:chi2-known-Sigma,sec:chi2-unknown-Sigma}.

We can also determine the separation radius $\rho_n$ required so that
this testing procedure has nontrivial power (type II error).
Focusing here on $\check{\mathcal E}_{\hat\sigma, \alpha}$ in
\eqref{eq:check-hat-sigma-ellipsoid},
rejection happens if and only if the following quantity is positive 
\begin{align*}
    &(\hat\sigma^2n)^{-1}\|
        (\bI_{T\times T}-\hbA/n)^{-1}
        (\bY-\bX\hbB)^\top\bz_0
        +
        n
        \hbB^\top\ba
    \|_2^2
    -
    q_{T,\alpha}^2 
               \\&=
               W_n^2
    -
    q_{T,\alpha}^2 
    \\&\quad 
    +
    (\hat\sigma^2n)^{-1}
    \|
    n
    (\bB^*)^\top\ba
    \|_2^2
    \\&\quad +
    2
    (\hat\sigma^2n)^{-1/2}
    n\ba^\top(\bB^*)^\top
    \check \bxi.
\end{align*}
where
$W_n^2 = (\hat\sigma^2n)^{-1}
    \|
    \check \bxi
    \|_2^2$
and $W_n = (\chi^2_T)^{1/2} + o_\P(1)$ by
\eqref{eq:asymptotic-chi2-hat-sigma-bxi-check-bxi}.
By \Cref{thm-normal} applied to $\bb=(\bB^*)^\top\ba \|(\bB^*)^\top\ba\|_2^{-1}$,
the last line
is of the form $2 \hat \sigma^{-2} \|(\bB^*)^\top\ba\|_2
 \mathcal N(0,\sigma^2)$ so that it is of order
 $\|(\bB^*)^\top\ba\|_2 O_\P(1)$.
The second line is positive, of order
$\sigma^{-2}  n (1+ o_\P(1)) \|(\bB^*)^\top\ba\|_2^2$;
this is the quantity that should dominate in order
to ensure that the above display is positive.
Since the first line 
$W_n^2 - q_{T,\alpha}= (W_n -  q_{T,\alpha})(W_n + q_{T,\alpha})$
is positive with probability at least $\alpha-o(1)$
by \Cref{prop:4.2}, we obtain
that if $\|(\bB^*)^\top \ba\|_2 \ge \rho_n$ for
$\rho_n/(\sigma n^{-1/2}) \to +\infty$, then
the type II error is at most $1-\alpha + o(1)$.
Although this type II error is typically a constant
close to 1 (e.g. if $\alpha=0.05)$,
this shows that the above test has at most constant type II error
as long as the separation radius satisfies $\rho_n \ggg \sigma n^{-1/2}$.
We can also find conditions on $\rho_n$ that ensures that the type II
error is smaller than any constant. The first line above
is of order $(W_n - q_{T,\alpha}) O_\P(\sqrt T) = O_\P(\sqrt T)$
since $W_n = O_\P(\sqrt T) + o_\P(1)$ and $q_{T,\alpha} = \sqrt T + O(1)$
by \Cref{prop:4.2} and
\eqref{eq:order-q_T-alpha}.
Thus
$\rho_n \ggg T^{1/4} \sigma n^{-1/2}$ is sufficient
in order for
$
    (\hat\sigma^2n)^{-1}
    \|
    n
    (\bB^*)^\top\ba
    \|_2^2
$ to dominate both the first and third lines with probability
approaching one.
In summary,
$\rho_n \ggg \sigma n^{-1/2}$ is sufficient to achieve a constant 
type II error, while
$\rho_n \ggg T^{1/4} \sigma n^{-1/2}$ is sufficient to grant
a vanishing type II error.

In single task models, coefficients $\bB_{jt}^*$
of order 
$o( \sigma n^{-1/2} )$ cannot be detected, cf.
the discussion after \eqref{H_1-H_0}.
Here on the other hand in the multi-task setting with $T\to+\infty$,
detection of non-zero vector
$(\bB^*)^\top\be_j$ is possible with constant power
even if the individual coefficients in 
$(\bB^*)^\top\be_j$ are 
$u_n \sigma (Tn)^{-1/2}$ for any slowly increasing $u_n$ with $u_n\to+\infty$.
If $u_n= o(\sqrt T)$, the coefficients 
$(\bB_{jt}^*)_{t=1,...,T}$ are individually impossible
to detect, while detection of the row vector $(\bB^*)^\top\be_j$ is possible
with constant type I and type II errors.
Similarly, if the individual coefficients $(\bB_{jt}^*)_{t=1,...,T}$
are of order
$u_n \sigma T^{-1/4}n^{-1/2}$ for any slowly increasing $u_n$ with $u_n\to+\infty$, the above testing procedure for the row vector 
$(\bB)^\top\be_j$ has vanishing type II error.

%
%

\section{Computing the interaction matrix efficiently}
\label{sec:computing}
\Cref{def-A} which defines $\hbA$ is convenient for theoretical purposes,
as the pseudoinverse suppresses invertibility issues
and the form
\eqref{def-A} naturally arises in the proofs, see for instance
\Cref{lemma:differential,lemma:divergence,lemma:RemainderII}.
However, \eqref{def-A} is not computationally tractable 
as it involves computing a pseudoinverse of size $pT \times pT$.
The goal of this section is to provide a computationally
tractable representation for $\hbA$; in particular we will see
that one only needs to compute inverses of matrices of size
$|\hat S|\times |\hat S|$.
A first step when implementing is to remove all covariates $j\in\{1,...,p\}$
such that $\hbB{}^\top\be_j = \mathbf 0$, as 
dropping those indices and the corresponding columns of $\bX$
does not change the value of $\hbA$ in \eqref{def-A}.
For the purpose of this section and only in this section,
we assume without loss of generality
that $\hat S=[p]$ and that all variables $j\in[p]$
are such that $\hbB{}^\top\be_j\ne {\mathbf 0}$.
However, we will keep the notation $\bX_{\hat S}$
and use summation sign $\sum_{j\in\hat S}$
to emphasize that the indices $j\notin \hat S$
and corresponding columns of $\bX$ have been dropped.

Before stating a formal proposition
with a computationally friendly representation of the matrix
$\hbA$, we explain the crux of the argument,
which relies on the 
Sherman-Morrison-Woodbury inversion formula.
Recall that
$\tbX{}^\top\tbX = (\bI_{T\times T}\otimes \bX_{\hat S}^\top \bX_{\hat S})$ and that
for every $j\in\hat S$
\begin{equation*}
    \tag{\ref{Hj}}
\bH^{(j)}
\coloneqq
\lambda
\|\hbB{}^\top \be_j\|_2^{-1}\bigl(\bI_{T\times T} - \hbB{}^\top\be_j \be_j^\top\hbB ~  \|\hbB{}^\top\be_j\|_2^{-2} \bigr) \quad \in\R^{T\times T},
\end{equation*}
as well as $\tbH = \sum_{j\in \hat S} \bH^{(j)} \otimes (\be_j\be_j^\top)$.
By splitting the part of $\bH^{(j)}$ proportional to the identity
and the rank one part, we find
\begin{align*}
    \tbX{}^\top\tbX + nT \tbH 
    &=
        (\bI_{T\times T}\otimes \bX_{\hat S}^\top \bX_{\hat S})
        +
        nT\lambda
        \sum_{j\in \hat S}
        \frac{
            \bI_{T\times T}\otimes \be_j\be_j^\top
        }{
            \|\hbB{}^\top \be_j\|_2
        }
        -
        \Big[
        \frac{
            \hbB{}^\top\be_j \be_j^\top \hbB
        }{
        \|\hbB{}^\top\be_j\|_2^{3}
        }
         \otimes
         (\be_j\be_j^\top)
    \Big]
    \\&=
    \left(\bI_{T\times T}\otimes 
        \big(
        \bX_{\hat S}^\top \bX_{\hat S}
        +
        \diag(\bv)
        \big)
    \right)
        - nT\lambda
        \sum_{j\in \hat S} 
        \Big[
        \frac{
            (\hbB{}^\top\be_j \be_j^\top \hbB) \otimes
            (\be_j\be_j^\top)
        }{
            \|\hbB{}^\top\be_j\|_2^{3}
        }
        \Big]
\end{align*}
where $\bv \in\R^{|\hat S|}$ is the vector with $v_j = nT \lambda \|\hbB{}^\top\be_j\|_2^{-1}$
and $\diag(\bv)$ is the square diagonal matrix with $\bv$ as its diagonal.
By the mixed product property
\eqref{mixed-product-property}
we have
$$
        (\hbB{}^\top\be_j \be_j^\top \hbB) \otimes
            (\be_j\be_j^\top)
=
(\hbB{}^\top\be_j \otimes \be_j)
(\hbB{}^\top\be_j \otimes \be_j)^\top
$$
so that, with $\bb^{(j)} = (
        nT\lambda
        \|\hbB{}^\top\be_j\|_2^{-3}
        )^{1/2}
        ~ \hbB{}^\top\be_j \in\R^T$
we obtain
\begin{align*}
\tbX^\top\tbX + nT \tbH
&=
    \left(\bI_{T\times T}\otimes 
        \big(
        \bX_{\hat S}^\top \bX_{\hat S}
        +
        \diag(\bv)
        \big)
    \right)
    - \sum_{j\in\hat S}
    (\bb^{(j)} \otimes \be_j)
    (\bb^{(j)} \otimes \be_j)^\top
    \\&=
    \left(\bI_{T\times T}\otimes 
        \big(
        \bX_{\hat S}^\top \bX_{\hat S}
        +
        \diag(\bv)
        \big)
    \right)
    - \bU \bU^\top
    \\&=
    \bM - \bU\bU^\top,
\end{align*}
where $\bU\in\R^{(|\hat S|T) \times |\hat S|}$ has columns
$(\bb^{(j)}\otimes \be_j)_{j\in \hat S}$ and
$\bM = 
    \bI_{T\times T}\otimes 
        \big(
        \bX_{\hat S}^\top \bX_{\hat S}
        +
        \diag(\bv)
        \big)
    . 
$
If $\bM$
is invertible and its inverse can be computed efficiently,
the inverse of the above display is given by the
Sherman-Morrison-Woodbury formula \cite{horn_johnson_2012}:
if the matrix  
    $
- \bI_{|\hat S|\times |\hat S|}
+ \bU^\top \bM^{-1} \bU$  is invertible then $\bM-\bU\bU^\top$
is also invertible and
$$
(\bM - \bU\bU^\top)^{-1}
= \bM^{-1}
-
\bM^{-1} \bU\left(
- \bI_{|\hat S|\times |\hat S|}
+ \bU^\top \bM^{-1} \bU
\right)^{-1}
\bU^\top \bM^{-1}.
$$
Since $\bv$ has positive entries, 
$\bX_{\hat S}^\top \bX_{\hat S} +\diag(\bv)$ is always invertible 
and so is $\bM$, with
\begin{equation}
    \label{bM-inverse}
    \bM^{-1} = 
\bI_{T\times T}\otimes 
        \big(
        \bX_{\hat S}^\top \bX_{\hat S}
        +
        \diag(\bv)
    \big)^{-1}.
\end{equation}
Hence we only need to perform two inversions of matrices of size $|\hat
S|\times |\hat S|$: the inversion of $\bX_{\hat S}^\top \bX_{\hat S} + \diag(\bv)$ 
and of
$-\bI_{|\hat S|\times |\hat S|} + \bU^\top \bM^{-1} \bU$.

\begin{proposition}
\label{prop:expression-A}
With the above notation for
    $\bv\in\R^{|\hat S|}$ and
$\bb^{(j)}\in\R^T$ for each $j\in\hat S$,
if the matrix $\bP$ defined entrywise by
\begin{align*}
    \bP =(P_{jk})_{(j,k)\in \hat S\times \hat S},
    \quad
    P_{jk}
    =
- I\{j=k\}
+
(\bb^{(j)}{}^\top\bb^{(k)})  ~ (\be_j^\top
    \big(
        \bX_{\hat S}^\top \bX_{\hat S}
        +
        \diag(\bv)
    \big)^{-1}
    \be_k
)       
\end{align*}
is invertible then
$$
    \hbA=
    \trace\Bigl[ {\bX_{\hat S}^\top} \bX_{\hat S}(\bX_{\hat S}^\top\bX_{\hat S} +\diag(\bv))^{-1}  \Bigr]
    ~
    \bI_{T\times T}
-
\Bigl[
\sum_{j\in\hat S}
\bb^{(j)}
\sum_{k\in\hat S}
~ (\be_j^\top
        \bQ
    \be_k
) (\be_j^\top
        \bP^{-1}
    \be_k
)
\bb^{(k)}{}^\top
\Bigr]
$$
where
$
    \bQ = 
    \big(
        \bX_{\hat S}^\top \bX_{\hat S}
        +
        \diag(\bv)
    \big)^{-1}
    \bX_{\hat S}^\top\bX_{\hat S}
    \big(
        \bX_{\hat S}^\top \bX_{\hat S}
        +
        \diag(\bv)
    \big)^{-1}.
$
\end{proposition}

\begin{proof}
    By definition of $\hbA$ and
    using the above Sherman-Morrison-Woodbury identity
\begin{align*}
    \hbA_{t,t'} &=\trace[
    (\be_t^\top \otimes \bX_{\hat S})
    (\tbX^\top\tbX + nT \tbH)^\dagger
    (\be_{t'} \otimes \bX_{\hat S}^\top)
] 
              \\&=
  \trace[
  (\be_t^\top \otimes \bX_{\hat S}) \bM^{-1}
       (\be_{t'} \otimes \bX_{\hat S}^\top)
  ] 
              \\&\qquad -
  \trace\bigl[
  (\be_t^\top \otimes \bX_{\hat S})
\bM^{-1} \bU\bigl(
- \bI_{|\hat S|\times |\hat S|}
+ \bU^\top \bM^{-1} \bU
\bigr)^{-1}
\bU^\top \bM^{-1}
       (\be_{t'} \otimes \bX_{\hat S}^\top)
  \bigr] 
  \\&=
  \trace[(\be_t^\top\be_{t'})\otimes (\bX_{\hat S}(\bX_{\hat S}^\top\bX_{\hat S} +\diag(\bv))^{-1} \bX_{\hat S}^\top ) ]
  \\&\qquad -
  \trace\bigl[
  \bU^\top \bM^{-1}((\be_{t'}\be_t^\top) \otimes \bX_{\hat S}^\top\bX_{\hat S}) \bM^{-1}\bU
\bigl(
- \bI_{|\hat S|\times |\hat S|}
+ \bU^\top \bM^{-1} \bU
\bigr)^{-1}
  \bigr].
\end{align*}
By \eqref{trace-property-kronecker}, the first term equals
$I\{t=t'\} \trace[ \bX_{\hat S}(\bX_{\hat S}^\top\bX_{\hat S} +\diag(\bv))^{-1} \bX_{\hat S}^\top ]$ which gives the first term
in the proposition, proportional to $\bI_{T\times T}$.
Using again the structure of $\bM^{-1}$ in \eqref{bM-inverse},
the second summand in the previous display is equal to
\begin{equation}
    \label{second-term-A} 
- \trace\bigl[\bU^\top\bigl( (\be_{t'}\be_t^\top) \otimes
    \bQ
\bigr)
    \bU
    \bP^{-1}
\bigr]
\end{equation}
where $\bP$ and $\bQ$ are given in the proposition,
    after noting that the definition of $\bP$ is equivalent to
$\bP = - \bI_{|\hat S|\times |\hat S|}
+ \bU^\top \bM^{-1} \bU$.
Since $\bU$ has columns $\bb^{(j)} \otimes \be_j$,
the entry $(j,k)\in\hat S\times \hat S$ of the matrix
$\bU^\top\bigl( (\be_{t'}\be_t^\top) \otimes
    \bQ
\bigr)
    \bU$
 is equal to
$$
(\bb^{(j)}{}^\top\be_{t'} \be_{t}^\top \bb^{(k)})  ~ (\be_j^\top
        \bQ
    \be_k
) =
(\be_{t'}^\top \bb^{(j)})(\be_{t}^\top \bb^{(k)})
~ (\be_j^\top
        \bQ
    \be_k
) 
.
$$
Since $\trace[\bA\bB] = \sum_{j,k} A_{jk}B_{jk}$ for two symmetric
matrices of the same size, we obtain
\begin{align*}
\eqref{second-term-A}
&=
-
\sum_{j\in\hat S}\sum_{k\in\hat S}
(\be_{t'}^\top \bb^{(j)})(\be_{t}^\top \bb^{(k)})
~ (\be_j^\top
        \bQ
    \be_k
) (\be_j^\top
        \bP^{-1}
    \be_k
)
\\ & =
\be_{t'}^\top
\Bigl[
-
\sum_{j\in\hat S}
\bb^{(j)}
\sum_{k\in\hat S}
~ (\be_j^\top
        \bQ
    \be_k
) (\be_j^\top
        \bP^{-1}
    \be_k
)
\bb^{(k)}{}^\top
\Bigr]\be_t.
\end{align*}
On the last line, the matrix in bracket is the second matrix
in the expression of $\hbA$.
\end{proof}

We now turn to implementation details. We recommend
an approach that makes use of optimized
vectorized code as often as possible
to compute the quantities in \Cref{prop:expression-A},
and if available to use a library with Einstein summation routine
as this allows the code to mimick the mathematical notation
in  \Cref{prop:expression-A}.
For concreteness, the following code lets us efficiently compute $\hbA$
with the Python library Numpy \cite{numpy},
and the Einstein summation function
\verb|numpy.einsum| which comes in handy.
Assume that
$\hbB$ has been computed, the rows in $[p]\setminus \hat S$ removed
and the result stored in an array
\verb|B_S| of size $|\hat S|\times T$,
that $\bX$ with the columns in $[p]\setminus\hat S$ removed
is stored in an array \verb|X_S| of size $n\times |\hat S|$,
and that the scalar $nT\lambda$ is stored in variable \verb|nTlambda|.
Then the vector $\bv$
and matrix with columns $(\bb^{(j)})_{j\in\hat S}$ in variable $\verb|b|$
can be computed as follows:
\begin{minted}[fontsize=\footnotesize]{python}
import numpy as np
norms = np.linalg.norm(B_S, axis=1)                          # shape (|Ŝ|, )
v = nTlambda * norms**(-1)                                   # shape (|Ŝ|, )
b = nTlambda**0.5 * np.einsum("j,jt->jt", norms**(-3/2), B_S)# shape (|Ŝ|, T)
\end{minted}
Finally, matrices $(\bX_{\hat S}^\top\bX_{\hat S} + \diag(\bv))^{-1}$ and $\bQ$
are computed using built-in symmetric matrix inversion,
while computation of $\bP$ and $\hbA$ again resorts to using \verb|np.einsum|:
\begin{minted}[fontsize=\footnotesize]{python}
gram = X_S.T @ X_S                                           # shape (|Ŝ|,|Ŝ|)
inverse = np.linalg.invh(gram + np.diag(v))                  # shape (|Ŝ|,|Ŝ|)
Q = inverse @ gram @ inverse                                 # shape (|Ŝ|,|Ŝ|)
P = - np.eye(p) + np.einsum("jt,kt,jk -> jk", b, b, inverse) # shape (|Ŝ|,|Ŝ|)
A = np.eye(T) * np.einsum("jk,jk->", gram, inverse) \
    - np.einsum("jt,ku,jk,jk->tu", b, b, Q, np.linalg.invh(P))# shape (T, T)
\end{minted}
In \verb|einsum|, we use indices \verb|t| and \verb|u| to loop over $[T]$,
and indices \verb|j| and \verb|k| to loop over $\hat S$.
All calls to \verb|einsum| can be further optimized
by pre-computing the optimal order in which tensor contractions
should be performed (see \verb|numpy.einsum_path|).

Empirically, we have observed that this implementation
using the Sherman-Morrison-Woodbury identity and the above 
code is several orders of magnitude faster than a naive one 
involving sparse matrices and the full inversion of 
$\tbX{}^\top\tbX + nT \tbH$.

\section{Numerical experiments}
\label{sec:simu}

We run simulations to illustrate the theorems proved in \Cref{sec:normal-thms,sec:chi2}.
The values of the parameters are fixed to $n=2000$, $p=6000$, $T=10$, $\eta_1=\eta_2=0$, $\sigma^2 = 1$.
The tuning parameter is
$\lambda = \max_j \bSigma_{jj}^{1/2} \frac{1}{\sqrt {nT}} \big(1+\sqrt{\frac 2T \log \frac ps} \big)$
(we explain below how $\bSigma$ is constructed).
The directions of interest are $\ba=\be_j\in \mathbb R^p$ and $\bb=\be_1\in \mathbb R^T$.

\subsection*{Quantile-quantile plots of the pivotal quantities}
The goal is to assess how the sparsity 
of $\bB^*$ and $\bSigma^{-1}\be_1$ influence the convergence in 
\Cref{thm-normal,thm:unknown-Sigma-normal,thm:chi2z0,thm:unknown-Sigma-chi2,thm:unknown-Sigma-chi2-hatsigma}.
Denote by $s$ and $s_\Omega$ the respective sparsity parameters that will vary in the experiments.
Given a target tuple $(s,s_\Omega)$ we generate $\bB^*$ with exactly $s$ non-zero rows and
$\bSigma$ with exactly $s_\Omega$ non-zero entries on the first column of $\bSigma^{-1}$,
so that $s_\Omega = \|\bSigma^{-1}\be_1\|_0$.

We explain first how $\bSigma$ is constructed so that it satisfies the constraints in \Cref{assum:main}
as well as the sparsity requirement on $\bSigma^{-1}\be_1$.
Start by sampling $\bM$, a $(p-1)\times (p-1)$ matrix with i.i.d. $\mathcal N(0,1)$ entries. 
Then perform the QR decomposition of $\bM$ to obtain an orthogonal matrix $\bQ$, 
the distribution of which is uniform in the sense of Haar measure on the orthogonal group $O(p-1)$.
Next, consider $\bD$, the diagonal $(p-1)\times (p-1)$ matrix with entries $\{1+ j/(p-2) : j\in \{0,\ldots, p-2\} \}$ 
and set $\bLambda = \bQ \bD \bQ^\top$.
Define the block matrix 
$$\tbLambda = \begin{bmatrix}
    3/2 & \hspace*{-\arraycolsep}\vline\hspace*{-\arraycolsep} & \bv^\top \\
    \hline
    \bv & \hspace*{-\arraycolsep}\vline\hspace*{-\arraycolsep} & \bLambda
\end{bmatrix}$$ 
where $\bv\in \mathbb R^{p-1}$ is a vector with sparsity $\|\bSigma^{-1}\be_1\|_0-1$ and norm $\|\bv\|_2 = 1$.
This ensures boundedness of the spectrum as the smallest eigenvalue of $\tbLambda$ satisfies the lower bound
\begin{align*}
    \lambda_{\min}(\tbLambda)
    &\geq \lambda_{\min} \Big(
        \begin{bmatrix}
        3/2 & -\|\bv\|_2 \\
        -\|\bv\|_2 & \lambda_{\min}(\bLambda)
        \end{bmatrix}
        \Big)
    = \frac{5-\sqrt{17}}{4}
    \gtrsim 0.219,
\end{align*}
where the last equality follows from $\lambda_{\min}(\bLambda)= 1$ and $\|\bv\|_2 = 1$.
Similarly, the largest eigenvalue of $\tbLambda$ can be bounded above by 
$
    \lambda_{\max}(\tbLambda)
    = \frac{7+\sqrt{17}}{4}
    \lesssim 2.8.
$
Finally set $\bSigma = \alpha^{-1} \tbLambda{}^{-1}$ where $\alpha$ is the greatest diagonal entry of $\tbLambda{}^{-1}$
so that $\max\{ \bSigma_{jj}, j=1,...,p\}=1$.
This construction
leads to $\lambda_{\min}(\bSigma)\approx 0.32,\lambda_{\max}(\bSigma)\approx 1.76$ and
$(\bSigma^{-1})_{jj}\approx 1.85$.

The row-sparse matrix $\bB^*$ is constructed as follows.
Initialize $\bB^*$ as a matrix filled with $\lambda$'s and alter it in two different ways:
\begin{enumerate}
    \item 
        \label{unfavorable}
        \emph{Setting with overlapping supports.}
In the first setting, we zero out rows of $\bB^*$ while forcing an overlap of the supports of $\bB^*$ and $\bSigma^{-1}\be_1$
(either $\supp(\bSigma^{-1}\be_1)\subset \supp(\bB^*)$ or the
reverse inclusion).
The intuition is that this makes inequality \eqref{eq:KKTbound} 
tight. This constraint is therefore expected to slow down convergence.
    \item 
        \label{favorable}
        \emph{No-overlap setting.}
    In the second setting this constraint is removed and the support
    of $\bB^*$ is picked uniformly
    at random as a subset of
    $\{1,...,p\}\setminus \supp(\bSigma^{-1}\be_1)$.
\end{enumerate}

Assume that the tuple $(s, s_\Omega)$ is fixed. 
We sample $N_{sim}=128$ instances of $(\bX,\bE)$. 
For each sample, we compute the estimator $\hbB$ using the function \verb|MultiTaskElasticNet| from the Python library Scikit-learn \cite{scikit-learn},
build the interaction matrix $\hbA$ using the implementation from \Cref{sec:computing} 
and collect the pivotal quantities appearing in the Theorems.
The Q-Q plots and histograms for different pairs $(s, s_\Omega)$ are then reported 
in \Cref{fig:QQ1,fig:QQ2} for the overlapping supports setting
\ref{unfavorable}
and \Cref{fig:QQ3,fig:QQ4} for the no-overlap setting
\ref{favorable}.

Asymptotic normality is observed empirically on \Cref{fig:QQ3} in the no-overlap setting,
both when $\bSigma$ is known (blue) and unknown (green). 
The convergence holds up well across a wide range of sparsity levels. 
In the overlapping supports setting of \Cref{fig:QQ1}, convergence is maintained if $\bSigma$ is known,
but in the unknown $\bSigma$ case it deteriorates fast when $\|\bSigma^{-1}\be_1\|_0$ grows.
This suggests
that condition \eqref{eq:assum:T-log-p-sqrt-n-unknown-Sigma}
is not an artefact of the proof.

The picture is different with chi-square results. 
In the no-overlap setting of \Cref{fig:QQ4},
convergence is observed across all sparsity levels for pivotal quantities
in \Cref{thm:chi2z0} (known $\bSigma$) and \Cref{thm:unknown-Sigma-chi2} (unknown $\bSigma$)
whereas an increase in $s$ slows down convergence in \Cref{thm:unknown-Sigma-chi2-hatsigma} (unknown $\bSigma$).
In the overlapping supports setting \ref{unfavorable} of \Cref{fig:QQ2}, pivotal quantities in \Cref{thm:chi2z0,thm:unknown-Sigma-chi2-hatsigma} exhibit the same behavior 
as in the previous setting whereas the one from \Cref{thm:unknown-Sigma-chi2} shows increasingly slower convergence as $\|\bSigma^{-1}\be_1\|_0$ grows.
Again, this suggests
that condition \eqref{eq:assum:T-log-p-sqrt-n-unknown-Sigma}
is not an artefact of the proof.

\subsection{The advantage of multi-task learning for narrower confidence intervals}
In \Cref{fig:Width} we illustrate the discussion around \eqref{eq:length-comparison}
by comparing the lengths of $95\%$ confidence intervals 
obtained via multi-task Lasso and single-task Lasso.
$\|\bSigma^{-1}\be_1\|_0$ is set to $5$ and the pair $(T,s)$ varies.
For a given $(T,s)$ and a sampled $(\bX, \bE)$ 
we compute the relative change $(\text{length}_{multi} - \text{length}_{single})/\text{length}_{single}$.
We collect these values over $N_{sim}=128$ samples and obtain the bottom figure.
Since the results with or without the overlap constraint in the supports are similar, only the no-overlap setting \ref{favorable} is shown.
In the upper figure, 
multi-task confidence interval lengths are pooled together over the samples and we compare them to the aggregate single-task lengths.
As a sanity check we observe that multi-task and single-task Lasso coincide when $T$ is equal to $1$.
For $s=15$, $\hbA$-based confidence intervals always have smaller length, which shrinks as $T$ increases.
When $T=20$ we observe a $40\%$ average gain in the width.
Exploiting several tasks thus provides better estimates than intervals based on the first task.
However, as $s$ grows, this effect fades gradually and when $s=100$ it is counterbalanced by high variance in the multi-task lengths.

\begin{figure}[H]
\makebox[\textwidth][c]{
    \includegraphics[scale=0.6]{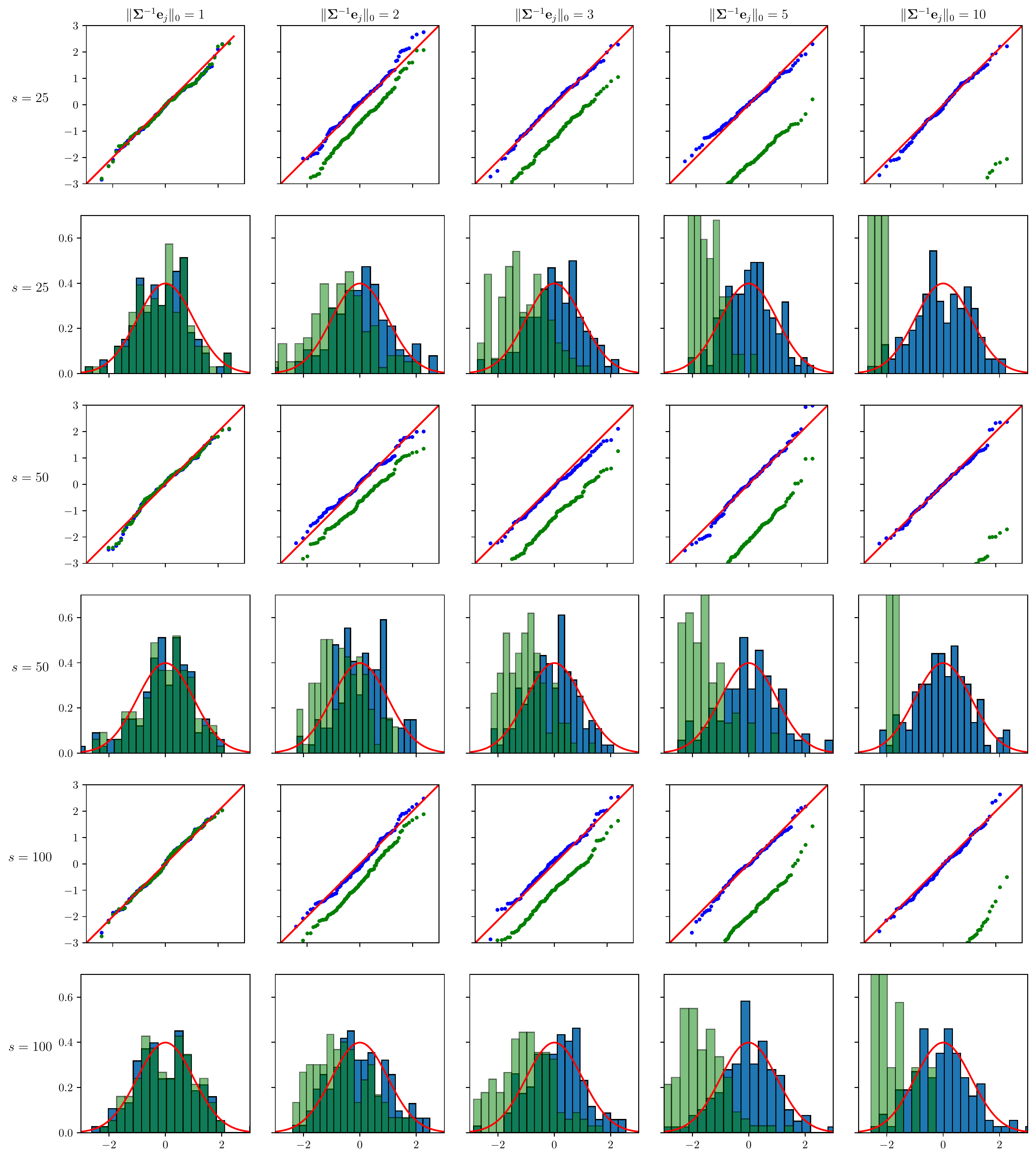}
    }
\caption{
    QQ-plots and histograms in the unfavorable setting \ref{unfavorable} for pivotal quantities in \Cref{thm-normal} (blue), \Cref{thm:unknown-Sigma-normal} (green).
    }
\label{fig:QQ1}
\end{figure}

\newpage 

\begin{figure}[H]
\makebox[\textwidth][c]{
    \includegraphics[scale=0.6]{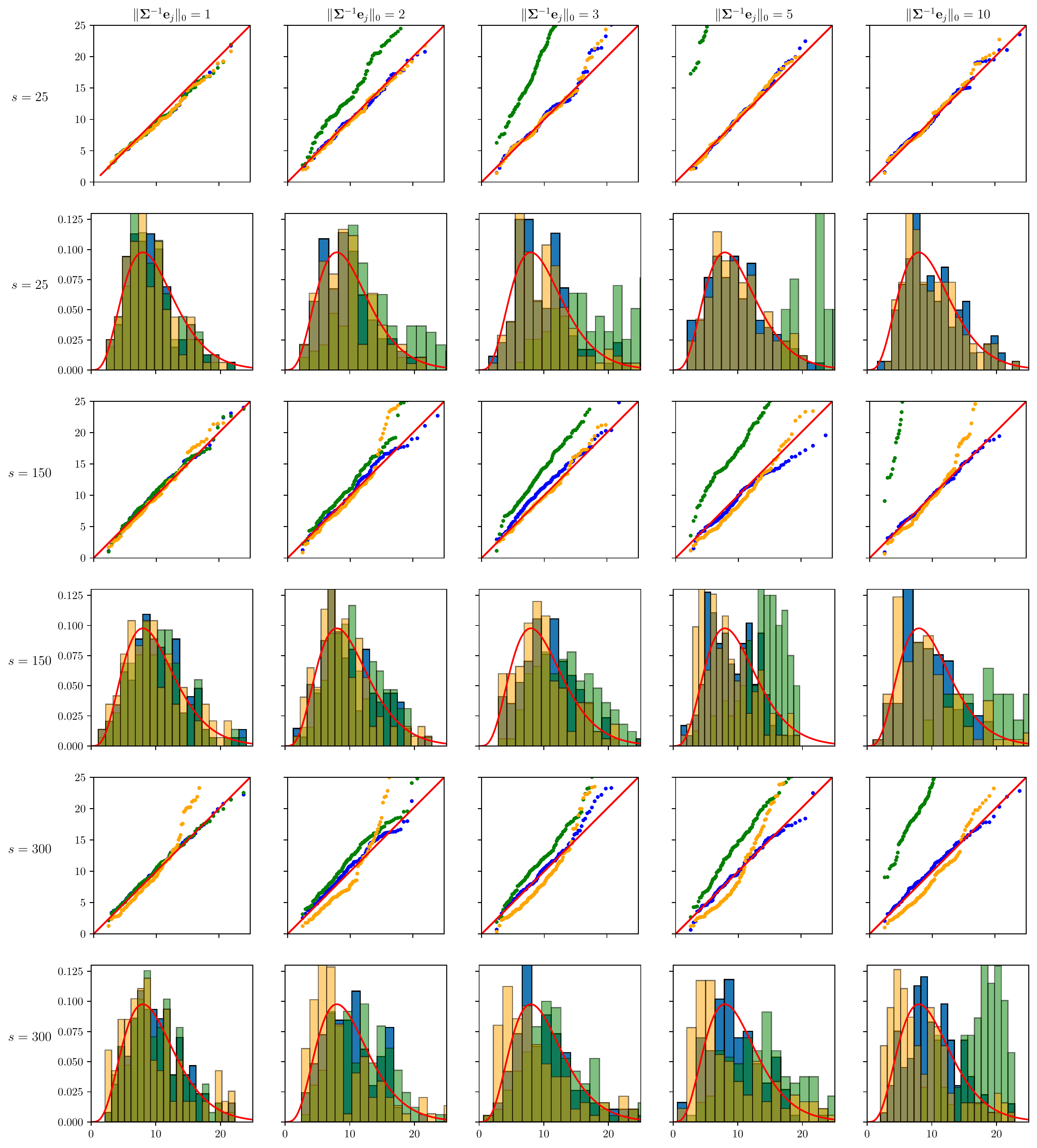}
    }
\caption{
    QQ-plots and histograms in the unfavorable setting \ref{unfavorable} for pivotal quantities in \Cref{thm:chi2z0} (blue), \Cref{thm:unknown-Sigma-chi2} (green), \Cref{thm:unknown-Sigma-chi2-hatsigma} (orange).
    }
\label{fig:QQ2}
\end{figure}

\newpage

\begin{figure}[H]
\makebox[\textwidth][c]{
    \includegraphics[scale=0.6]{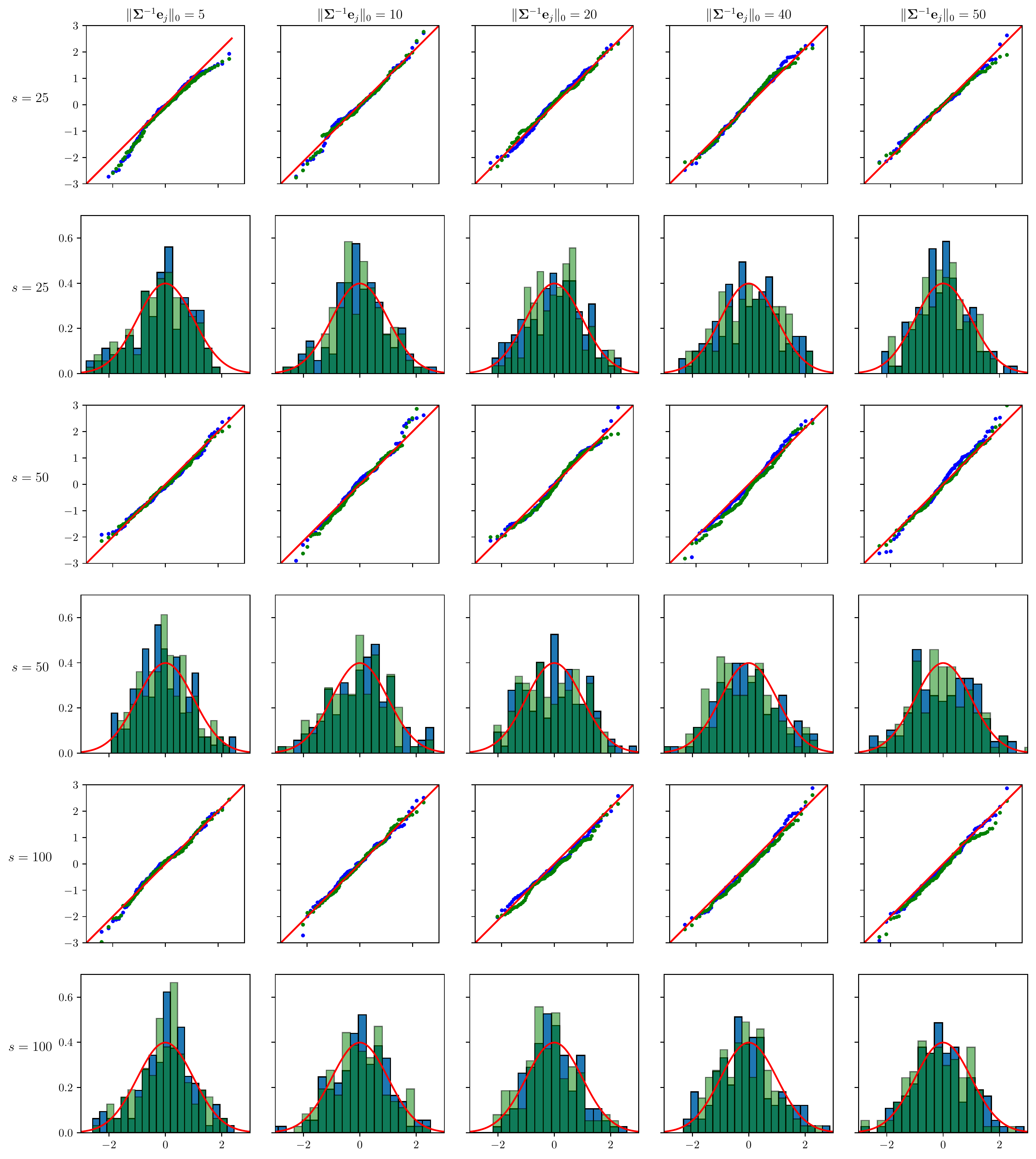}
    }
\caption{
    QQ-plots and histograms in the favorable setting \ref{favorable} for pivotal quantities in \Cref{thm-normal} (blue), \Cref{thm:unknown-Sigma-normal} (green).
    }
\label{fig:QQ3}
\end{figure}

\newpage 

\begin{figure}[H]
\makebox[\textwidth][c]{
    \includegraphics[scale=0.6]{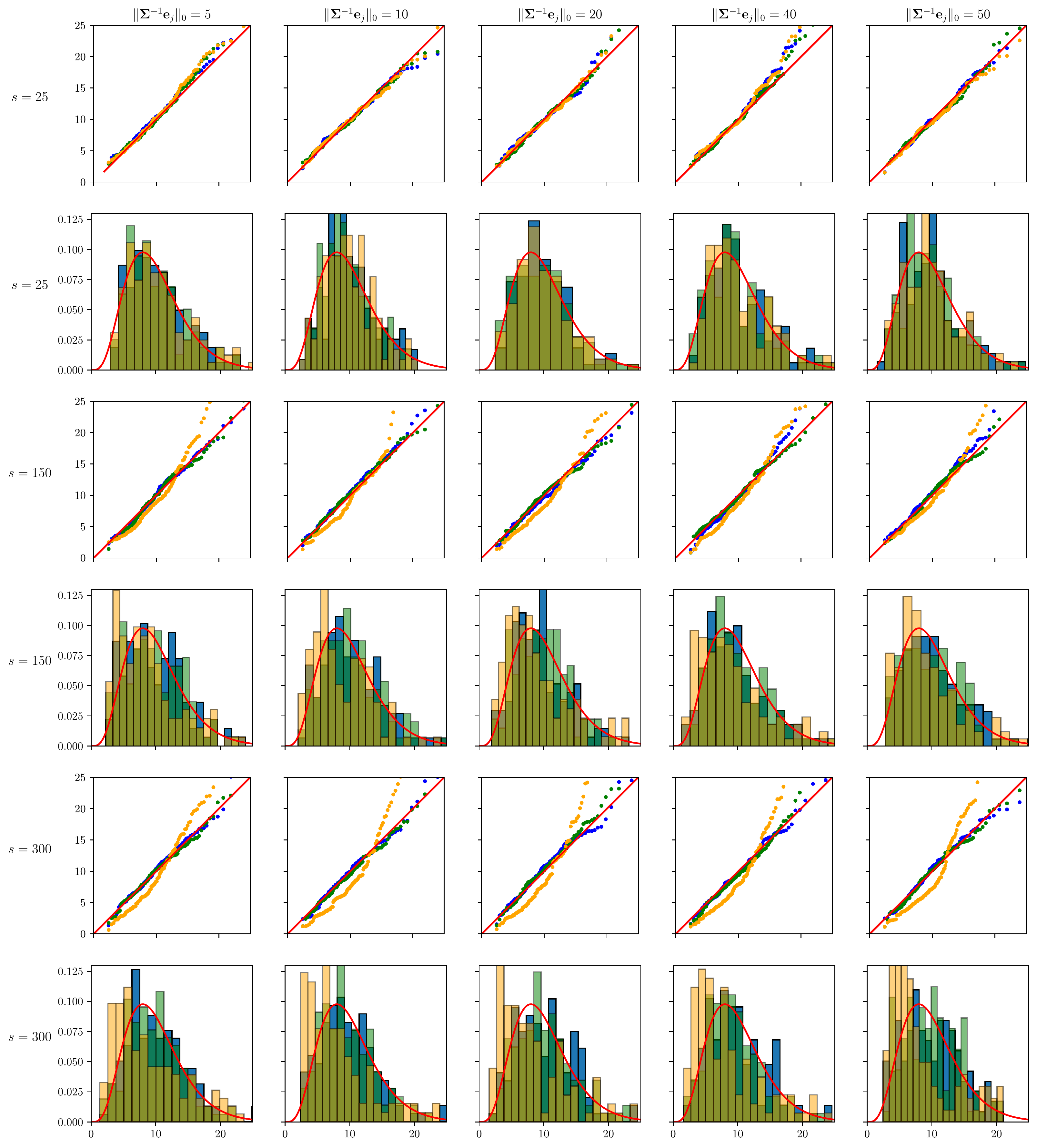}
    }
\caption{
    QQ-plots and histograms in the favorable setting \ref{favorable} for pivotal quantities in \Cref{thm:chi2z0} (blue), \Cref{thm:unknown-Sigma-chi2} (green), \Cref{thm:unknown-Sigma-chi2-hatsigma} (orange).
    }
\label{fig:QQ4}
\end{figure}

\newpage 

\begin{landscape}
\begin{figure}[htbp]
  \hspace*{-2cm}\includegraphics[scale=0.45]{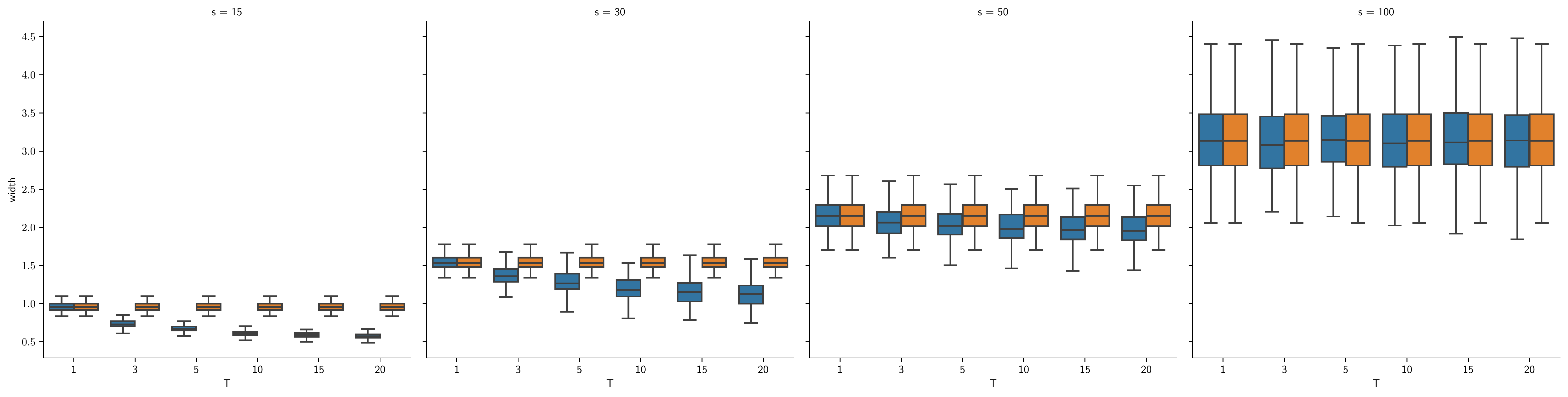}
  \\
  \vspace{0.3in}
  \hspace*{-2cm}\includegraphics[scale=0.45]{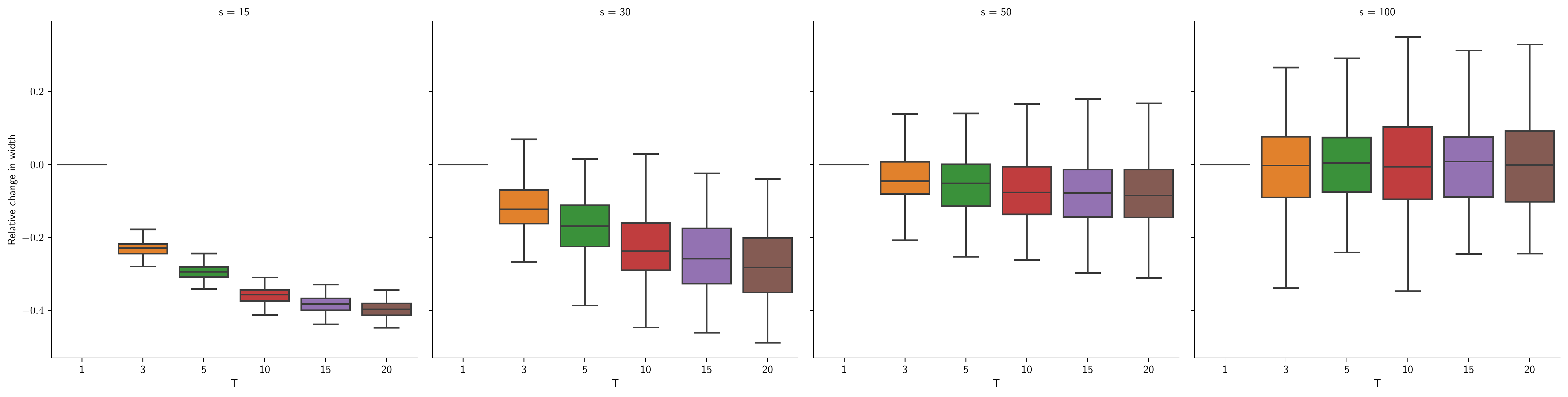}
  \\
  \vspace{0.3in}
\caption{
    Top: boxplots for lengths of 95\% confidence intervals using multi-task Lasso (blue) and single-task Lasso on the first task (orange). \\
    Bottom: boxplots for relative change in length with single-task as reference.
    Only the no-overlap setting \ref{favorable} is shown and $\|\bSigma^{-1}\be_1\|_0$ is set to $5$.
    }
\label{fig:Width}
\end{figure}
\end{landscape}

\newpage 
\bibliographystyle{plain}
\bibliography{biblio}

\newpage
\section*{SUPPLEMENT}
\appendix

\section{Intuition}
\label{sec:intuition}

Let us give some rationale behind the pivotal quantities stated in the
main theorems.
In this paragraph and only in this paragraph for the sake of providing
some intuition, we assume that $\ba=\be_j$ for some
canonical basis vector in $\R^p$ and 
that $\bSigma=\bI_{p\times p}$ so that entries of $\bX$ are i.i.d. $\mathcal N(0,1)$.
In this setting, the random vector $\bz_0=\bX\be_j$ has i.i.d. $\mathcal N(0,1)$ entries
and is independent of $\bX_{-j}$, the matrix $\bX$ with $j$-th column removed.
Since $\bz\sim \mathcal N_n({\mathbf 0}, \bI_{n\times n})$, Stein's formula 
\cite{stein1972bound,stein1981estimation}
states
that $\E[\bz_0^\top \mathbf f(\bz_0)]=\E[\sum_{i=1}^n (\partial/\partial z_{0i})f_i(\bz_0)]$
for any differentiable vector field $\mathbf f = (f_1,...,f_n)$ with $\mathbf f:\R^n\to\R^n$, under integrability conditions. For the sake of the current informal argument,
assume that Stein's formula provides reasonable approximation.
Then applying Stein's formula to $\mathbf f = (\bY-\bX\hbB)\be_t$
for each task $t=1,...,T$ (here, $\be_t$ is the $t$-th canonical basis vector in $\R^T$),
by nontrivial computations that are made rigorous in the proofs given in the supplement,
the approximations
\begin{equation}
\begin{aligned}
    \bz_0^\top (\bY-\bX\hbB)\be_1  & \approx n \ba^\top(\hbB-\bB^*)\be_1
    - \sum\nolimits_{t=1}^T\hbA_{1t}\ba^\top(\hbB-\bB^*)\be_t
    ,
    \\
    \bz_0^\top (\bY-\bX\hbB)\be_2  & \approx n \ba^\top(\hbB-\bB^*)\be_2 
    - \sum\nolimits_{t=1}^T\hbA_{2t} \ba^\top(\hbB-\bB^*)\be_t
    ,
    \\ \vdots &
    \\
    \bz_0^\top (\bY-\bX\hbB)\be_T  & \approx n \ba^\top(\hbB-\bB^*)\be_T
    - \sum\nolimits_{t=1}^T\hbA_{Tt} \ba^\top(\hbB-\bB^*)\be_t
\end{aligned}
\label{eq:linear-system-intuition}
\end{equation}
hold up to smaller order terms, where $\hbA$ is the interaction matrix
in \Cref{def-A}. By viewing \eqref{eq:linear-system-intuition} as a linear system
with $T$ equations and the $T$ unknowns $(\ba^\top(\hbB-\bB^*)\be_t)_{t=1,...,T}$,
and assuming that solving the linear system maintains the approximations, we obtain
that 
$$
\begin{pmatrix}
\ba^\top(\hbB-\bB^*)\be_1 \\
\ba^\top(\hbB-\bB^*)\be_2 \\
\vdots \\
\ba^\top(\hbB-\bB^*)\be_T \\
\end{pmatrix}
\approx
\Big(n \bI_{T\times T} - \hbA\Big)^{-1}
\begin{pmatrix}
\bz_0^\top(\bY-\bX\hbB)\be_1 \\
\bz_0^\top(\bY-\bX\hbB)\be_2 \\
\vdots \\
\bz_0^\top(\bY-\bX\hbB)\be_T \\
\end{pmatrix}
$$
or equivalently
$
(\hbB-\bB^*)^\top\ba
=
\bigl(n\bI_{T\times T} - \hbA\bigr)^{-1}
(\bY-\bX\hbB)^\top\bz_0
$.
Thus the matrix product of $(n\bI_{T\times T} - \hbA)^{-1}$ times
the residuals 
projected onto $\bz_0$
provides us with estimates of the bias of $\hbB$ on the direction $\ba\in\R^p$.
This informal argument is the crux of the rigorous methodology
developed in the next subsections.
In the sequel, we drop the assumption that $\bSigma=\bI_{p\times p}$.
When $\bSigma\ne \bI_{p\times p}$ is known as in \Cref{sec:known-Sigma}, the score
vector $\bz_0$ in \eqref{eq:linear-system-intuition} has to be replaced
by a random vector proportional to $\bX\bSigma^{-1}\ba$.
When $\bSigma$ is unknown as in \Cref{sec:unknown-Sigma}, the score vector
has to be estimated.

\section{Proof of Proposition~\NoCaseChange{\ref{propMatrixA}}}
\label{sec:proof-matrix}

We restate the proposition for convenience.

\propMatrixA*

\begin{proof}
    (i) 
    We have the following equalities:
        \begin{align*}
            \bu^\top\hbA\bv 
        &\stackrel{(i)}{=}
        \trace\big[ (\bu^{\top} \otimes \bX_{\hat S})
        [\tbX{}^T\tbX + nT \tbH ]^{\dagger}
        (\bv \otimes (\bX_{\hat S})^{{\top}})
    \big]\\
     	&\stackrel{(ii)}{=}
    \trace\big[ (\bv\bu^{\top} \otimes (\bX_{\hat S})^\top\bX_{\hat S})
        [\tbX{}^T\tbX + n T \tbH ]^{\dagger}
    \big]
    \\
    	&\stackrel{(iii)}{=}
    \trace\big[ 
        [{\tbX{}^T\tbX + n T \tbH ]^{\dagger}}^\top
    (\bv\bu^{\top} \otimes (\bX_{\hat S})^\top\bX_{\hat S})^\top
    \big]
    \\
    	&\stackrel{(iv)}{=}
    \trace\big[ 
        [\tbX{}^T\tbX + n T \tbH ]^{\dagger}
    (\bu\bv^{\top} \otimes (\bX_{\hat S})^\top\bX_{\hat S})
    \big] 
    \\
    	&\stackrel{(v)}{=}
    \trace\big[ (\bu\bv^{\top} \otimes (\bX_{\hat S})^\top\bX_{\hat S})
        [\tbX{}^T\tbX + n T \tbH ]^{\dagger}
    \big]
    = \bv^\top\hbA\bu
        \end{align*}
        where $(i)$ follows from \eqref{eq:quadra-A}, $(ii)$ is a consequence of $\trace[\bM_1\bM_2]=\trace[\bM_2\bM_1]$ and the mixed product property \eqref{mixed-product-property}, $(iii)$ and $(v)$ follow from $\trace[\bM] = \trace [\bM{}^\top]$, $(iv)$ holds because the pseudoinverse preserves symmetry.

This proves that $\hbA$ is symmetric.
Since the pseudoinverse of a positive semi-definite matrix is positive semi-definite as well, we also have 
\begin{equation}
    \label{eq:expression-bu-A-bu}
    \bu^\top\hbA\bu = 
    \|([\tbX{}^T\tbX + n T\tbH ]^{\dagger})^{1/2}
    (\bu\otimes \bX_{\hat S}^\top)\|_F^2 \ge0
\end{equation}
so that $\hbA$ is positive semi-definite.

(ii) 
Recall that $\tbX =  \bI_{T\times T} \otimes \bX_{\hat S}$.
By properties of Gram matrices, $\rank(\tbX{}^\top\tbX) = \rank(\tbX) = |\hat S|T$, hence by the rank-nullity theorem, $\ker(\tbX{}^\top\tbX)$ has dimension $(p-|\hat S|)T$.
By definition of $\bX_{\hat S}$, each vector $\be_t\otimes \be_j$ is in the kernel of $\tbX$ for $j\notin\hat S$ and $t\in[T]$, hence in $\ker(\tbX{}^\top\tbX)$. These $(p-|\hat S|)T$ vectors are linearly independent, so they form a basis of $\ker(\tbX{}^\top\tbX)$.\\
Besides, since $\tbH = \sum_{k\in \hat S} \bH^{(k)} \otimes (\be_k\be_k^\top)$, the mixed product property of Kronecker products \eqref{mixed-product-property} implies that $\tbH (\be_t\otimes \be_j) = \mathbf 0$ for $j\notin\hat S$ and $t\in[T]$, hence $\ker(\tbX{}^\top\tbX) \subset \ker(\tbX{}^\top\tbX + nT\tbH)$. Since these matrices are positive semi-definite, it is easy to check that the reverse inclusion holds as well, so that $\ker(\tbX{}^\top\tbX) = \ker(\tbX{}^\top\tbX + nT\tbH)$.

Since $\tbH$ is positive semi-definite,
$\tbX{}^\top\tbX\preceq \tbX{}^\top\tbX + nT\tbH$ holds in the sense of the positive semi-definite order, and 
\begin{equation}
    \label{eq:psd-order-pseudo-inverse}
(\tbX{}^\top\tbX + nT \tbH)^\dagger\preceq (\tbX{}^\top\tbX)^\dagger
\end{equation}
holds
because the two matrices have the same kernel, see \cite{stackexchange1}.
Next, using \eqref{eq:expression-bu-A-bu},
\begin{align*}
   \bu^\top\hbA\bu  &
   =
\|([\tbX{}^T\tbX]^{\dagger})^{1/2}
(\bu\otimes (\bX_{\hat S})^\top)\|_F^2
+
\trace[(\bu^\top \otimes \bX_{\hat S})
\{
[\tbX{}^T\tbX + nT \tbH ]^{\dagger}
-
[\tbX{}^\top\tbX]^\dagger
\}
(\bu \otimes (\bX_{\hat S})^\top)
]
                  \\&\le
\|([\tbX{}^T\tbX]^{\dagger})^{1/2}
(\bu\otimes (\bX_{\hat S})^\top)\|_F^2
                  \\&=
\trace[
(\bu^\top \otimes \bX_{\hat S})
(\bI_{T\times T} \otimes (\bX_{\hat S}^\top\bX_{\hat S})^\dagger)
(\bu \otimes (\bX_{\hat S})^\top)
]
                  \\&=
                  (\bu^\top\bu) \trace[\bX_{\hat S}(\bX_{\hat S}^\top\bX_{\hat S})^\dagger (\bX_{\hat S})^\top]\\
                  &= \|\bu\|^2 |\hat S|,
\end{align*}
where the first inequality follows from \eqref{eq:psd-order-pseudo-inverse}
and the third and fourth line follow respectively from
$\tbX{}^\top\tbX = \bI_{T\times T}\otimes \bX_{\hat S}$ and the mixed product
property \eqref{mixed-product-property}. The last line stems from
the fact that $\bX_{\hat S}(\bX_{\hat S}^\top\bX_{\hat S})^\dagger\bX_{\hat S}^\top$
is a projection matrix of rank $|\hat S|$ when $\rank(\bX_{\hat S})=|\hat S|$.

(iii) Since $\bX_{\hat S}$ has rank $|\hat S|$, we have by (ii) that $\|\hbA\|_{op}\leq |\hat S|<n$. Since $\hbA$ is positive-semi definite, its spectral norm is its largest eigenvalue, hence all the eigenvalues of $\hbA/n$ are $<1$, and $\bI_{T\times T}-\hbA/n$ is positive definite. For any $\bM\in\R^{T\times T}$ with $\|\bM\|_{op} < 1$
    we have
    $(\bI_{T\times T} - \bM)^{-1} = \sum_{k=0}^\infty \bM^k$.
    By the triangle inequality and the submultiplicativity
    of the operator norm,
    $$\|(\bI_{T\times T} - \bM)^{-1} - \bI_{T\times T}\|_{op}
    \le
    \|\bM\|_{op} \sum_{k=1}^\infty \|\bM\|_{op}^{k-1}
    =
    \|\bM\|_{op}/(1-\|\bM\|_{op}).
    $$
\end{proof}

\section{Preliminaries}
\label{sec:prelim}

In this section we develop a series of technical lemmas that will be used 
for proving \Cref{sec:normal-thms,sec:chi2}.
    We consider model \eqref{model-matrix} and the estimator $\hbB$ from \eqref{eq:hbB-Lasso}.
    Let $\eta_1>0$, $\eta_2\geq 2$, $\eta_3,\eta_4\in (0,1)$ and set
    $\lambda,\lambda_0$ as in \eqref{eq:lambda-lambda-0}.
    Define the sparsity level
    \begin{equation}
        \bar s = 
        s\Big(1/T
        +
        \frac{4\|\bSigma\|_{op} (1+\eta_4)^2
        }{\kappa^2} (2+\eta_2+1/\sqrt T)^2
        \Big) 
        \frac{2}{
        (\lambda/\lambda_0 -1)^2
        }
        \label{eq:upper-bound-hat-S}
            \end{equation}
    and note that $\bar s$ is of the same order as $s$ when the spectrum of $\bSigma$
    is bounded away from 0 and infinity as in \Cref{assum:main}.
    Let $\mathcal C = \{\bU\in \mathbb R^{p\times T}: \|\bU\|_{2,1}\leq 3\sqrt s \|\bU\|_F\}$, $\kappa = (1-\eta_3)\phi_{\min}(\bSigma)^{1/2}$ and define the events 
    \begin{align*}
        &\Omega_1=\Big\{\max_{\bU\in \mathcal C, \bU\ne 0}\Big|\frac{\|\bX \bU\|_F}{ \|\bSigma^{1/2}\bU\|_F\sqrt n} - 1\Big|  {<} \eta_3 \Big\}
            , \quad
        \Omega_2=\Big\{\sum_{j=1}^p (\|\bE^T\bX \be_j\|_2-nT\lambda_0)_+^2  {<}  s n^2T \lambda_0^2\Big\}
        ,
        \\
        &\Omega_3 = 
        \Big\{
            \max_{B\subset[p]:|B|\le s + 2 \bar s + 1}
            \Big(\max_{\bv\in\R^p:\supp(\bv)\subset B}\Big|\frac{\|\bX\bv\|_2}{\sqrt n\|\bSigma^{1/2}\bv\|_2} - 1\Big|
    \Big)  {<} 
    \eta_4 
        \Big\},
        \quad
        \Omega_4 =
        \Big\{\|\bE\|_{op} {<} \sigma(2\sqrt n + \sqrt T)\Big\}
      \end{align*}
      as well as
      \begin{equation}
         \Omega_* = \Omega_1\cap\Omega_2\cap\Omega_3\cap\Omega_4.
      \end{equation}
      Since the only randomness is with respect to $(\bX,\bE)$, we view
      the underlying probability space as $\Omega = (\R^{n\times p})\times (\R^{n\times T})$
      and
      $\Omega_1,\Omega_2,\Omega_3,\Omega_4,\Omega_*$ as subsets of $\Omega$ so that $\Omega_i$ occurs
      if and only if $(\bX,\bE)\in\Omega_i$ for each $i=1,2,3$.

\begin{restatable}{lemma}{lemmaProba}
    \label{lemma:proba}
Let \Cref{assum:main} be fulfilled.
Then  $\P(\Omega_*)\to 1$.
\end{restatable}

\begin{restatable}{lemma}{lemmaRisk}
    \label{lemma:risk}
        On $\Omega_*$  we have:
        \begin{enumerate}
             \item $\hbB - \bB^*\in\mathcal C$,
             \item  $n^{-1/2}\|\bX(\hbB - \bB^*)\|_F \leq 
                (1-\eta_{{3}})
                \bar R$,
            \item $\|\bSigma^{1/2}(\hbB - \bB^*)\|_F
                \le \bar R$,
            \item
                $
                \|\hbB-\bB^*\|_{2,1}
                \le
                3\sqrt s
                \|\hbB-\bB^*\|_F
                \le
                3 \sqrt s
                \phi_{\min}(\bSigma)^{-1/2}
                \bar R
                $,
            \item $\|\bY-\bX\hbB\|_F^2 \leq 8\sigma^2nT + 2(1-\eta_3)^2n\bar R^2$,
         \end{enumerate}  
        where
        \begin{equation*}
            \bar R \coloneqq (1-\eta_{{3}})^{-1}{\kappa}^{-1} {2(1+\eta_1)(3+\eta_2)\sigma \max_j \bSigma_{jj}^{1/2}}\sqrt{sT/n} \left(1 + \sqrt{{(2/T)\log(p/s)}}\right).
        \end{equation*}
        Moreover, $\bar R \xrightarrow[n\to \infty]{} 0$ under \Cref{assum:main}.
\end{restatable}

\begin{restatable}{lemma}{lemmaSparsity}
    \label{lemma:sparsity}
            On $\Omega_*$, inequality 
            $|\hat S|\le \bar s$ holds with $\bar s$ in \eqref{eq:upper-bound-hat-S}.
\end{restatable}

\begin{restatable}{lemma}{lemmaRank}
    \label{lemma:rank}
            On $\Omega_*$ we have $\rank(\bX_{\hat S}) = |\hat S|$.
\end{restatable}

\begin{restatable}{lemma}{lemmaKKTstrict}
    \label{lemma:kkt-strict}
        For almost every $(\bX,\bE)$, the KKT conditions of $\hbB$ in \eqref{eq:hbB-Lasso} hold strictly in the sense that
            $\P(\max_{j\notin \hat S} \|(\bY-\bX\hbB)^\top\bX\be_j\|_2 < nT \lambda) = 1$.
\end{restatable}

\begin{restatable}{lemma}{lemmaLipschitz}
    \label{lemma:lipschitz}
        Given the noise matrix $\bE$  and two design matrices $\bX,\bXbar$ define $\hbB$ in \eqref{eq:hbB-Lasso}
            and $\bBbar$ by
            $$\bBbar = {\textstyle \argmin_{\bB\in\R^{p\times T}} }
            \Big(
                \tfrac{1}{2nT} \|\bE + \bXbar(\bB^* - \bB)\|_F^2 + \lambda \|\bB\|_{2,1}
            \Big).$$
            If $\bX,\bXbar,\bE$ are such that
            both $\{(\bX,\bE), (\bXbar,\bE) \}\subset \Omega_*$ 
            then
            \begin{align*}
            n^{1/2 } \|\bSigma^{1/2}(\hbB-\bBbar)\|_F
            &\le
            \C(\eta_4)
            (\bar R + \|\bE\|_{op}n^{-1/2})
            \|(\bX-\bXbar)\bSigma^{-1/2}\|_F
            ,
            \\
            \|\bXbar(\bBbar-\bB^*) - \bX(\hbB-\bB^*)\|_F
            &\le
            \C(\eta_4)
            (\bar R + \|\bE\|_{op}n^{-1/2})
            \|(\bX-\bXbar)\bSigma^{-1/2}\|_F
            \end{align*}
            for some constants that depend on $\eta_4$ only
            and $\bar R$ is defined in \Cref{lemma:risk}.
\end{restatable}

\begin{restatable}{lemma}{lemmaDifferential}
    \label{lemma:differential}
        For almost every $(\bX,\bE)$ in the open set
            $\Omega_1\cap\Omega_2\cap\Omega_3$,
            $\hbB$ is a Fréchet differentiable function of $\bX$.
            For almost every $(\bX,\bE)$ in 
            $\Omega_1\cap\Omega_2\cap\Omega_3$, if
            $$\hbB(\bw) = {\textstyle \argmin_{\bB\in\R^{p\times T}} }
            \Big(
                \tfrac{1}{2nT} \|\bE + (\bX + \bw\ba^\top)(\bB^* - \bB)\|_F^2 + \lambda \|\bB\|_{2,1}
            \Big)$$
            is the estimate \eqref{eq:hbB-Lasso} with $\bX$ replaced
            by the perturbed design $\bX+\bw\ba^\top$,
            then for any $\bb\in\R^T$
            \begin{equation*}
                \Big((\bX+\bw\ba^\top)( \hbB(\bw)- \bB^*)\Big)\bb
                -
                \Big(\bX(\hbB-\bB^*)\Big)\bb
                =
                \big(\mathsf{D}(\bb) \big) \bw
                +
                o(\|\bw\|)
            \end{equation*}
            as $\|\bw\|\to 0$,
            where $\mathsf{D}: \R^T \to \R^{n\times n}$ is a linear map given by
            $
            \mathsf{D}(\bb)
            =
            \mathsf{D}^*(\bb)
            +
            \mathsf{D}^{**}(\bb)
            $ 
            with
            \begin{align*}
                \mathsf{D^*}(\bb)
            &=
            (\ba^\top(\hbB-\bB^*)\bb )
            \bI_{n\times n}
            -
            (\bb^\top \otimes \bX_{{\hat S}})
            \big(\tbX{}^\top\tbX + nT \tbH \big)^\dagger
            \big(((\hbB-\bB^*)^\top \ba)\otimes \bX_{\hat S}^\top\big)
            \\ \nonumber
            &=
            (\ba^\top(\hbB-\bB^*)\bb )
            \bI_{n\times n}
            -
            (\bb^\top \otimes \bX_{{\hat S}})
            \left(\tbX^\top\tbX +  nT \tbH\right)^\dagger
 \begin{pmatrix}
     \ba^\top (\hbB - \bB^*)\be_1 \bX_{{\hat S}}^\top \\
   \vdots\\
   \ba^\top (\hbB - \bB^*)\be_T \bX_{{\hat S}}^\top 
 \end{pmatrix} ,
 \\
            \mathsf{D^{**}}(\bb)
            &=
            (\bb^\top \otimes \bX_{{\hat S}})
              \big(\tbX{}^\top\tbX + nT \tbH \big)^\dagger
              ((\bY-\bX\hbB)^\top \otimes \ba_{\hat S} ) 
            \\ \nonumber
            &=(\bb^\top \otimes \bX_{{\hat S}})
            \left(\tbX^\top\tbX +  nT\tbH\right)^\dagger
 \begin{pmatrix}
     \ba_{{\hat S}}  \be_1^\top(\bY - \bX\hbB)^\top\\
  \vdots \\
  \ba_{{\hat S}}  \be_T^\top(\bY - \bX\hbB)^\top
\end{pmatrix}
            \end{align*}
            for all $\bb\in\R^T$ and $\bw\in\R^n$.
            Note that $\mathsf{D},\mathsf{D}^*$ and $\mathsf{D}^{**}$ implicitly depend on $(\bX,\bE)$.
            Hence the matrix $\mathsf{D}(\bb)$ of size $n\times n$
            is the Jacobian of the map
            $\bw \mapsto (\bX+\bw\ba^\top)(\hbB(\bw)-\bB^*)\bb$
            at $\bw= \mathbf 0$.

\end{restatable}

\begin{restatable}{lemma}{lemmaDivergence}
    \label{lemma:divergence}
            For any $\bb \in \R^T$
            we have 
            on $\Omega_*$
            \begin{align}
                \trace[\mathsf{D}^*(\bb)]
            &=
            \bb^\top(n \bI_{T\times T} - \hbA)
            (\hbB - \bB^*)^\top \ba,
            \\
                \sum_{t=1}^T
            \Big(
                \trace[\mathsf{D}^{**}(\be_t)]
            \Big)^2
            &\le 
            \C(\bSigma) \sigma^2 sT 
            \label{eq:divergence-D**-inequality}
            \end{align}
            for some constant depending on $\bSigma$
            and $\eta_1,...,\eta_4$ only.
\end{restatable}

\begin{restatable}{lemma}{lemmaRemainderII}
    \label{lemma:RemainderII}
        Under \Cref{assum:main}, as $n,p\to+\infty$ we have
            $$
\frac{1}{\sigma^2 n}
\E\Big[
    I\{\Omega_*\}
\sum_{t=1}^T
\Big(
\bz_0^\top\bX(\hbB - \bB^*)\be_t
- \trace[\mathsf{D}(\be_t)]
\Big)^2
\Big] \longrightarrow 0.
$$
Since $\Omega_*$ has probability approaching one, this implies that
$
\frac{1}{\sigma^2 n}
\sum_{t=1}^T
(
\bz_0^\top\bX(\hbB - \bB^*)\be_t
- \trace[\mathsf{D}(\be_t) ]
)^2$ converges to 0 in probability.
\end{restatable}

We now prove each lemma. The lemmas are restated before their proofs for convenience.

\lemmaProba*
\begin{proof}[Proof of  \Cref{lemma:proba}]
    The fact that $\Omega_* = \Omega_1\cap\Omega_2\cap\Omega_3\cap\Omega_4$ has probability
    approaching one under \Cref{assum:main} follows
    from the propositions in \Cref{sec:probabliistic-propositions}:
    \Cref{prop:restricted-eig} {\textit{(iii)}} with $k={9s}$ and ${x=\log n}$, 
    \Cref{prop:noise},
    \Cref{prop:gordon} applied with $k=s + 2 \bar s + 1$,
    and 
    $\P(\Omega_4)\ge 1- e^{-n/2}$ by \cite[Theorem II.13]{DavidsonS01}.
\end{proof}

\lemmaRisk*
\begin{proof}[Proof of \Cref{lemma:risk}]
In the whole proof we place ourselves on the event $\Omega_*$. We prove first that $\hbB-\bB^*\in \mathcal C$. \\By the definition of $\hbB$, 
$\frac 1{nT} \|\bX\hbB-\bY\|_F^2 + 2\lambda \|\hbB\|_{2,1} \leq \frac 1{nT} \|\bX\bB^*-\bY\|_F^2 + 2\lambda \| \bB^*\|_{2,1}$.
Rewriting the LHS as $\frac 1{nT} \|\bX(\hbB-\bB^*)+  (\bX\bB^*-\bY)\|_F^2 + 2\lambda \|\hbB\|_{2,1}$ and expanding the square yields 
$\|\bX(\hbB-\bB^*)\|_F^2 \leq  2 \langle \bX(\hbB-\bB^*),\bE  \rangle_F + 2nT\lambda(\|\bB^*\|_{2,1} - \|\hbB\|_{2,1})$.\\
The following chain of inequalities holds
\begin{align*}
  \langle \bX(\hbB-\bB^*),\bE  \rangle_F 
  &\stackrel{(i)}{\leq} \sum_{j=1}^p \|(\hbB-\bB^*)^{\top}\be_j\|_2 \cdot \|(\bX^T\bE)^{\top}\be_j\|_2\\
  &\stackrel{(ii)}{\leq} \sum_{j=1}^p \|(\hbB-\bB^*)^{\top}\be_j\|_2 \left[(\|\bE^T\bX \be_j\|_2-nT\lambda_0)_+ + nT\lambda_0\right]\\
  &\stackrel{(iii)}{\leq} \|\hbB-\bB^*\|_F  \Big(\sum_{j=1}^p (\|\bE^T\bX \be_j\|_2-nT\lambda_0)_+^2\Big)^{1/2} + nT\lambda_0 \|\hbB-\bB^*\|_{2,1}\\
  &\stackrel{(iv)}{\leq} \|\hbB-\bB^*\|_F \sqrt s n \sqrt T \lambda_0 + nT\lambda_0 \|\hbB-\bB^*\|_{2,1}.
\end{align*}
$(i)$ and $(iii)$ follow from Cauchy-Schwarz inequality, $(ii)$ stems from the inequality $a\leq (a-b)_+ + b$ and $(iv)$ holds on $\Omega_2$.
Thus 
\begin{equation}
  \label{eq:ineq-proof-risk}
  \|\bX(\hbB-\bB^*)\|_F^2 \leq  2\|\hbB-\bB^*\|_F \sqrt s nT \lambda_0 
  + 2nT\big[
      \lambda_0 \|\hbB-\bB^*\|_{2,1} + \lambda(\|\bB^*\|_{2,1} - \|\hbB\|_{2,1})
  \big].
\end{equation}
Besides, the quantity inside the bracket on the right hand side satisfies
\begin{align*}
  &\phantom{{}\stackrel{(i)}{=}{}} 
  \lambda_0 \|\hbB-\bB^*\|_{2,1} + \lambda(\|\bB^*\|_{2,1} - \|\hbB\|_{2,1})\\
  &\stackrel{(i)}{=} \lambda_0\sum_{j\in S} \|(\hbB-\bB^*)^{\top}\be_j\|_2 + \lambda_0\sum_{j\notin S} \|\hbB^{\top}\be_j\|_2 + \lambda \Big(\sum_{j\in S} \|{\bB^*}^{\top}\be_j\|_2- \|\hbB^{\top}\be_j\|_2 \Big) - \lambda \sum_{j\notin S} \|\hbB^{\top}\be_j\|_2\\
  &\stackrel{(ii)}{\leq} \lambda_0 \sqrt s \|\hbB-\bB^*\|_F + \lambda_0\sum_{j\notin S} \|\hbB^{\top}\be_j\|_2  + \lambda \sum_{j\in S} \|(\bB^*-\hbB)^{\top}\be_j\|_2 - \lambda \sum_{j\notin S} \|
    \hbB^{\top}\be_j\|_2\\
  &\stackrel{(iii)}{\leq} (\lambda_0+\lambda) \sqrt s \|\hbB-\bB^*\|_F + (\lambda_0-\lambda)\sum_{j\notin S} \|\hbB^{\top}\be_j\|_2,
\end{align*}
where $(ii)$ follows from Cauchy-Schwarz and the reverse triangle inequality applied respectively on the first and third summands of $(i)$, whereas $(iii)$ is a consequence of Cauchy-Schwarz.
Combining this bound with \eqref{eq:ineq-proof-risk} and plugging in the value $\lambda=(1+\eta_2)\lambda_0$ yields
\begin{equation}
  \label{eq:ineq-2-proof-risk}
  \|\bX(\hbB-\bB^*)\|_F^2 \leq 2nT \sqrt s \lambda_0\Big(2+\eta_2+\frac 1{\sqrt T}\Big)\|\hbB-\bB^*\|_F - 2nT \eta_2\lambda_0 \sum_{j\notin S} \|\hbB^{\top}\be_j\|_2.
\end{equation}
Non-negativity of the LHS, the equality 
$\sum_{j\notin S} \|\hbB^{\top}\be_j\|_2 =  \|\hbB-\bB^*\|_{2,1} - \sum_{j\in S} \|(\hbB-\bB^*)^{\top}\be_j\|_2$ and Cauchy-Schwarz lead to
$\|\hbB-\bB^*\|_{2,1}\leq (1+\frac 3{\eta_2} + \frac{1}{\eta_2 \sqrt T}) \sqrt s \|\hbB-\bB^*\|_{F}$.\\
Since $T\geq 1$ and $\eta_2\geq 2$, we get $\|\hbB-\bB^*\|_{2,1}\leq 3\sqrt s \|\hbB-\bB^*\|_{F}$, that is $\hbB-\bB^*\in \mathcal C$.

The inequality 
\begin{equation}
  \label{eq:ineq-3-proof-risk}
  \|\hbB-\bB^*\|_F\leq \|\bSigma^{-1/2}\|_{op}\|\bSigma^{1/2}(\hbB-\bB^*)\|_F = \phi_{\min}(\bSigma)^{-1/2}\|\bSigma^{1/2}(\hbB-\bB^*)\|_F
\end{equation}
 combined with \eqref{eq:ineq-2-proof-risk} and the event $\Omega_1$ yields
\begin{align*}
   \|\bX(\hbB - \bB^*)\|_F 
   &\leq 2\kappa^{-1} \Big(2+\eta_2+ T^{-1/2} \Big) \sqrt n T \lambda_0 \sqrt s \\
   &\leq {2{\kappa^{-1}}(1+\eta_1)(3+\eta_2)\sigma \max_j \bSigma_{jj}^{1/2}}\sqrt{sT} \Big(1 + \sqrt{{(2/T)\log(p/s)}}\Big)\\
   &= \sqrt n (1-\eta_3) \bar R.
\end{align*} 
Reusing $\Omega_1$, we obtain
$
\|\bSigma^{1/2}(\hbB - \bB^*)\|_F \leq ({\sqrt n (1-\eta_3)} )^{-1} {\|\bX(\hbB - \bB^*)\|_F } \leq \bar R
$.
\newline
Combining this last bound with \eqref{eq:ineq-3-proof-risk} yields 
$\|\hbB-\bB^*\|_F \leq \phi_{\min}(\bSigma)^{-1/2}\bar R$, hence (iv).
\newline
For inequality (v), using $\Omega_4$, $(a+b)^2\le 2a^2+2b^2$ and $\|\bE\|_F^2\le\rank(\bE)\|\bE\|_{op}^2$ we have 
\begin{align*}
    \|\bY - \bX\hbB\|_F^2 
    &\leq 2\|\bE\|_F^2 + 2\|\bX(\hbB - \bB^*)\|_F^2\\
    &\leq 2 \rank(\bE) \|\bE\|_{op}^2 + 2n(1-\eta_3)^2 \bar R^2\\
    &\leq 2 \min(n,T) (2\sigma\max(\sqrt n,\sqrt T))^2  + 2n(1-\eta_3)^2 \bar R^2\\
    &\leq 8\sigma^2nT + 2(1-\eta_3)^2n\bar R^2.
\end{align*}

Regarding the limit of $\bar R$, note that $\bar R \propto \big(\frac {sT}n\big)^{1/2} + \big(\frac sn \log(\frac ps)\big)^{1/2}$.
By \Cref{assum:main}, each summand goes to $0$ as $n$ goes to $\infty$.
\end{proof}

\lemmaSparsity*
\begin{proof}[Proof of \Cref{lemma:sparsity}]
The KKT conditions of \eqref{eq:hbB-Lasso} 
are given by
\begin{equation}
    \label{eq:KKT-multitask}
\begin{split}
    (\bY-\bX\hbB)^\top\bX\be_j &=  nT\lambda \|\hbB^\top\be_j\|_2^{-1} \hbB^\top \be_j
                               \qquad\text{ for all } j \in \hat S,
    \\
    \|( \bY-\bX\hbB)^\top\bX\be_j\|_2 &\le nT\lambda
                                   \qquad\text{ for all } j \notin \hat S
    .
\end{split}
\end{equation}
This implies that
$
\forall j\in \hat S,
\|(\bY-\bX\hbB)^\top \bX \be_j\|_2 = n T \lambda.
$
Since $\|(\bY-\bX\hbB)^\top \bX \be_j\|_2 
\le\|\bE^\top \bX \be_j\|_2  + \|(\bX(\bB^*-\hbB))^\top\bX\be_j\|_2 $
by the triangle inequality, we have for any $j\in \hat S$,
\begin{align}
    n T\lambda &\le (\|\bE^\top\bX\be_j\|_2 -nT \lambda_0)_+
+ nT \lambda_0 + \|(\bX(\bB^* - \hbB))^\top\bX\be_j\|_2 ,
\\nT(\lambda - \lambda_0) &\le 
(\|\bE^\top\bX\be_j\|_2 -nT \lambda_0)_+
+ \|(\bX(\bB^* - \hbB))^\top\bX\be_j\|_2 .
\end{align}
Summing the squares of the above inequalities for a subset $B\subset\hat S$
and using $(a+b)^2\le 2a^2+2b^2$,
we get
$$
\frac{
|B| n^2(\lambda-\lambda_0)^2T^2
}{2}
\le \sum_{j\in B}(\|\bE^\top\bX\be_j\|_2 -nT \lambda_0)_+^2
+
\trace(\{\bX(\bB^* - \hbB)\}^\top \Big\{\sum_{j\in B}\bX\be_j\be_j^\top\bX^{{\top}} \Big\}
\{\bX(\bB^* - \hbB)\}).
$$
The first term is bounded from above by $sn^2 T \lambda_0^2$ on the event $\Omega_2$. Dividing by $n^{2}T^2$ we find
$$
\frac{
|B|(\lambda-\lambda_0)^2
}{2}
\le
s \lambda_0^2/T
+
\frac{1}{nT^2}\|\bX(\hbB-\bB^*)\|_F^2 \psi_{\max}(B)
$$
where $\psi_{\max}(B)$ is the largest eigenvalue of
$\frac 1 n \sum_{j\in B}\bX\be_j\be_j^\top\bX^{{\top}}$,
or equivalently the largest eigenvalue of
$\frac 1 n (\bX_B\bX_B^{{\top}})$, which is also the largest eigenvalue of $\frac 1 n (\bX_B^{\top}\bX_B)$.
On the event of $\Omega_*$, we obtain
$$
\frac{
|B|(\lambda-\lambda_0)^2
}{2}
\le s\lambda_0^2 /T
+ 
\frac{\psi_{\max}(B)}{nT^2}
\frac{4}{\kappa^2} (2+\eta_2+1/\sqrt T)^2 n T^2 \lambda_0^2 s,
$$
or equivalently
$$
\frac{|B|(\lambda/\lambda_0 -1)^2}{2}
\le
s\Big(1/T
+
\frac{4\psi_{\max}(B)}{\kappa^2} (2+\eta_2+1/\sqrt T)^2
\Big).
$$

Let $\bar s$ be as in \eqref{eq:upper-bound-hat-S}
and assume that $|\hat S|\le \bar s$ is violated on $\Omega_*$.
Then on $\Omega_3$, any
$B\subset \hat S$ with size $|B|= \lfloor \bar s \rfloor + 1 $ satisfies
$\forall \bv\in \mathbb R^p$, $\|\bX_B \bv_B\|_2 \leq (1+\eta_4)\sqrt n \|\bSigma^{1/2}\|_{op} \|\bv_B\|_2$. Squaring yields
$\psi_{\max}(B) \le \|\bSigma\|_{op}{(1+\eta_4)^2}$.
Then
$$
\frac{
|B|(\lambda/\lambda_0 -1)^2
}{2}
\le
s\Big(1/T
+
\frac{4\|\bSigma\|_{op}
{(1+\eta_4)^2}
}{\kappa^2} (2+\eta_2+1/\sqrt T)^2
\Big)
$$
which shows that $|B|\le \bar s$ by definition of $\bar s$, a contradiction.
\end{proof}

\lemmaRank*

\begin{proof}
	By \Cref{lemma:sparsity}, we have $|\hat S|\leq \bar s$ on $\Omega_*$. Since $s\leq s+2\bar s + 1$, the event $\Omega_3$ yields  $\forall \bv\in \mathbb R^p$, $\supp(\bv) \subset \hat S \implies (1-\eta_4)\sqrt n \|\bSigma^{1/2}\bv\|_2 \leq \|\bX_{\hat S}\bv\|_2$. If $\bv$ is such that $\supp(\bv) \subset \hat S$ and $\bX_{\hat S}\bv=\mathbf 0$, then we must have $\bv=\mathbf 0$. Equivalently, the linear span of $(\be_j)_{j\in \hat S}$ has intersection $\{\mathbf 0\}$ with $\ker(\bX_{\hat S})$, hence $\ker(\bX_{\hat S})$ must be contained in the span of $(\be_j)_{j\notin \hat S}$. Thus $\dim \ker(\bX_{\hat S})\leq p-|\hat S|$ and by the rank-nullity theorem, $\rank(\bX_{\hat S})\geq |\hat S|$. By definition of $\bX_{\hat S}$, it is also clear that $\rank(\bX_{\hat S})\leq |\hat S|$, hence the conclusion.
\end{proof}

\lemmaKKTstrict*

\begin{proof}[Proof of \Cref{lemma:kkt-strict}]
    This follows from the argument in Lemma 6.4 of
    \cite[arXiv version v1, 24 Feb 2019]{bellec2019biasing}.
\end{proof}

\lemmaLipschitz*

\begin{proof}[Proof of \Cref{lemma:lipschitz}]
    By \Cref{lemma:sparsity}, $\hbB - \bBbar$ has at most $2\bar s$ non-zero rows.
     $\Omega_3$ applied on each column of $\bSigma^{1/2}(\hbB - \bBbar)$ gives $(1-\eta_4)^2 n \|\bSigma^{1/2}(\hbB - \bBbar)\|_F^2 \le\|\bX(\hbB-\bBbar)\|_F^2$. Similarly, using $\Omega_3$ with $\bXbar$ and summing the resulting inequality with the previous one yields
        $$2(1-\eta_4)^2 n \|\bSigma^{1/2}(\hbB - \bBbar)\|_F^2
                                    \le\|\bX(\hbB-\bBbar)\|_F^2 + \|\bXbar(\hbB-\bBbar)\|_F^2.$$
    Define $\varphi:\bB\mapsto \frac 1{2nT}\|\bE+\bX(\bB^*-\bB)\|_F^2 + \lambda \|\bB\|_{2,1}$, $\psi:\bB\mapsto \frac 1{2nT} \|\bX(\hbB-\bB)\|_F^2$ and $\gamma:\bB\mapsto \varphi(\bB) - \psi(\bB)$. When expanding the squares, it is clear that $\gamma$ is the sum of a linear function and of the convex penalty, thus $\gamma$ is convex. Additivity of subdifferentials yields $\partial \varphi (\hbB) = \partial \gamma(\hbB) + \partial \psi(\hbB) = \partial \gamma(\hbB)$. By optimality of $\hbB$ we have ${\mathbf 0}_{p\times T}\in \partial \varphi(\hbB)$, thus ${\mathbf 0}_{p\times T}\in \partial \gamma(\hbB)$. This implies $\gamma(\hbB)\leq \gamma(\bBbar)$. Letting $\bHbar = \bBbar - \bB^*$ and $\bH = \hbB - \bB^*$, the last inequality rewrites as 
    $$
    \|\bX(\hbB-\bBbar)\|_F^2
  \le
   \|\bE - \bX\bHbar\|_F^2 
   - \|\bE - \bX\bH\|_F^2
   + g(\bBbar) - g(\hbB).
    $$
    Summing the similar inequality obtained by replacing $\bX$ with $\bXbar$ yields 
    $$\|\bX(\hbB-\bBbar)\|_F^2 + \|\bXbar(\hbB-\bBbar)\|_F^2 \leq \|\bE - \bX\bHbar\|_F^2 
   - \|\bE - \bX\bH\|_F^2
   + \|\bE - \bXbar\bH\|_F^2 
   - \|\bE - \bXbar\bHbar\|_F^2.$$
        
    Combining the above displays, we obtain
    \begin{align*}
    &
    \phantom{{}\leq{}} 
    2(1-\eta_4)^2 n \|\bSigma^{1/2}(\hbB - \bBbar)\|_F^2
    \\
    & \le\|\bX(\hbB-\bBbar)\|_F^2 + \|\bXbar(\hbB-\bBbar)\|_F^2
  \\& \le
   \|\bE - \bX\bHbar\|_F^2 
   - \|\bE - \bX\bH\|_F^2
   + \|\bE - \bXbar\bH\|_F^2 
   - \|\bE - \bXbar\bHbar\|_F^2 
  \\&=
  \langle \bX(\bH-\bHbar), 2\bE - \bX(\bH+\bHbar) \rangle_F
  +
  \langle \bXbar(\bHbar-\bH), 2\bE - \bXbar(\bH+\bHbar) \rangle_F
  \quad {\scriptstyle \text{ thanks to } \langle a-b,a+b\rangle_F = \|a\|_F^2 - \|b\|_F^2}
  \\&=
  \langle (\bX-\bXbar)(\bH-\bHbar), 2\bE\rangle_F
  + \langle \bH-\bHbar, ({\bXbar{}^\top\bXbar - \bX^\top\bX})(\bH + \bHbar)\rangle_F
  \\&=
  \langle \bH-\bHbar, 2(\bX-\bXbar)^\top\bE\rangle_F
  + \langle \bH-\bHbar, \big[{ \bXbar{}^\top(\bXbar-\bX) + (\bXbar-\bX)^{\top}\bX}\big](\bH + \bHbar)\rangle_F.
    \end{align*}
    The second summand rewrites as $\langle\bXbar(\bH-\bHbar), (\bXbar-\bX)(\bH + \bHbar)\rangle_F + \langle(\bXbar-\bX)(\bH-\bHbar), \bX(\bH + \bHbar)\rangle_F.$
        By Cauchy-Schwarz and the submultiplicativity of the Frobenius norm, the second summand is bounded above by
        $$ \|\bXbar(\bH-\bHbar)\|_F \|(\bX-\bXbar)\bSigma^{-1/2}\|_F \|\bSigma^{1/2}(\bH + \bHbar)\|_F + \|(\bX-\bXbar)\bSigma^{-1/2}\|_F \|\bSigma^{1/2}(\bH - \bHbar)\|_F \|\bX(\bH + \bHbar)\|_F.$$
        Since $\bH-\bHbar = \hbB-\bBbar$ and $\bH+\bHbar = \hbB+\bBbar-2\bB^*$ have respectively at most $2\bar s$ and $2\bar s + s$ non-zero rows, using $\Omega_3$ twice gives the following bound on the second summand:
        $$ 
            2(1+\eta_4)\sqrt n
    	\|\bSigma^{1/2}(\hbB - \bBbar)\|_F
        \|(\bX-\bXbar)\bSigma^{-1/2}\|_F
        \|\bSigma^{1/2}(\hbB + \bBbar - 2\bB^*)\|_F.
        $$
        Combining the above displays, we find
        \begin{align*}
    &2(1-\eta_4)^2 n \|\bSigma^{1/2}(\hbB - \bBbar)\|_F^2
            \\\leq {}& 
    	2 \|\bSigma^{1/2}(\hbB - \bBbar)\|_F
        \|(\bX-\bXbar)\bSigma^{-1/2}\|_F
        \big(\|\bE\|_{op}
            +
        (1+\eta_4)\sqrt n
        \|\bSigma^{1/2}(\hbB + \bBbar - 2\bB^*)\|_F
        \big).
        \end{align*}

    Thanks to \Cref{lemma:risk}
    we have $\|\bSigma^{1/2}(\hbB - \bB^*)\|_F\le \bar R$
    and $\|\bSigma^{1/2}(\bBbar- \bB^*)\|_F\le \bar R$, this shows that
    $
    \|\bSigma^{1/2}(\hbB - \bBbar)\|_F
    \le
    {\|(\bX-\bXbar)\bSigma^{-1/2}\|_F
        \big(
        \|\bE\|_{op} n^{-1}
        +
        2 (1+\eta_4) n^{-1/2} \bar R
        \big)
        (1-\eta_4)^{-2}}
    $, hence 
    $${n^{1/2 }\|\bSigma^{1/2}(\hbB-\bBbar)\|_F
                \le
                2 (1+\eta_4)(1-\eta_4)^{-2}
                (\bar R + \|\bE\|_{op}n^{-1/2})
                \|(\bX-\bXbar)\bSigma^{-1/2}\|_F.}
            $$
    We also have by the triangle inequality
    \begin{align*}
     &   \|\bXbar(\bBbar-\bB^*) - \bX(\hbB-\bB^*)\|_F
    \\&\le
        \|\bXbar(\bBbar-\hbB)\|_F
        +
        \|(\bXbar-\bX)(\hbB - \bB^*)\|_F
    \\&\le
    (1+\eta_4) n^{1/2} \|\bSigma^{1/2}(\bBbar - \hbB)\|_F
    + 
    \|(\bX-\bXbar)\bSigma^{-1/2}\|_{{F}}
    \|\bSigma^{1/2}(\hbB - \bB^*)\|_F
    \\&\le
    {\|(\bX-\bXbar)\bSigma^{-1/2}\|_F
        \Big[
            2 (1+\eta_4)(1-\eta_4)^{-2}
                    (\bar R + \|\bE\|_{op}n^{-1/2})
        +
       \bar R
        \Big]}\\
    &\le 
            {4 (1+\eta_4)(1-\eta_4)^{-2}
            \big(\bar R + \|\bE\|_{op}n^{-1/2}\big)
            \|(\bX-\bXbar)\bSigma^{-1/2}\|_F
        ,}
    \end{align*}
    where the last line follows from the inequality $2 (1+\eta_4)(1-\eta_4)^{-2}\geq 2$ for $\eta_4\in (0,1)$.
\end{proof}

\lemmaDifferential*
\begin{proof}[Proof of \Cref{lemma:differential}]
By \Cref{lemma:lipschitz} and Rademacher's theorem, we know that
the Fréchet derivative of $\hbB$ with respect to $\bX$ exists almost everywhere,
so that $\mathsf{D}(\bb)$ exists for almost every $(\bX,\bE)\in\Omega_*$.
By \Cref{lemma:kkt-strict}, we also have that for almost every $(\bX,\bE)$,
the KKT conditions are strict in the sense given in \Cref{lemma:kkt-strict}.
In the following, we consider $(\bX,\bE)\in\Omega_*$ such that
$\mathsf{D}(\bb)$ exists and such that the KKT conditions are strict;
almost every $(\bX,\bE)\in\Omega_*$ satisfy these two conditions.

Since we know that the Jacobian $\mathsf{D}(\bb)$ exists by Rademacher's theorem,
it is enough to characterize its value, for instance by
computing the directional derivative in any fixed direction $\bw\in\R^n$.
To this end, for a real $u$ in a neighborhood of $0$,
let $\bX(u) = \bX + u\bw \ba^\top$
and $\bB(u) = \hbB(u\bw)$ . Define the active set
$\hat S(u)=\{j\in[p]: \|\bB(u)^\top\be_j\|_2 >0 \}$.
We also write $\dot\bX = (d/du) \bX \big|_{u=0} = \bw \ba^\top$,
and $\dot\bB = (d/du) \bB(u) \big|_{u=0}$.
At $0$, we have $\bX(0)=\bX$ and $\bB(0)=\hbB$ is the estimator computed
at $(\bX,\bY)$ with $\bY=\bX\bB^*+\bE$.

As in \eqref{eq:KKT-multitask},
the KKT conditions for $\bB(u)$
read, for $j\in\hat S(u)$ (i.e., $\be_j^\top\bB(u) \neq \mathbf 0$),
$$\be_j^\top\bX(u)^\top\big[\bE-\bX(u)(\bB(u)-\bB^*)\big] = \frac{nT \lambda}{\|\bB(u)^\top\be_j\|_2} \be_j^\top \bB(u)
\qquad \in\R^{1\times T}$$
and for $j\notin \hat S(u)$ (i.e., $\be_j^\top\bB(u) = \mathbf 0$),
$$\Vert \be_j^\top\bX(u)^\top\big[\bE-\bX(u)(\bB(u)-\bB^*)\big]\Vert_2 < nT \lambda.$$
By Lipschitz continuity of $u\mapsto \bB(u)$ established in \Cref{lemma:lipschitz},
the set $\hat S(u)$ is constant in a neighborhood of 0 because
the KKT conditions on $\hat S(u)^c$ are bounded away from $nT\lambda$
on a neighborhood of 0 by continuity, and because
the nonzero rows of $\bB(u)$ are bounded away from ${\mathbf 0}$ in a neighborhood of 
$0$ again by continuity of $\bB(u)$.
Differentiation of the above display for $j\in\hat S(u)$ at
$u=0$ and the product rule yield
$$\be_j^\top\Big[\dot\bX^\top(\bE-\bX(\hbB-\bB^*))
- \bX^\top(\dot\bX (\hbB -\bB^*) + \bX \dot\bB)
\Big]
= nT \be_j^\top \dot\bB \bH^{(j)}$$
with $\bH^{(j)}$ in \eqref{Hj}. Rearranging
and using $\dot\bX = \bw \ba^\top$, 
$$\be_j^\top\Big[\ba\bw^\top(\bE-\bX(\hbB-\bB^*))
- \bX^\top(\bw \ba^\top (\hbB -\bB^*))
\Big]
=  \be_j^\top \Big[nT \dot\bB \bH^{(j)} + \bX^\top\bX \dot\bB\Big]
\qquad \in\R^{1\times T}.$$
Let $\bP_{\hat S} = \sum_{j\in\hat S} \be_j\be_j^\top\in\R^{p\times p}$.
Multiplying by $\be_j$ to the left and summing over $j\in \hat S$, we obtain
$$\bP_{\hat S}\Big[\ba\bw^\top(\bE-\bX(\hbB-\bB^*))
- \bX^\top(\bw \ba^\top (\hbB -\bB^*))
\Big]
=  \bP_{\hat S}\Big[ nT \dot\bB \bH^{(j)} + \bX^\top\bX \dot\bB\Big]
\qquad \in\R^{p\times T}.$$
Since $\hat S(u)$ is locally constant for $u$ in a neighborhood of 0,
we have $\bP_{\hat S}\dot\bB = \dot\bB$
thus $\bX\dot\bB = \bX_{\hat S} \dot\bB$,
hence
$$\ba_{\hat S}\bw^\top(\bE-\bX(\hbB-\bB^*))
- \bX_{\hat S}^\top(\bw \ba^\top (\hbB -\bB^*))
= nT \Big[ \sum_{j\in \hat S}\be_j\be_j^\top \dot\bB \bH^{(j)}\Big] + \bX_{\hat S}^\top\bX_{\hat S} \dot\bB \bI_{T\times T}
\quad \in\R^{p\times T} .$$
We now use the relationship between vectorization
and Kronecker product \eqref{kronecker-vectorization-relation}.
Applying \eqref{trace-property-kronecker}
to the previous display for each term, we find
\begin{align*}
     &((\bY-\bX\hbB)^\top \otimes \ba_{\hat S} ) \vec(\bw^\top)
    - \Big(\big((\hbB-\bB^*)^\top \ba\big)\otimes \bX_{\hat S}^\top\Big)\vec(\bw)
    \\={}&\Big(
            \Big[nT \sum_{j\in \hat S} (\bH^{(j)} \otimes \be_j\be_j^\top)  \Big]
            + \bI_{T\times T}\otimes \bX_{\hat S}^\top\bX_{\hat S}
        \Big)\vec(\dot\bB)
        \\={}&\Big(\tbX{}^\top\tbX + nT \tbH \Big)\vec(\dot\bB).
\end{align*}
Since $\vec(\cdot)$ is always a column vector, $\vec(\bw^\top) = \vec(\bw) = \bw$.
Finally, we have again using  \eqref{kronecker-vectorization-relation} and the chain rule,
for any fixed $\bb\in\R^T$,
\begin{align*}
    \mathsf{D}(\bb)\bw &=\tfrac{d}{du} \bX(u)(\bB(u)-\bB^*)\bb~\big|_{u=0}\\
    &= \bw \ba^\top(\bB(0)-\bB^*)\bb
    + \bX \dot\bB \bb \\
    &= {\bw \ba^\top(\bB(0)-\bB^*)\bb
        + \bX_{\hat S} \dot\bB \bb}.
\end{align*}
By \Cref{lemma:rank}, $\rank(\bX_{\hat S})=|\hat S|$. 
The argument developed in the proof of \Cref{propMatrixA} (ii) shows that the nullspace of the matrix $\tbX{}^\top\tbX + nT \tbH$
is exactly the linear span of $\{\be_t\otimes \be_j,~ (j,t)\in\hat S^{c} \times [T]\}$.
Because $\bP_{\hat S}\dot\bB = \dot\bB$,
$\vec(\dot\bB)$ is in $\ker(\tbX{}^\top\tbX + nT \tbH)^{\perp} = \range(\tbX{}^\top\tbX + nT \tbH)$.
Since for any symmetric matrix $\bM$, $\bM^\dagger \bM$
is the orthogonal projection on the range of $\bM$, 
we have $
\big(nT \tbH  + \tbX{}^\top\tbX\big)^\dagger
\big(nT \tbH  + \tbX{}^\top\tbX\big)\vec(\dot\bB)
= \vec(\dot\bB).
$
Since $\bX_{{\hat S}} \dot\bB \bb$ is a column vector, using \eqref{kronecker-vectorization-relation}
again,
\begin{align*}
    \bX_{{\hat S}}\dot\bB\bb
    &=\vec(\bX_{{\hat S}}\dot\bB\bb)
  \\&= (\bb^\top \otimes \bX_{{\hat S}}) \vec(\dot \bB)
  \\&= (\bb^\top \otimes \bX_{{\hat S}})
      \big(\tbX{}^\top\tbX + nT \tbH \big)^\dagger
      \Big[
      ((\bY-\bX\hbB)^\top \otimes \ba_{\hat S} ) 
      - \Big(\big((\hbB-\bB^*)^\top \ba\big)\otimes \bX_{\hat S}^\top\Big)
     \Big]\bw.
\end{align*}
Since this holds for all $\bw$, this provides the desired
expression for $\mathsf{D}(\bb)$ for all $\bb$.
\end{proof}

\lemmaDivergence*
\begin{proof}[Proof of \Cref{lemma:divergence}]

For the first equality, 
$$
\trace[\mathsf{D^{*}}(\bb)] = \trace[(\ba^\top(\hbB-\bB^*)\bb )
            \bI_{n\times n}
            -
            (\bb^\top \otimes \bX_{{\hat S}})
            \big(\tbX{}^\top\tbX + nT \tbH \big)^\dagger
            \big(((\hbB-\bB^*)^\top \ba)\otimes \bX_{\hat S}^\top\big)
            ]
$$
and the conclusion follows from \eqref{eq:quadra-A}.

For \eqref{eq:divergence-D**-inequality}, the following bounds will be useful.
Inequality $\|\bX_{\hat S}\|_{op}^2 \le \|\bSigma\|_{op} (1+\eta_4)^2 n$
holds on $\Omega_*$.
Furthermore since $\ker \bN = \ker \bN^\dagger$
for all symmetric matrices $\bN$
and since $\tbH$ is positive semi-definite,
on $\Omega_*$ we find
\begin{align}
    \nonumber
\|(\tbX^\top\tbX +  nT\tbH)^\dagger\|_{op}^2
&= 
    \big[
        \min_{\bu\in\R^{np}:\bu\in \ker(\tbX^\top\tbX +  nT\tbH)^\perp}
        \bu^\top(\tbX^\top\tbX + nT\tbH) \bu
\big]^{-2}
\\&\le
    \big[
        \min_{\bu\in\R^{np}:\bu\in \ker(\tbX^\top\tbX +  nT\tbH)^\perp}
        \bu^\top(\tbX^\top\tbX) \bu
\big]^{-2}
    \nonumber
\\&\le {\phi_{\min}(\bSigma)^{-2} (1-\eta_4)^{-4} n^{-2}}
.
\label{eq:bound-operator-norm-big-dagger}
\end{align}

We now work on
$\sum_{t=1}^T \trace[\mathsf{D}^{**}(\be_t)]^2= \|\bv\|^2$,
the left hand side of \eqref{eq:divergence-D**-inequality}.
For brevity, define $\bM = 
(\tbX^\top\tbX +  nT\tbH)^\dagger
(\bI_{T\times T} \otimes \ba_{\hat S})
(\bY-\bX\hbB)^\top.$
Then if $\be_t\in\R^T$ and $\be_i\in\R^n$ denote canonical basis vectors,
$\sum_{t=1}^T \trace[\mathsf{D}^{**}(\be_t)]^2= \|\bv\|^2$
where $\bv\in\R^T$ has components $\bv_t= \trace[\mathsf{D}^{**}(\be_t)]$
so that
\begin{equation*}
\bv_t  =
\sum_{i=1}^n \be_i^\top[\mathsf{D}^{**}(\be_t) ]\be_i
= \sum_{i=1}^n \be_i^\top (\be_t^\top \otimes \bX_{\hat S}) \bM \be_i
= \be_t^\top \sum_{i=1}^n(\bI_{T\times T} \otimes (\be_i^\top\bX_{\hat S}) ) \bM\be_i,
\end{equation*}
where the last equality stems from two applications of
the mixed product property
\eqref{mixed-product-property}:
\begin{multline*}
\be_i^\top (\be_t^\top \otimes \bX_{\hat S}) 
= (1\otimes \be_i^\top)(\be_t^\top \otimes \bX_{\hat S})
= (\be_t^\top) \otimes (\be_i^\top \bX_{\hat S})
= (\be_t^\top \bI_{T\times T}) \otimes (1 (\be_i^\top \bX_{\hat S}))
\\= \be_t^\top  (\bI_{T\times T} \otimes (\be_i^\top \bX_{\hat S})).
\end{multline*}
Thus
    $\bv = \sum_{i=1}^n(\bI_{T\times T} \otimes (\be_i^\top\bX_{\hat S}) ) \bM\be_i$ and since
$\|\bv\|_2^2 = \bv^\top\bv = (\bv^\top \otimes 1) \bv$,
it follows that
$\|\bv\|^2 = 
\sum_{i=1}^n (\bv^\top \otimes (\be_i^\top\bX_{\hat S}) ) \bM\be_i
=
\trace[
(\bv^\top\otimes \bX_{\hat S}) \bM]
$ by \eqref{mixed-product-property}.

By the definition of $\bM$,
using the commutation property of the trace
we have
\begin{align*}
\|\bv\|_2^2&=
  \trace\bigl[
      (\bY-\bX\hbB)^\top
      (\bv^\top \otimes \bX_{\hat S})
      (\tbX^\top\tbX +  nT\tbH)^\dagger
      (\bI_{T\times T} \otimes \ba_{\hat S})
  \bigr].
\end{align*}
By the Cauchy-Schwarz inequality for $\langle\cdot,\cdot\rangle_{F}$
and
using $\|\bU\bV\|_F\le \|\bU\|_{op}\|\bV\|_F$ twice,
we find
\begin{align*}
\|\bv\|_2^2 
  &\le
\|
(\bY-\bX\hbB)^\top 
     (\bv^\top \otimes \bX_{\hat S})
\|_F
\|
      (\tbX^\top\tbX +  nT\tbH)^\dagger
      (\bI_{T\times T} \otimes \ba_{\hat S})
\|_F
\\&\le
\|
      \bY-\bX\hbB
\|_{op}
\| \bv^\top \otimes  \bX_{\hat S} \|_F
\|
      (\tbX^\top\tbX +  nT\tbH)^\dagger
\|_{op}
\|
      (\bI_{T\times T} \otimes \ba_{\hat S})
\|_F
\end{align*}
and the second factor equals
$\| \bv^\top \otimes  \bX_{\hat S} \|_F
= \|\bv\|_2 \|\bX_{\hat S} \|_F
$ by \eqref{norm-property-kronecker} for the Frobenius norm.

We introduce the notation $\lesssim$ to denote an inequality up to a constant that depends
on $\eta_1,...,\eta_4$ and $\phi_{\min}(\bSigma),\phi_{\max}(\bSigma)$ only. 
On $\Omega_*$ we have the operator norm bound \eqref{eq:bound-operator-norm-big-dagger},
the bound $\|\bX_{\hat S}\|_F \le |\hat S|^{1/2} \|\bX_{\hat S}\|_{op} \lesssim (|\hat S|n)^{1/2}$ as well as
$\|
      (\bI_{T\times T} \otimes \ba_{\hat S})
\|_F
= \sqrt T \| \ba_{\hat S}\|_2 \lesssim \sqrt T$ so that
$$
\|\bv\|_2
\lesssim
\|\bY-\bX\hbB\|_{op}
\sqrt{ns} n^{-1} \sqrt T
$$
and $
\|\bY-\bX\hbB\|_{op}\le \|\bE\|_{op} + \|\bX(\bB^*-\hbB)\|_F
\le
\sigma (\sqrt T + 2\sqrt n)
+ \sqrt n \bar R
$
thanks to $\Omega_4$ and \Cref{lemma:risk}.
Since $T\le n$ and $\bar R\lesssim 1$ under \Cref{assum:main}, we have proved that 
$\|\bv\|_2 \lesssim \sigma \sqrt{s T}$ holds on $\Omega_*$
which is exactly the desired bound \eqref{eq:divergence-D**-inequality}.

\end{proof}

\lemmaRemainderII*

\begin{proof}[Proof of \Cref{lemma:RemainderII}]

Recall that we assume the normalization $\|\bSigma^{-1/2}\ba\|^2=1$.
Following the notation in \cite{bellec2019biasing} we define the quantities:
 $$\bu_0 = \bSigma^{-1}\ba, \quad \bz_0=\bX\bu_0, \quad \bQ_0=\bI_{p\times p}-\bu_0\ba^\top.$$
We have the decomposition $\bX = \bX \bQ_0 + \bz_0 \ba^\top$,
the vector $\bz_0$ is independent of $\bX\bQ_0$ and $\bz_0$ has distribution $\mathcal N_n(\mathbf 0, \bI_{n\times n})$. 
Given a value of $(\bE,\bX\bQ_0)$, define the open set
$$U_0 = \{ \bz_0\in\R^n: (\bE,\bX\bQ_0 +\bz_0 {\ba^{\top}})\in\Omega_*\} \subset \R^n.$$
Since $\Omega_*$ is open , so is the set $U_0$.
Given a value of $(\bE,\bX\bQ_0)$ we also define the function
$U_0\to\R^{p\times T}$ given by
$$\hbB(\bz_0) = {\textstyle \argmin_{\bB\in\R^{p\times T}} }
\Big(
\tfrac{1}{2nT} \|\bE + (\bX\bQ_0 + \bz_0\ba^\top)(\bB^* - \bB)\|_F^2 + \lambda \|\bB\|_{2,1}
\Big)$$
as well as
$$\bF:U_0\to\R^{n\times T}, \qquad \bF:\bz_0\mapsto (\bX\bQ_0+\bz_0 {\ba^{\top}})(\hbB(\bz_0){ - \bB^*}).$$
Since $\bar R\to 0$ under \Cref{assum:main} and $\|\bE\|_{op}n^{-1/2}$
is bounded by an absolute constant on $\Omega_4$ when $T\le n$,
\Cref{lemma:lipschitz} shows that $\bF$ is $L$-Lipschitz
for some constant $L$ of the form $L=\sigma \C(\eta_1,...,\eta_4,\bSigma)$
where the constant depends only on $\eta_1,...,\eta_4$ and the minimal and
maximal eigenvalues of $\bSigma$.
By Kirszbraun's Theorem, there exists an $L$-Lipschitz function
$\tilde\bF:\R^n\to\R^{n\times T}$
which is an extension of $\bF$, i.e., it satisfies $\bF(\bz_0)=\tilde\bF(\bz_0)$ for all $\bz_0\in U_0$.
Since $\bF(\bz_0)$ is bounded from above by $n^{1/2}{(1-\eta_3)}\bar R$ in $U_0$ by
\Cref{lemma:risk}, we define the function $\bar\bF:\R^n\to\R^{n\times T}$ by
$$
\bar\bF(\bz_0) = \bPi \circ \tilde \bF(\bz_0)
$$
where $\bPi:\R^{n\times T}\to\R^{n\times T}$ is the convex projection onto the Frobenius ball of radius
$n^{1/2}\bar R$ in $\R^{n\times T}$. Since convex projections are 1-Lipschitz functions,
the function $\bar\bF$ is also an $L$-Lipschitz extension of $\bF$.

If $\bar{\mathsf{D}}(\bb)$ denotes the Jacobian such that
${\bar\bF(\bw)\bb - \bar\bF(\mathbf 0)\bb = \bar{\mathsf{D}}(\bb)\bw + o(\|\bw\|)}$  for all $\bb\in\R^T$,
then $\bar{\mathsf{D}}(\bb)=\mathsf{D}(\bb)$ on $U_0$
because two functions that coincide on an open set have the same gradient
on this open set.
This implies
\begin{align*}
\E\Big[
    I\{\Omega_*\}
\sum_{t=1}^T
\Big(
    \bz_0^\top\bF(\bz_0)\be_t
- \trace[\mathsf{D}(\be_t)]
\Big)^2
\Big]
&=
\E\Big[
    I\{\Omega_*\}
\sum_{t=1}^T
\Big(
    \bz_0^\top\bar\bF(\bz_0)\be_t
    - \trace[\bar{\mathsf{D}}(\be_t)]
\Big)^2
\Big]
\\&\le
\E\Big[
\sum_{t=1}^T
\Big(
    \bz_0^\top\bar\bF(\bz_0)\be_t
    - \trace[\bar{\mathsf{D}}(\be_t)]
\Big)^2
\Big]
\end{align*}
where the second display simply follows from $I\{\Omega_*\}\le1$.
By the main result of \cite{bellec_zhang2018second_stein}
 we find
 \begin{align*}
    \E\left[
    (
    \bz_0^\top\bar\bF(\bz_0)\be_t
    - \trace[\bar{\mathsf{D}}(\be_t)]
    )^2\right] 
    &=
    \E\left[
    \|\bar\bF(\bz_0)\be_t\|_2^2
    +
    \trace(
\{\bar{\mathsf{D}}(\be_t) \}^2
)
\right]
  \\&\le
    \E\left[
    \|\bar\bF(\bz_0)\be_t\|_2^2
    +
    \|\bar{\mathsf{D}}(\be_t) \|_F^2
    \right]
\end{align*}
for each $t=1,...,T$ since $\bz_0\sim \mathcal N_n({\mathbf 0},\bI_{n \times n})$.
Summing this inequality over $t=1,...,T$ yields
\begin{align*}
    &\frac{1}{n\sigma^2}\E\Big[
\sum_{t=1}^T
\Big(
    \bz_0^\top\bar\bF(\bz_0)\be_t
    - \trace[\bar{\mathsf{D}}(\be_t)]
\Big)^2
\Big]
\\&\le
\frac{1}{n\sigma^2}\E\Big[
    \|\bar\bF(\bz_0)\|_F^2
    +\sum_{t=1}^T
    \|\bar{\mathsf{D}}(\be_t)\|_F^2
\Big]
\\&\le \frac{\bar R^2}{\sigma^2}
+ \frac{1}{n\sigma^2}
\E\Big[
    \sum_{t=1}^T
    \|\bar{\mathsf{D}}(\be_t)\|_F^2
\Big]
\\&=
\frac{\bar R^2}{\sigma^2}
+
\frac{1}{n\sigma^2}
\E\Big[
    I\{\Omega_*\}
    \sum_{t=1}^T
    \|\bar{\mathsf{D}}(\be_t)\|_F^2
\Big]
+
\frac{1}{n\sigma^2}
\E\Big[
    I\{\Omega_*^c\}
    \sum_{t=1}^T
    \|\bar{\mathsf{D}}(\be_t)\|_F^2
\Big].
\end{align*}
Note that the first term, $\bar R^2/\sigma^2$, converges to 0,
as stated in \Cref{lemma:risk}.
We now bound the third term, on $\Omega_*^c$.
The quantity $
    \sum_{t=1}^T
    \|\bar{\mathsf{D}}(\be_t)\|_F^2
$ is exactly the squared Frobenius norm of the 
Jacobian of the map
$\bar\bF:\R^n\to\R^{n\times T}$ (this Jacobian has dimensions $(nT)\times n$
but we do not need to write it explicitly or choose a specific vectorization
of $\R^{n\times T}$ into $\R^{nT}$).
Since $\bar\bF$ is $L$-Lipschitz,
the operator norm of the Jacobian is at most $L$.
Since the rank of the Jacobian of a map from $\R^n$ to any other linear
space is at most $n$,
the rank of the Jacobian is at most $n$.
If follows from $\|\bJ\|_F^2\le \rank(\bJ) \|\bJ\|_{op}^2$ with $\bJ\in\R^{(nT)\times n}$ the Jacobian
of $\bar\bF$ that
$$
    \sum_{t=1}^T
    \|\bar{\mathsf{D}}(\be_t)\|_F^2
    =\|\bJ\|_F^2
    \le n L^2
$$
so that
$\frac{1}{n\sigma^2}\E[
    I\{\Omega_*^c\}
    \sum_{t=1}^T
    \|\bar{\mathsf{D}}(\be_t)\|_F^2
    ]\le \P(\Omega_*^c) L^2/\sigma^2
$ which converges to 0 under \Cref{assum:main} thanks to $\P(\Omega_*)\to 1$ in
\Cref{lemma:proba}.

It remains to show that
$\frac{1}{n\sigma^2}\E[
    I\{\Omega_*\}
    \sum_{t=1}^T
    \|\bar{\mathsf{D}}(\be_t)\|_F^2
    ]
$ converges to 0. This quantity is equal to
$\frac{1}{n\sigma^2}\E[
    I\{\Omega_*\}
    \sum_{t=1}^T
    \|{\mathsf{D}}(\be_t)\|_F^2
    ]
$ since the derivatives of $\bar\bF$ and $\bF$ coincide on $U_0$.
To bound this quantity, we use the explicit formulae obtained in \Cref{lemma:differential}
with
$\|{\mathsf{D}}(\be_t)\|_F^2
\le
2\|{\mathsf{D}^*}(\be_t)\|_F^2
+2\|{\mathsf{D}^{**}}(\be_t)\|_F^2$.
We can use the following property of Kronecker products.
If $\bM,\bQ$ are two matrices, and $\be_t$ is the $t$-th canonical basis
vector in $\R^T$, then by the mixed product property \eqref{mixed-product-property}
\begin{align}
\sum_{t=1}^T \|(\be_t^\top \otimes \bM )\bQ\|_F^2
&=
\sum_{t=1}^T\trace[ \bQ^\top (\be_t\otimes \bM^\top)(\be_t^\top \otimes \bM )\bQ]
\cr&=
\trace[ \bQ^\top \sum_{t=1}^T\Big[(\be_t\otimes \bM^\top)(\be_t^\top \otimes \bM )\Big]\bQ]
\cr&=
\trace[ \bQ^\top (\bI_{T\times T}\otimes \bM^\top\bM)\Big]\bQ]
\cr&=
\| (\bI_{T\times T} \otimes \bM) \bQ\|_F^2.
\label{eq:sum-over-T-helper}
\end{align}
Since $\|{\mathsf{D}^*}(\be_t)\|_F^2
            \leq 
            2(\ba^\top(\hbB-\bB^*)\be_t)^2 \|\bI_{n\times n}\|_F^2 
            +
             2 \|(\be_t^\top \otimes \bX_{{\hat S}})
            \big(\tbX{}^\top\tbX + nT \tbH \big)^\dagger
            \big(((\hbB-\bB^*)^\top \ba)\otimes \bX_{\hat S}^\top\big) \|_F^2$,
thanks to \eqref{eq:sum-over-T-helper}
with $\bM=\bX_{\hat S}$ and $\bQ=
\big(\tbX{}^\top\tbX + nT \tbH \big)^\dagger
                \big(((\hbB-\bB^*)^\top \ba)\otimes \bX_{\hat S}^\top\big)
$ for the second term we find
    $$\sum_{t=1}^T \|{\mathsf{D}^*}(\be_t)\|_F^2 
    \leq 
    2n \|(\hbB-\bB^*)^{\top}\ba\|_2^2 
    +
    2 \|(\bI_{T\times T} \otimes \bX_{{\hat S}}) \big(\tbX{}^\top\tbX + nT \tbH \big)^\dagger
                \big(((\hbB-\bB^*)^\top \ba)\otimes \bX_{\hat S}^\top\big)\|_F^2.$$
The first summand is bounded by 
$2n\|(\hbB-\bB^*)\|_F^2 \;\|\ba\|_2^2 
\leq 2n \phi_{\min}(\bSigma)^{-1/2}\bar R \; \phi_{\max}(\bSigma)$ 
and the second summand by
\begin{align*}
    &\stackrel{\text{(i)}}{\leq}
    2 \|(\bI_{T\times T} \otimes \bX_{{\hat S}})\|_{op}^2 \;
      \|\big(\tbX{}^\top\tbX + nT \tbH \big)^\dagger\|_{op}^2 \;
      \|((\hbB-\bB^*)^\top \ba)\otimes \bX_{\hat S}^\top\|_F^2 \\
    &\stackrel{\text{(ii)}}{\leq}
    2 \|\bX_{{\hat S}}\|_{op}^2 \;
      \|\big(\tbX{}^\top\tbX + nT \tbH \big)^\dagger\|_{op}^2 \;
      \|(\hbB-\bB^*)^\top \ba\|_F^2 \; \|\bX_{\hat S}^\top\|_F^2\\
    &\stackrel{\phantom{\text{(iii)}}}{\leq} 
    2 \|\bX_{{\hat S}}\|_{op}^2 \;
      \|\big(\tbX{}^\top\tbX + nT \tbH \big)^\dagger\|_{op}^2 \;
      \|(\hbB-\bB^*)\|_F^2 \;\|\ba\|_2^2  \; \rank(\bX_{\hat S})\|\bX_{\hat S}\|_{op}^2\\
    &\stackrel{\text{(iii)}}{\leq} 
    2 (\phi_{\max}(\bSigma)(1+\eta_4)^2 n)^2 
    (\phi_{\min}(\bSigma)^{-2} (1-\eta_4)^{-4} n^{-2})
    (\phi_{\min}(\bSigma)^{-1/2}\bar R \; \phi_{\max}(\bSigma))
    \bar s\\
    &=
    2 \phi_{\max}(\bSigma)^{3} \phi_{\min}(\bSigma)^{-5/2} \bar s \bar R.
\end{align*}
Above, 
(i) follows from $\|\bM\bN\bU\|_F \le \|\bM\|_{op} \|\bN\|_{op} \|\bU\|_F$,
(ii) is a consequence of \eqref{norm-property-kronecker} and (iii) holds on $\Omega_*$.
Thus $\sum_{t=1}^T \|{\mathsf{D}^*}(\be_t)\|_F^2 \lesssim n\bar R$.

Likewise,
\begin{align*}
    \sum_{t=1}^T \|{\mathsf{D}^{**}}(\be_t)\|_F^2 
    &\leq \|(\bI_{T\times T} \otimes \bX_{{\hat S}})
              \big(\tbX{}^\top\tbX + nT \tbH \big)^\dagger
              ((\bY-\bX\hbB)^\top \otimes \ba_{\hat S} ) \|_F^2\\
    &\leq (\phi_{\max}(\bSigma)(1+\eta_4)^2 n) 
         (\phi_{\min}(\bSigma)^{-2} (1-\eta_4)^{-4} n^{-2})
         (8\sigma^2nT + 2(1-\eta_3)^2n\bar R^2)
         \phi_{\max}(\bSigma)\\
    &\lesssim \sigma^2 T
\end{align*}
Thus
$\frac{1}{n\sigma^2}\E[
    I\{\Omega_*\}
    \sum_{t=1}^T
    \|\bar{\mathsf{D}}(\be_t)\|_F^2
    ]\lesssim \frac{\bar R}{\sigma^2} + \frac Tn$ and the right hand side converges to $0$ under \Cref{assum:main}. 
\end{proof}

\section{Proof that \texorpdfstring{$\P(\Omega_*)\to 1$}{P(Omega*) goes to 0}}
\label{sec:probabliistic-propositions}

\subsection{$\Omega_1$: Restricted Eigenvalues for random matrices in multi-task learning}

\begin{proposition}
 \label{prop:restricted-eig}
    Let $\bG\in\R^{n\times p}$ be a random matrix with i.i.d. $\mathcal N(0,1)$ entries
    and let $\mathcal A$ be a subset of $\R^{p\times T}$ with $\|\bB\|_F=1$ for all
    $\bB\in\mathcal A$.

    \begin{enumerate}
        \item
            For any two $\bA,\bB\in\mathcal A$,
            $\P(|~\|\bG\bA\|_F - \|\bG\bB\|_F|\ge \C \sqrt{x}\|\bB-\bA\|_F)
            \le 6 e^{{- x}}$ for all $x>0$. 
        \item 
            $\sup_{\bA,\bB\in\mathcal A}|~\|\bG\bA\|_F - \|\bG\bB\|_F| \le \C \E
     \sup_{\bB\in \mathcal A} | \trace[\bB^\top \bG']| + \C \sqrt{x}$
     with probability at least $1-e^{-x}$,
     where $\bG'\in\R^{p\times T}$ has i.i.d. $\mathcal N(0,1)$ entries.

     $\sup_{\bA\in\mathcal A}|~\|\bG\bA\|_F - \sqrt n| \le \C \E
     \sup_{\bB\in \mathcal A} | \trace[\bB^\top \bG']| + \C \sqrt{x}$ 
     also holds with probability at least $1-{3}e^{-x}$.
        \item
            If $\bX$ has i.i.d. $\mathcal N_p({\mathbf 0},\bSigma)$ rows with $\max_{j\in[p]}\bSigma_{jj} \le 1$
            and 
            \begin{equation}
                \label{def-cone-C}
                \mathcal C =  \{\bA\in\R^{p\times T}: \|\bA\|_{2,1} \le \sqrt{k} \|\bA\|_F \},
            \end{equation}
            then with probability at least $1-{3} e^{-x}$,
            \begin{align*}
        &
                \sup_{\bA\in\mathcal C: \|\bSigma^{1/2}\bA\|_F=1}
        \Big| n^{-1/2}\|\bX\bA\|_F - 1\Big|
        =
                \sup_{
                \bB\in\R^{p\times T}:\bSigma^{-1/2}\bB\in \mathcal C,
            \|\bB\|_F=1}
            \Big| n^{-1/2}\|\bX\bSigma^{-1/2}\bB\|_F - 1\Big|
        \\&
        \le \C \sqrt{x / n}
        + \C n^{-1/2} \E \sup_{\bB\in\R^{p\times T}:\bSigma^{-1/2}\bB\in\mathcal C, \|\bB\|_F=1}|\trace[\bB^\top\bG']|
      \\&\le
      \C \sqrt{x / n }
      + \C
      \sqrt{{[ k T +  k \log(p/k)} ]/({\phi_{\min}(\bSigma)n} )}
            \end{align*}
            This implies that for any constant $\eta_3\in(0,1)$,
            if $\{kT+k\log(p/k)\}/(n\phi_{\min}(\bSigma))\to 0$
            then
            $\P(\max_{\bA\in\mathcal C: \|\bSigma^{1/2}\bA\|_F=1}
            \big| n^{-1/2} \|\bX\bA\|_F - 1|\le \eta_3)\to 1$.
    \end{enumerate}
\end{proposition}
The proof follows the argument from \cite{liaw2017simple}, adapted to the multi-task setting.
\begin{proof}[Proof of (i)]
    We distinguish two cases.

    \emph{Case (a): $\sqrt{xn} > n/4$}.
    In this case we use that
        $$\|\bG\bA\|_F-\|\bG\bB\|_F
        \le \|\bG(\bA-\bB)\|_F
        = \Big(\sum_{i=1}^n \|(\bA-\bB)^\top \bG^\top \be_i\|_2^2 \Big)^{1/2}
        $$ 
        and we apply \cite[Theorem 6.3.2]{vershynin2018high} 
        to the vector $\vec(\bG^\top) \in \mathbb R^{np\times 1}$ and 
        the block diagonal matrix with $n$ blocks, each block being $(\bA-\bB)^\top$.
        This yields 
        $$\P(|~\|\bG\bA\|_F - \|\bG\bB\|_F|\ge \sqrt{x}\|\bB-\bA\|_{op} 
                                                  + \sqrt n  \|\bB-\bA\|_{F})
                    \le 2 e^{-\C x}.$$
    Here, $\sqrt{n} \le 4\sqrt{x}$
    and we can bound from above the first
    term to obtain the desired bound.

    \emph{Case (b): $\sqrt{xn} \le n/4$}.
    Write
    $\|\bG\bA\|_F - \|\bG\bB\|_F = \frac{\|\bG\bA\|_F^2 - \|\bG\bB\|_F^2}{\|\bG\bA\|_F + \|\bG\bB\|_F}$.
    We will use repeatedly the following concentration bounds:
        if $\bz\sim \mathcal N_q(\mathbf{0},\bI_{q\times q})$ and $\bM\in\R^{q\times q}$
        is symmetric positive semi-definite, then
        \begin{equation}
            \label{eq:concentration-quadratic-lower-bound}
        \P\left(\bz^\top\bM\bz-\trace\bM < 2\sqrt{x}\|\bM\|_F \right)
        \le e^{-x}.
        \end{equation}
        This is a straightforward consequence
        of \cite[Lemma 1]{laurent2000adaptive}
        after diagonalizing the symmetric positive semi-definite matrix $\bM$.
        Furthermore, for any $\bM\in\R^{q\times q}$,
        \begin{equation}
            \label{eq:concentration-quadratic-upper-bound}
        \P\left(\bz^\top\bM\bz-\trace\bM > 2\sqrt{x}\|\bM\|_F + 2 x \|\bM\|_{op}\right)
        \le e^{-x}
        \end{equation}
        see for instance
        \cite[Example 2.12]{boucheron2013concentration} or
        \cite[Lemma 3.1]{bellec2018optimal}.
    
    If $\bg_1^\top,...,\bg_n^\top$ are the rows of $\bG$ then
    $\|\bG\bA\|_F^2
    = \sum_{i=1}^n \bg_i^\top \bA \bA^{\top} \bg_i$
    is of the above form with $q=np$ and $\bM$ is
    block diagonal with $n$ blocks equal to $\bA\bA^\top\in\R^{p\times p}$.
    Thus
    $\|\bG\bA\|_F^2
    \ge n \|\bA\|_F^2 - 2 \sqrt{x n} \|\bA\bA^\top\|_F
    \ge n - 2 \sqrt{xn}
    $ with probability at least $1-e^{-x}$ by \eqref{eq:concentration-quadratic-lower-bound} and thanks to $\|\bA\|_F=1$.
    The same holds for a lower bound
    on $\|\bG\bB\|_F^2$.
    For the numerator, thanks to \eqref{eq:concentration-quadratic-upper-bound},
    with probability at least $1-e^{-x}$:
    \begin{align*}
    \|\bG\bA\|_F^2 - \|\bG\bB\|_F^2
    &=\sum_{i=1}^n \bg_i^\top (\bA-\bB)(\bA+\bB)^\top\bg_i
  \\&\le
  2\sqrt{xn} \|(\bA-\bB)(\bA+\bB)^\top\|_F
    + 2x \|(\bA-\bB)(\bA+\bB)^\top\|_{op}.
    \end{align*}
    By the union bound,
    with probability at least $1-3e^{-x}$,
    \begin{equation*}
        \|\bG\bA\|_F - \|\bG\bB\|_F 
        \le
        \frac{
            2\sqrt{xn} \|(\bA-\bB)(\bA+\bB)^\top\|_F
    + 2x \|(\bA-\bB)(\bA+\bB)^\top\|_{op}
        }{
        2(n-2\sqrt{xn})_+^{1/2}
    }.
    \end{equation*}
    Since here $\sqrt{xn} \le n/4$,
    the denominator is at
    least $2(n/2)^{1/2}$
    and using the submultiplicativity of the Frobenius norm
    with $\|\bA+\bB\|_F\le 2$ for the numerator we find
    $$\frac{\|\bG\bA\|_F - \|\bG\bB\|_F}{\|\bA-\bB\|_F}
    \le
    2\frac{\sqrt{xn} + x}{(n/2)^{1/2}}
    \le
    \C \sqrt x.
    $$
\end{proof}

\begin{proof}[Proof of (ii)]
    Since (i) proves that the process $Z_{\bA} = \|\bG\bA\|_F$
    has subgaussian increment with respect to the Frobenius norm,
    (ii) follows by Talagrand Majorizing Measure theorem, for example as stated in \cite[Theorem 4.1]{liaw2017simple}.

       The second statement follows by taking a fixed $\bB\in\mathcal A$
    and using $|\sqrt n -\|\bG{\bB}\|_F|\le \C \sqrt{x}$ with probability
    at least $1-2e^{-x}$
    by \cite[Theorem 6.3.2]{vershynin2018high} applied to the block diagonal matrix with $n$ blocks,
    each block being $\bB^\top$.
\end{proof}

\begin{proof}[Proof of (iii)]
    Recall that $\bG'\in\R^{p\times T}$ has i.i.d. $\mathcal N(0,1)$ entries.
    By application of (ii), it is sufficient to control the Gaussian
    width
    \begin{equation}
    \E \sup_{\bB\in\R^{p\times T}:\bSigma^{-1/2}\bB\in\mathcal C, \|\bB\|_F=1}|\trace[\bB^\top\bG']|
    =
    \E \sup_{\bA\in\mathcal C:\|\bSigma^{1/2}\bA\|_F=1}|\trace[\bA^\top\bSigma^{1/2}\bG']|.
    \label{eq:gaussian-width-to-bound}
    \end{equation}
    Let $\bA\in\mathcal C$ and
    let $\bg_1^{\top},...,\bg_p^{\top}$ be the rows of $\bSigma^{1/2}\bG'$.
    For any fixed $j\in[p]$, the random vector $\bg_j\in\R^{T\times 1}$
    has $\mathcal N_T(\mathbf{0}_{1\times T},\bSigma_{jj}\bI_{T\times T})$ distribution.
    By the triangle inequality and the Cauchy-Schwarz inequality we have
    for some $m,t >0$
    \begin{align*}
        |\trace[\bA^\top\bSigma^{1/2}\bG']|
        &\le
        \sum_{j=1}^p
        \|\bA^{{\top}} \be_j\|_2 \|\bg_j\|_2
        =
        \|\bA\|_{2,1}(m+t)
        +
        \sum_{j=1}^p
        \|\bA^{{\top}} \be_j\|_2(\|\bg_j\|_2-m-t)
        \\
        &\le
        \|\bA\|_F \sqrt k (m+t)
        +
        \|\bA\|_F
        \Big(
            \sum_{j=1}^p
            (\|\bg_j\|_2-m-t)_+^2
        \Big)^{1/2}
    \end{align*}
    where for the second line we used that $\bA\in\mathcal C$.
    We have $\|\bA\|_F\le \|\bSigma^{-1/2}\|_{op}$
    if $\|\bSigma^{1/2}\bA\|_F=1$.
    Next, we now define $m$ such that 
    $m^2$ is the median of the $\chi^2_T$ distribution,
    and $t=\sqrt{2\log (p/k)}$.
    As explained in the proof of \Cref{prop:noise} around \eqref{ineq-E-noise-event}
    we have $m\le \sqrt T$ \cite{stack_exchange3590470median_chi_square}
    as well as
    $
    \E
            \sum_{j=1}^p
            (\|\bg_j\|_2-m-t)_+^2
    \le k$. 
        By the inequality $\sqrt a + \sqrt b \leq \sqrt {2(a+b)}$, 
        \eqref{eq:gaussian-width-to-bound}
        is bounded from above 
        by $\|\bSigma^{-1/2}\|_{op}
        (
        \sqrt{{2} k(T + 2\log(p/k))}
        + \sqrt k)
        \leq
         \|\bSigma^{-1/2}\|_{op} \sqrt{8 k(T +\log(p/k))}$
    and the proof is complete.

\end{proof}

\subsection{$\Omega_2$: Control of the noise}

\begin{proposition}
    \label{prop:noise}
    Let $a_+=\max(0,a)$.
    If $\bE\in\R^{n\times T}$ has i.i.d. $\mathcal N(0,\sigma^2)$ entries
    and $\bX\in\R^{n\times p}$ has i.i.d. $\mathcal N_p({\mathbf 0},\bSigma)$ rows independent of $\bE$,
    then
    \begin{align}
    \sum_{j=1}^p \Big(
    \frac{\|\bE^\top\bX\be_j\|_2}{\sigma(1+\eta_1)\sqrt{n\bSigma_{jj}}} - \sqrt{T} - \sqrt{2\log(p/s)} 
        \Big)_+^2
         &\le
     \sum_{j=1}^p \Big(
         \frac{\|\bE^\top\bX\be_j\|_2}{\sigma\|\bX\be_j\|_2} - \sqrt{T} -\sqrt{2\log(p/s)}
     \Big)^2
    \nonumber
         \\&\le s
         \label{eq:prop-noise-conclusion}
    \end{align}
     with probability at least $1-{4}/\{(2\log(p/s)+2)(4\pi\log(p/s)+4)^{1/2}\}-p e^{-n\eta_1^2/2}$.
    Consequently, on the same event with
    $$
    \definitionOfLambdaZero 
    $$
    we have
    $\sum_{j=1}^p (\|\bE^\top\bX\be_j\|_2 - nT \lambda_0)_+^2
    \le
    \sigma^2(1+\eta_1)^2n\max_j\bSigma_{jj}
    s
    \le
    s n^2 T \lambda_0^2
    $.
\end{proposition}
\begin{proof}
    Since $\bX\be_j$ has i.i.d. $\mathcal N(0,\bSigma_{jj})$ entries,
    $\P(\|\bX\be_j\|_2 \ge \bSigma_{jj}^{1/2}(\sqrt{n} + t))\le e^{-t^2/2}$
    holds by standard bounds on $\chi^2_n$ random variables, e.g.,
    as a consequence of \cite[Theorem 5.5]{boucheron2013concentration}.
    The choice $t=\eta_1\sqrt{n}$ and 
    the union bound over $\{1,...,p\}$ provides the first inequality in
    \eqref{eq:prop-noise-conclusion}.

    Since $\bE$ is independent of $\bX$,
    conditionally on $\bX$ the random variable $\bg_j\coloneqq
    \bE^\top\bX\be_j /(\sigma\|\bX\be_j\|_2)$
    has standard normal distribution $\mathcal N_T({\mathbf 0}, \bI_{T\times T})$.
    Since the conditional distribution does not depend on $\bX$,
    the unconditional distribution of $\bg_j$ is also $\mathcal N_T({\mathbf 0}, \bI_{T\times T})$.
    By \cite[Theorem 10.17]{boucheron2013concentration} applied
    to the 1-Lipschitz function $\bg_j \mapsto \|\bg_j\|_2$,
    inequality
    $\P(\|\bg_j\|_2 \ge m_j + t) \le \P(Z_j\ge t)$ holds,
    where $Z_j\sim \mathcal N(0,1)$ and $m_j$ is the median of the random variable
    $\|\bg_j\|_2$. It follows that for any $t>0$
    \begin{equation}
        \label{ineq-E-noise-event}
        W\coloneqq
    \sum_{j=1}^p ( \|\bg_j\|_2 - m_j - t )_+^2
    \quad\text{ satisfies }\quad
    \E[W]
    \le
    \E\sum_{j=1}^p (Z_j - t )_+^2
    \le \frac{4p e^{-t^2/2}}{(t^2+2)(2\pi t^2 + 4)^{1/2}},
    \end{equation}
    where the second inequality follows from \cite[Lemma G.1]{bellec_zhang2018second_stein}.
    By the argument in \cite{stack_exchange3590470median_chi_square},
    the median of the $\chi^2_T$ distribution is smaller than $T$ 
    so that $m_j\le \sqrt{T}$. Furthermore,
    for $t=(2\log(p/s))^{1/2}$ we have $\E[W] \le sq$
    where $q^{-1}=(t^2+2)(2\pi t^2+4)^{1/2}/4 > 1$. The second inequality in
    \eqref{eq:prop-noise-conclusion} thus holds with probability at least $1-q$
    by Markov's inequality $\P(W>\E[W]q^{-1})\le q$.
\end{proof}

\subsection{$\Omega_3$: Restricted Isometry Properties}

The following bound is well known in the literature on the RIP property
for Gaussian matrices, as a consequence of Gordon's Lemma, see, 
e.g., \cite{ZhangH08}.
We provide the argument here for completeness.

\begin{proposition}[Bound on upper sparse eigenvalues of random matrices, Gordon's lemma]
    \label{prop:gordon}
    Let $p\ge n$. If $\bX\in\R^{n\times p}$ has i.i.d. $\mathcal N(0,\bSigma)$ rows, then

    (i) for any set $B\subset [p]$ we have
    $$\P\Big(\max_{\bv\in\R^p:\supp (\bv)\subset B}\Big|\frac{\|\bX\bv\|}{\sqrt n\|\bSigma^{1/2}\bv\|} - 1\Big|\le \sqrt{|B|/n} + t \Big)\ge 1- 2e^{- n t^2/2}$$
    by Gordon's escape through the mesh theorem
    and its consequence,
    cf. for instance in \cite[Theorem II.13]{DavidsonS01}
    applied to the Gaussian matrix $\bX\bSigma^{-1/2}$
    and the intersection of the unit ball with the $|B|$ dimensional 
    linear span of $\{\bSigma^{1/2}\be_j, j\in B \}$.

    (ii) Let $\eta_4\in (0,1)$ be a constant. 
    If $k$ is such that ${\sqrt{k/n}} \le \eta_4/2$
    and
    $k\log(ep/k)/n \le \eta_4^2/16$, then
    simultaneously for all $B$ with $|B|\le k$
    $$
    \P\left(
        \max_{B\subset[p]:|B|\le k}
        \Big(\max_{\bv\in\R^p:\supp(\bv)\subset B}\Big|\frac{\|\bX\bv\|}{\sqrt n\|\bSigma^{1/2}\bv\|} - 1\Big|
\Big)>
    \eta_4
    \right)\le
    2 \exp(-n \eta_4^2/16).
    $$
\end{proposition}
\begin{proof}
    For (ii),
    by the union bound with $t=\eta_4/2$ we have
    $$
    \P\left(
    \max_{B\subset[p]:|B|\le k}
    \Big(
        \max_{\bv\in\R^p:\supp(\bv)\subset B}\Big|\frac{\|\bX\bv\|}{\sqrt n\|\bSigma^{1/2}\bv\|} - 1\Big|
    \Big)
    > \eta_4 \right)\le
    2
    \binom{ p }{ {k} } e^{-n\eta_4^2/8}.
    $$
    Since $\log \binom{p}{k} \le k \log(e{p}/k)$, the right hand side
    is bounded from above by $2\exp(- n \eta_4^2/16)$ by assumption on $k$.
\end{proof}

\section{Proof of Theorems~\NoCaseChange{\ref{thm-normal} and \ref{thm:variance}}}
\label{sec:proof-thm-normal}



\begin{proof}
    By replacing $\bb$ by $\bb/\|\bb\|_2$ if necessary,
    we assume that $\|\bb\|_2=1$ without loss of generality.
    The proof is based on the decomposition 
    \begin{align*}
    &
        {(n\sigma^2)^{-1/2}}
        \big(n\ba^T (\hbB - \bB^*) \bb + \bz_0^T(\bY-\bX\hbB)(\bI_{T\times T}-\hbA/n)^{-1}\bb\big)
        \\&=
        {(n\sigma^2)^{-1/2}}
        {\bz_0^T \bE \bb}
        + \br^\top\bb + \tbr^\top\bb
    \end{align*}
    with the remainder terms 
    $\br^\top\bb$ and $\tbr^\top\bb$
    defined by the random vectors $\br,\tbr\in\R^T$
    \begin{align*}
        \br^\top
    &= 
    (n\sigma^2)^{-1/2} \bz_0^T \bE \left[(\bI_{T\times T}-\hbA/n)^{-1}-(\bI_{T\times T}) \right], \\
        \tbr^\top& = 
        {(n\sigma^2)^{-1/2}}
        \left[\ba^T(\hbB - \bB^*)(n\bI_{T\times T}-\hbA) - \bz_0^T\bX(\hbB - \bB^*)\right](\bI_{T\times T}-\hbA/n)^{-1}.
    \end{align*}
    Since $\bz_0\sim \mathcal N_n({\mathbf 0},\bI_{n\times n})$ is independent of $\bE\bb\sim \mathcal N_n(\mathbf 0, \sigma^2\bI_{n\times n})$ we have $\bz_0^\top\bE\bb/\|\bz_0\|_2 \sim \mathcal N(0, \sigma^2)$.
    Since $\|\bz_0\|_2^2n^{-1}\smash{\xrightarrow[]{\P}} 1$ by the law of large  numbers,
    we obtain that $(n\sigma^2)^{-1/2}\bz_0^\top\bE\bb\smash{\xrightarrow[]{d}} \mathcal N(0, 1)$
    by Slutsky's theorem.
    To conclude with another application of Slutsky's theorem,
    it remains to prove that $\|\br\|_2$ and $\|\tbr\|_2$ both converge to 0
    in probability, and to prove
    that for the denominator, $(n\sigma^2)^{-1/2}
    \|(\bY-\bX\hbB) (\bI_{T\times T}-\hbA/n)^{-1} \bb\|_2\smash{\xrightarrow[]{\P}} 1$.

    For $\br$,
    on $\Omega_*$ we have
    $\|(\bI_{T\times T}-\hbA/n)^{-1}-(\bI_{T\times T})\|_{op} \le
    \bar s /(n-\bar s)
    $
    by \Cref{propMatrixA}(iii) and \Cref{lemma:sparsity}.
    It follows that
    \begin{align*}
        \E[\min(1,\|\br\|_2)]  
        & \le
        \P(\Omega_*^c)
        +
        \E\big[I\{\Omega_*\}(n\sigma^2)^{-1/2} \|\bE^\top\bz_0\|_2 
            \|(\bI_{T\times T}-\hbA/n)^{-1}-(\bI_{T\times T})\|_{op}
        \big]
      \\&\le
        \P(\Omega_*^c)
        + (n\sigma^2)^{-1/2}
        \big(\bar s/(n-\bar s)\big)
        \E\big[\|\bE^\top\bz_0\|_2 
        \big]
      \\&\le
        \P(\Omega_*^c)
        + (n\sigma^2)^{-1/2}
        \big(\bar s/(n-\bar s)\big)
        \sqrt{nT\sigma^2}
      \\&=
        \P(\Omega_*^c)
        + 
        \big(\bar s/(n-\bar s)\big)
        \sqrt{T}
    \end{align*}
    by Jensen's inequality and $\E[\|\bE^\top\bz_0\|_2^2]=nT\sigma^2$.
    The last line converges to 0 by \Cref{lemma:proba}
    and \Cref{assum:main}.
    Since $W_n\smash{\xrightarrow[]{\P}}0$ if and only if $\E[\min(1,|W_n|)]\to 0$,
    this proves the convergence $\|\br\|_2\smash{\xrightarrow[]{\P}} 0$.

For $\tbr$, we use that
$$\E[\min(1,\|\tbr\|_2)] \le \P(\Omega_*^c) + \E[I\{\Omega_*\}\|\tbr\|_2]$$
with $\P(\Omega_*^c)\to0$ as above. For the second term,
on $\Omega_*$ we have
$\|(\bI_{T\times T}-\hbA/n)^{-1} \|_{op}
\le
\|\bI_{T\times T}\|_{op} + \|\bI_{T\times T}-(\bI_{T\times T}-\hbA/n)^{-1} \|_{op}
\le 1 + \bar s /(n-\bar s)=(1-\bar s/n)^{-1}$ by \Cref{propMatrixA} and \Cref{lemma:sparsity}. It follows that
\begin{align*}
    I\{\Omega_*\}\|\tbr\|_2
    &\leq 
    I\{\Omega_*\} \tfrac{1}{\sigma\sqrt n}(1-\tfrac{\bar s}{n})^{-1} 
            \|(n\bI_{T\times T}-\hbA)(\hbB - \bB^*)^{\top} \ba 
            - (\hbB - \bB^*)^{\top} \bX^{\top} \bz_0 \|_2
            \\  
    &=
    I\{\Omega_*\} \tfrac{1}{\sigma\sqrt n}(1-\tfrac{\bar s}{n})^{-1} 
    \Big[\sum_{t=1}^T \Big(\trace[{\mathsf{D^*}}(\be_t)] - \bz_0^{\top} \bX (\hbB - \bB^*)  \be_t \Big)^2\Big]^{1/2}.
\end{align*}
where the equality is a consequence of \Cref{lemma:divergence}.
Since $\mathsf{D^*}=\mathsf{D}-\mathsf{D^{**}}$, 
and using the inequalities $(a+b)^2\leq 2a^2 + 2b^2$ and $\sqrt{a+b}\leq \sqrt a + \sqrt b$,
\begin{align*}
    &\phantom{{}\leq}
    \E\Big[I\{\Omega_*\}  
           \sum_{t=1}^T \Big(\bz_0^{\top}(\hbB - \bB^*) \bX \be_t - \trace[{\mathsf{D^*}}(\be_t)] \Big)^2  
      \Big]^{1/2}\\
    &\leq 
    \Big[ 
        2\E\Big(I\{\Omega_*\}  \sum_{t=1}^T \big[\bz_0^{\top}(\hbB - \bB^*) \bX \be_t - \trace[{\mathsf{D}}(\be_t)]\big]^2 \Big)
        +
        2\E\Big(I\{\Omega_*\}  \sum_{t=1}^T \trace[{\mathsf{D}^{**}}(\be_t)]^2\Big)\Big]^{1/2}\\
    &\leq o((n\sigma^2)^{1/2}) + O\big(\sigma \min(T,(sT)^{1/2}) \big)
\end{align*}
by \Cref{lemma:RemainderII} and
inequality \eqref{eq:divergence-D**-inequality} in \Cref{lemma:divergence}.
Combining the above displays yields
\begin{align*}
    \E[\min(1,\|\tbr\|_2)] &\leq 
    \P(\Omega_*^c) + (n\sigma^2)^{-1/2} (1-\tfrac{\bar s}{n})^{-1}
    \big[o((n\sigma^2)^{1/2}) + O\big(\sigma (sT)^{1/2}\big) \big]
    = o(1),
\end{align*}
or equivalently $\|\tbr\|_2\smash{\xrightarrow[]{\P}} 0$.

Let us prove \Cref{thm:variance}, that is
$(n\sigma^2)^{-1/2}
\|(\bY-\bX\hbB) (\bI_{T\times T}-\hbA/n)^{-1} \bb\|_2\smash{\xrightarrow[]{\P}} 1$.
By the law of large numbers, we have
$\|\bE \bb\|_2^2/(n\sigma^2)\smash{\xrightarrow[]{\P}} 1$, so it suffices to show that 
$$(n\sigma^2)^{-1/2} \|\bE\big[(\bI_{T\times T}-\hbA/n)^{-1} -\bI_{T\times T}\big] \bb - \bX(\bB^*-\hbB) (\bI_{T\times T}-\hbA/n)^{-1} \bb \|_2  \smash{\xrightarrow[]{\P}} 0.$$
Techniques similar to those above show that 
$(n\sigma^2)^{-1/2} \|\bE\big[(\bI_{T\times T}-\hbA/n)^{-1} -\bI_{T\times T}\big] \bb \|_2 \smash{\xrightarrow[]{\P}} 0$
by \Cref{propMatrixA}(iii),
and that
$(n\sigma^2)^{-1/2} \|\bX(\bB^*-\hbB) (\bI_{T\times T}-\hbA/n)^{-1} \bb \|_2  \smash{\xrightarrow[]{\P}} 0$
by \Cref{lemma:risk} and $\bar R\to 0$.

An application of Slutsky's lemma completes the proof of \Cref{thm-normal}.
\end{proof}

\section{Proof for \texorpdfstring{$\chi^2_T$}{chi\texttwosuperior} limits, and confidence ellipsoid with nominal coverage}
\label{section:proofs-chi2-quantiles}

\begin{lemma}[Differentiation with respect to $\bE$]
    \label{lemma:differentiation-E}
    Here, we consider differentiation with respect to $\bE$ for fixed $\bX$.
    We have
    \begin{align*}
        \E\left[I\{\Omega_*\}
         \|\bE^\top\bX(\hbB - \bB^*) - \sigma^2 \hbA\|_F^2
    \right]
  &\le
    \sigma^2 nT {\bar R}^{2}
    +
    \sigma^4 nT.
    \end{align*}
\end{lemma}
\begin{proof}
    Let $\bF:\R^{n\times T}\to\R^{n\times T}$ be the function
    $\bF:\bE\mapsto \bX(\hbB-\bB^*)$.
    The function $\bF$ is 1-Lipschitz by \cite[Proposition 3.1]{bellec2016bounds}.
    Furthermore, $\|\bF\|_F\le \sqrt n \bar R$ on $\Omega_*$ by \Cref{lemma:risk},
    so that if $\bPi:\R^{n\times T}\to\R^{n\times T}$ is the convex projection
    onto the Frobenius ball of radius $\sqrt n\bar R$, the composition
    $\bar\bF = \bPi \circ \bF$ coincides with $\bF$ on $\Omega_*$.
    The function $\bar\bF$ is also 1-Lipschitz by composition of two
    1-Lipschitz functions, and
    since $\Omega_*$ is open, the derivatives of $\bar\bF$ and $\bF$
    with respect to $\bE$ coincide in $\Omega_*$ where the derivatives exist
    (this existence of the derivatives is granted almost everywhere
    by Rademacher's theorem).

    For any $t,t'\in[T]$, by the main result of \cite{bellec_zhang2018second_stein}
    applied to the function $\bE\be_{t'} \mapsto \bar\bF\be_t$,
    we have
    \begin{align*}
    &\E\Bigl[
        \bigl( \be_{t'}^\top\bE^\top\bar\bF\be_{t} - \sigma^2\sum_{i=1}^n\frac{\partial \be_i^\top\bar\bF\be_t}{\partial E_{it'}} \bigr)^2
    \Bigr]
  \\&=
      \sigma^2 \E\left[\|\bar\bF\be_{t}\|_2^2\right]
    +
    \sigma^4 \E\Big[\sum_{i=1}^n \sum_{i'=1}^n
    \Big(\frac{\partial}{\partial E_{i't'}}
        \be_i^\top\bar\bF\be_{t}\Big)
    \Big(\frac{\partial}{\partial E_{it'}}
        \be_{i'}^\top\bar\bF\be_{t}\Big) \Big]
  \\&\le
      \sigma^2 \E\left[\|\bar\bF\be_{t}\|_2^2\right]
    +
    \sigma^4 \E\Big[\sum_{i=1}^n\sum_{i'=1}^n
    \Big(\frac{\partial}{\partial E_{i't'}}
        \be_i^\top\bar\bF\be_{t}\Big)^2
    \Big].
    \end{align*}
We now sum the above inequalities for all $t,t'\in[T]$ to find
\begin{align*}
&\sum_{t=1}^T\sum_{t'=1}^T\E\Bigl[
    \bigl( \be_{t'}^\top\bE^\top\bar\bF\be_{t} - \sigma^2\sum_{i=1}^n\frac{\partial \be_i^\top\bar\bF\be_t}{\partial E_{it'}} \bigr)^2
\Bigr]
\\&\le
\sigma^2 T \E[\|\bar\bF\|_F^2]
+
\sigma^4 \E\Big[\sum_{t=1}^T\sum_{t'=1}^T\sum_{i=1}^n\sum_{i'=1}^n
    \Big(\frac{\partial}{\partial E_{i't'}}
        \be_i^\top\bar\bF\be_{t}\Big)^2
    \Big]
\\&\le
\sigma^2 T n \bar R^2
+ \sigma^4 nT,
\end{align*}
where for the last inequality we used that
$\|\bar\bF\|_F \le \bar R \sqrt n$ by construction of $\bar\bF$
and that $\bar\bF:\R^{n\times T}\to\R^{n\times T}$ is 1-Lipschitz, so that
the Frobenius norm of the Jacobian of $\bar\bF$ (which is a matrix of size $(nT)\times (nT)$) is at most $\sqrt{nT}$.
Finally, on $\Omega_*$ we have $\bF=\bar\bF$ and their derivatives
coincide, and by differentiating the KKT conditions of $\hbB$ we find
$
\sum_{i=1}^n\frac{\partial \be_i^\top\bF\be_t}{\partial E_{it'}}
=\hbA_{tt'}$ on $\Omega_*$ for $\bF=\bX(\hbB-\bB^*)$.
This completes the proof.
\end{proof}

\begin{restatable}{theorem}{thmChiSquareInitial}
    \label{thm:chi2initial}
    Let $\ba\in\R^p$ with $\|\bSigma^{-1/2}\ba\|_2=1$.
    Let $\bxi$ be defined in \eqref{eq:def-bxi}
    and $\hat\sigma^2 = \|\bY - \bX\hbB\|_F^2/(nT)$.
    Then under \Cref{assum:main}, 
    $|\hat \sigma/\sigma - 1| = o_\P(T^{-1/2})$ as well as
    \begin{equation}
        \label{eq:conclusion-thm-chi2initial}
        \max\{
            (\sigma^2n)^{-1/2},
            (\hat\sigma^2n)^{-1/2}
        \}
            \big\|\bxi - \sqrt n\bE^\top \bz_0 \|\bz_0\|_2^{-1}\big\|_2 =  o_\P(1).
    \end{equation}
\end{restatable}
\begin{proof}[Proof of \Cref{thm:chi2initial}]
    By definition of $\bxi$ we have
    $$
    (n\sigma^2)^{-1/2}
    \|\bxi - \bE^\top\bz_0\|_2
    =
    (n\sigma^2)^{-1/2}
    \|(\hbB-\bB^*)^\top \bX^\top \bz_0 - (n\bI_{T\times T}-\hbA) (\hbB-\bB^*)^\top \ba\|_2
    $$
    which converges to 0 in probability by \Cref{lemma:RemainderII}.
    Next, with $\chi^2_T = \sigma^{-2} \big\|\bE^\top\bz_0\|\bz_0\|_2^{-1}\big\|_2^2$,
    \begin{equation}
        \label{eq:previous-displayfejwio}
    (n\sigma^2)^{-1/2}\big\|\sqrt n\bE^\top\bz_0\|\bz_0\|_2^{-1} - \bE^\top\bz_0\big\|_2
    =
    (\chi_T^2)^{1/2}
    \big|1- n^{-1/2} \|\bz_0\|_2 \big|.
    \end{equation}
    By the Cauchy-Schwarz inequality
    we have $\E[(\chi_T^2)^{1/2} \big|1- n^{-1/2} \|\bz_0\|_2 \big|] \leq \sqrt{T/n} \, \E[(\|\bz_0\|_2-\sqrt n)]^{1/2}$.
    Combining Theorem 3.1.1 and Equation 2.15 in \cite{vershynin2018high} yields 
    $\E[(\|\bz_0\|_2-\sqrt n)]^{1/2}\leq C$ for some absolute constant $C$.
    Thus, by \Cref{assum:main} we have $T/n\to0$ so that
    \eqref{eq:previous-displayfejwio}
    converges to $0$ in $L^1$, hence in probability.
    This proves $(\sigma^2n)^{-1/2}
    \big\|\bxi - \sqrt n\bE^\top\bz_0\|\bz_0\|_2^{-1} \big\|_2
    = o_\P(1)$.

    We now prove the same bound with $\sigma^2n$ replaced by $\hat\sigma^2n$.
    Let $\Omega_8 = \{
        |\|\bE\|_F/\sigma - \sqrt{nT}|\le \sqrt{\log n}
    \}$.
    Then $\P(\Omega_8)\to 1$ by \cite[Theorem 3.1.1]{vershynin2018high}
    and
    \begin{align}
        \nonumber
    I\{\Omega_8\cap\Omega_*\}
    |\hat \sigma/\sigma - 1|
    &\le
    I\{\Omega_*\}
    \|\bX(\hbB-\bB^*)\|_F(\sigma^2 nT)^{-1/2}
    +
    I\{\Omega_8\}
    |\sqrt{nT} - \|\bE\|_F/\sigma|
    (nT)^{-1/2}
  \\&\le
    (1-\eta_3)\bar R / \sqrt{\sigma^2 T}
    +
    (nT)^{-1/2} \sqrt{\log n}
    \label{upper-bound-sigma-hat-sigma}
    \end{align}
    by \Cref{lemma:risk} for the first term.
    This proves that $|\hat \sigma /\sigma - 1| = o_\P(T^{-1/2})$.
    under \Cref{assum:main} so that using
    $\frac 1 2 |\frac1u-1| \le |u-1|$ for $u\in[\frac12,\frac32]$ we obtain for $n$ large enough
    $$
    (1/2)
    I\{\Omega_8\cap\Omega_*\}
    |\sigma/\hat\sigma-1|
    \le
    I\{\Omega_8\cap\Omega_*\}
    |\hat\sigma/\sigma-1|
    \le
    \eqref{upper-bound-sigma-hat-sigma}.
    $$
    Hence $\sigma/ {\hat \sigma} = 1 + o_\P(1)$, thus 
    \begin{align*}
        (n\hat \sigma^2)^{-1/2}\big\|\sqrt n\bE^\top\bz_0\|\bz_0\|_2^{-1} - \bE^\top\bz_0\big\|_2
        &= (\sigma/ {\hat \sigma}) (n \sigma^2)^{-1/2}\big\|\sqrt n\bE^\top\bz_0\|\bz_0\|_2^{-1} - \bE^\top\bz_0\big\|_2 
        \\ &=
       ( 1 + o_\P(1))o_\P(1) = o_\P(1).
    \end{align*}
\end{proof}

\thmChiSquareQuantiles*

\begin{proof}[Proof of \Cref{thm:chi2z0}]
    \Cref{thm:chi2initial} applied with $\bz= \bz_0 \|\bz_0\|_2^{-1}$
    yields the bound
    $(\sigma^2n)^{-1/2}\|\bxi - \sqrt n \bE^\top \bz\|_2 = o_\P(1)$.
    The proof then follows from \Cref{lemma:bxi-bz}.
\end{proof}

\begin{lemma}
    \label{lemma:bxi-bz}
    Let \Cref{assum:main} be fulfilled.
    Let $\bz,\bxi$ be random vectors valued in $\R^n$.
    Assume that
    $\bz$ is a measurable function of $\bX$ with $\P(\|\bz\|_2=1)=1$ 
    and let $\bP_{\bz}^\perp = \bI_n - \bz\bz^\top$.
    Then the random variable $F_{T,n-T} = \frac{n-T}{T} \|{({\bE^\top \bP_{\bz}^\perp\bE})^{-1/2}} \bE^\top \bz\|_2^2$
    has the $F$ distribution with degrees of freedom
    $T$ and $n-T$, and the following holds:
    \begin{enumerate}
        \item $\sqrt{TF_{T,n-T}} = \sqrt{\chi^2_T} + o_\P(1)$
            as $n\to+\infty$ when $T/n\to 0$ where $\chi^2_T$
            is a random variable with chi-square distribution with $T$
            degrees of freedom,
        \item
           $\P(\lambda_{\min}(\hbGamma)\ge n\sigma^2/2)\to 1$,
        \item
            $\sqrt{n-T}\|\hbGamma^{-1/2}\bE^\top\bz\|_2 -\sqrt{TF_{T,n-T}}
            \le
            o_\P(1) + 
            O_\P(\frac{T}{\sqrt n})$,
        \item
            $\sqrt{n-T}\|\hbGamma^{-1/2}\bE^\top\bz\|_2 -\sqrt{TF_{T,n-T}}
            \ge
            -o_\P(1)
            - O_\P(
            \frac{T}{\sqrt n}+
            \frac{sT + s\log(p/s)}{n}\sqrt T
            )
            $,
        \item
        $\sqrt{n-T}\|\hbGamma^{-1/2}\bE^\top\bz\|_2 -\sqrt{TF_{T,n-T}}
            \le
            o_\P(1) + 
            O_\P(\tfrac{s(s+T)\log^2(p/s)}{n\sqrt T})
        $.
    \end{enumerate}
    Consequently, if 
            $
            (\sigma^2n)^{-1/2}\|\bxi - {\sqrt n} \bE^\top\bz\|_2 = o_\P(1)$ then
            \begin{align}
                (1-\tfrac{T}{n})^{1/2}
                \|\hbGamma^{-1/2}\bxi\|_2 &\le
            (\chi^2_T)^{1/2}
            + o_\P(1)
        + O_\P\bigl(
            \min\bigl\{
                \tfrac{T}{\sqrt n},
                \tfrac{\log^2(p/s)}{n^{1/4}}
            \bigr\}
        \bigr)
            \label{eq:conclusion-lemma-bxi-upper-bound}
            \\
                (1-\tfrac{T}{n})^{1/2}
                \|\hbGamma^{-1/2}\bxi\|_2 
                &\ge
                (\chi^2_T)^{1/2}
                - o_\P(1)
                -
                O_\P\bigl(
                \tfrac{T}{\sqrt n}
                +
                \tfrac{(sT + s\log \tfrac p s )\sqrt T}{n}
                \bigr).
                \label{eq:conclusion-lemma-bxi-lower-bound}
            \end{align}
\end{lemma}

\begin{proof}[Proof of \Cref{lemma:bxi-bz}]
    For (i), we introduce the quantity
\begin{equation}
    H\coloneqq 
(n-1)
\|
({\bE^\top \bP_{\bz}^\perp\bE})^{-1/2} \bE^\top\bz
\|_2^2
=
(n-1)\bg^\top \bW^{-1} \bg
\end{equation}
where $\bg = \sigma^{-1}\bE^\top\bz$ 
and $\bW = \sigma^{-2} \bE^\top\bP_{\bz}^\perp \bE$.
Since $\bE$ and $\bz$ are independent and since $\|\bz\|_2=1$, 
$\bg$ has distribution $\mathcal N_T(\mathbf{0},\bI_{T\times T})$.
$\bP_{\bz}^\perp$ can be orthogonally diagonalized as $\bQ\big(\sum_{i=1}^{n-1} \be_i \be_i{}^\top\big) \bQ^{\top}$ 
where $\bQ$ is an $n\times n$ orthogonal matrix, thus 
$\bW = \sum_{i=1}^{n-1} \bn_i \bn_i^\top$ 
where the random vectors $\bn_i = \sigma^{-1} \bE^\top \bQ \be_i$
are iid with standard normal
$\mathcal N_T(\mathbf{0},\bI_{T\times T})$ distribution.
Therefore $\bW$ has the
Wishart distribution with identity covariance and $n-1$ degrees-of-freedom.
Since $\bE^\top \bz$ and $\bE^\top \bP_{\bz}^\perp$ are independent, so are $\bE^\top \bz$ 
and $(\bE^\top \bP_{\bz}^\perp) (\bE^\top \bP_{\bz}^\perp)^\top =  \bE^\top \bP_{\bz}^\perp \bE$, 
thus $\bg$ and $\bW$ are independent.
By \cite[Theorem 5.8]{Hardle2019}
$H$ has the Hotelling distribution with parameters $T,n-1$,
and 
$$
\frac{n-1 - T + 1}{T}
\frac{H}{n-1}
\sim F_{T, n-1 - T + 1}
= F_{T,n-T}
$$
where the right-hand side is the $F$ distribution with 
degrees-of-freedom $T$ and $n-T$.
Furthermore, since $F_{T,n-T} = \frac{\chi^2_T/T}{\chi^2_{n-T}/(n-T)}$
for some random variables having chi-square distributions with respective parameter $T$ and $n-T$,
we have
$$|\sqrt{TF_{T,n-T}} - \sqrt{\chi^2_T}|
  =  |\sqrt{\chi^2_{T}/(\chi^2_{n-T}/(n-T))}
- \sqrt{\chi^2_T}
|
= O_\P(\sqrt T) | 1 - \sqrt{\chi^2_{n-T}/(n-T)}| 
$$
    where the last equality follows from $\E[(\chi^2_{T})^{1/2}]\leq \E[\chi^2_{T}]^{1/2} = \sqrt T$ and
the a.s. convergence of $\chi^2_{n-T}/(n-T)$ to $1$.
Furthermore $|1-|a||\le |1-a^2|$ and the Central Limit Theorem yield
$$
| 1 - \sqrt{\chi^2_{n-T}/(n-T)}|
\le \frac{|\chi^2_{n-T} - (n-T)|}{n-T}
= O_\P((n-T)^{-1/2}).
$$
Thus $\sqrt{TF_{T,n-T}} = (\chi^2_T)^{1/2} + O_\P((\frac nT -1)^{-1/2})$,
and since $\frac nT\to 0$ 
we have $\sqrt{TF_{T,n-T}} = (\chi^2_T)^{1/2} + o_\P(1)$
and $\P(\sqrt{TF_{T,n-T}} \le q_{T,\alpha})\to 1-\alpha$
by \Cref{prop:4.2}. This proves (i).

    Next we exhibit a lower bound on the eigenvalues of $\hbGamma$.
    Let $\bH = \hbB -\bB^*$ and consider the decomposition
    \begin{equation}
        \hbGamma 
        = \bE^\top\bE + (\bX\bH)^\top(\bX\bH) -  [\bE^\top\bX\bH + (\bX\bH)^\top\bE ].
      \label{eq:decomposition-hbGamma}
    \end{equation}
    Since $(\bX\bH)^\top(\bX\bH)$ is positive semidefinite 
    we have
    \begin{equation}
            \label{lower-decomposition-hbGamma}
        \hbGamma ~\succeq~
        \bE^\top\bE
        - 2 \|\bE^\top\bX\bH\|_{op}
     \bI_{T\times T}.
        \end{equation}
    Since $\bE$ has i.i.d. $\mathcal N(0,\sigma^2)$ entries, 
    if $s_{\min}(\bE)$ and $s_{\max}(\bE)$ denote the smallest and greatest singular values of $\bE$ 
    we have $\sigma(\sqrt n - \sqrt T) \leq \E[s_{\min}(\bE)] \leq \E[s_{\max}(\bE)]\leq \sigma(\sqrt n +\sqrt T)$ by \cite[Theorem II.13]{DavidsonS01}.
    Since $s_{\min}(\bE)$ and $s_{\max}(\bE)$ are $1$-Lipschitz functions of $\bE$ when considered as a vector in $\mathbb R^{nT}$, 
    Gaussian concentration as stated in \cite[Theorem~B.6]{giraud2012high} yields the existence of exponential random variables
    $Z_1,Z_2\sim \text{Exp}(1)$ such that almost surely
    $$\sigma(\sqrt n - \sqrt T - \sqrt{2Z_1}) \leq s_{\min}(\bE) \leq s_{\max}(\bE) \leq \sigma(\sqrt n +\sqrt T+\sqrt{2Z_2}).$$
    Letting $Z=2\max(Z_1,Z_2)$, we have
    $$
    \sigma^2 (\sqrt n - \sqrt T - \sqrt Z)_{+ }^2 \bI_{T\times T}  \preceq
     \bE^\top\bE \preceq \sigma^2 (\sqrt n + \sqrt T + \sqrt Z)^2 \bI_{T\times T}.
    $$
    Thanks to \eqref{lower-decomposition-hbGamma} and the inequality $(1-x)_+^2\geq 1-2x$ for $x\geq 0$
    we have
    \begin{equation}
        \label{eq:lower-eigenvalues-of-hbGamma}
        \hbGamma \succeq \sigma^2 n [1 - 2(\sqrt{T/n}+\sqrt{Z/n}) 
        -2 \|\bE^\top\bX\bH\|_{op} /(\sigma^2 n)
        ] \bI_{T\times T}.
    \end{equation}
    On the event $\Omega_9 = \{
        1 - 2(\sqrt{T/n}+\sqrt{Z/n}) 
        -2 \|\bE^\top\bX\bH\|_{op}/(\sigma^2 n)
    > 1/2 \}$ we have $\lambda_{\min}(\hbGamma) 
    \ge \lambda_{\min}(\bE^\top\bE - 2\|\bE^\top\bX\bH\|_{op}\bI_{T\times T})
    \ge \sigma^2n/2$.
We now proceed to show that $\P(\Omega_9)\to 1$.
We have by the triangle inequality for the norm $\E[(\cdot)^2]^{1/2}$ that
\begin{align}
    \label{eq:upper-bound-TZnorm}
& 
\E\Big[I\{\Omega_*\}
\Big(
\sqrt{T/n} + \sqrt{Z/n} + \|\bE^\top\bX\bH\|_{op}/(\sigma^2 n)
\Big)^2
\Big]^{1/2}
\\&\le
\sqrt{T/n} + \E[Z]^{1/2}/\sqrt n
+
\bar s/n
+
\E[I\{\Omega_*\}\|\bE^\top\bX\bH-\sigma^2\hbA\|_{op}^2/(\sigma^2n)^2]^{1/2} \nonumber
\\&\le
\sqrt{T/n} + \E[Z]^{1/2}/\sqrt n
+
\bar s/n
+
[(T/n)(1 + {\bar R}^2/\sigma^2 )]^{1/2} \nonumber
\end{align}
where we used \Cref{propMatrixA}(ii) and \Cref{lemma:sparsity}
to bound $\|\hbA\|_{op}$ from above by $\bar s$ on $\Omega_*$
for the first inequality, and
\Cref{lemma:differentiation-E} the second inequality.
Hence under \Cref{assum:main}, the previous display converges to 0.
Next, $\P(\Omega_9^c)=\P(\Omega_9^c\cap\Omega_*^c) + \P(\Omega_9^c\cap\Omega_*)$,
Markov's inequality and an application of Jensen's
inequality yield
\begin{align*}
    \P(\Omega_9^c) &= \P(\Omega_*^c\cap\Omega_9^c) + \P\Big(1/4 \leq I\{\Omega_*\}\big(  \sqrt{T/n} + \sqrt{Z/n} + \|\bE^\top\bX\bH\|_{op}/(\sigma^2 n)\big)\Big) \\
    &\leq \P(\Omega_*^c) + 4\E\Big[I\{ \Omega_*\} \big(\sqrt{T/n} + \sqrt{Z/n} + \|\bE^\top\bX\bH\|_{op}/(\sigma^2 n) \big)\Big]
    \leq \P(\Omega_*^c) + 4\eqref{eq:upper-bound-TZnorm}
\end{align*}
where $4\eqref{eq:upper-bound-TZnorm}$ refers to four times the quantity
\eqref{eq:upper-bound-TZnorm} which converges to 0.
Thus the event $\Omega_9$ has probability
approaching one and claim (ii) follows.

We now prove (iii)-(v).
Let $\Omega(n)$ be a sequence of events with $\P(\Omega(n))\to 1$, 
$V_n$ be any sequence of random variables and $a_n$ be any deterministic sequence of real numbers.
It is easily seen that 
$I\{ \Omega(n)\}V_n = o_\P(a_n)$ implies $V_n= o_\P(a_n)$
and $I\{ \Omega(n)\}V_n = O_\P(a_n)$ implies $V_n= O_\P(a_n)$.
This observation will allow us to transition seamlessly from bounds on $I\{ \Omega(n)\}V_n$ to bounds on $V_n$ 
by choosing, e.g., $\Omega(n)= \Omega_*\cap\Omega_9$ or other events of
probability approaching one in our problem.
It will be useful to note that by the same argument as above
$\phi_{\min}(\bE^\top\bP_{\bz}^\perp\bE)
\ge \sigma^2 (\sqrt{n-1} - \sqrt T -\sqrt{2Z_3})_+^2$
where $Z_3\sim \text{Exp}(1)$,
so that $\phi_{\min}(\bE^\top\bP_{\bz}^\perp\bE) \ge \sigma^2n/2$
on an event $\Omega_8$ of probability approaching one.
We will use the following fact: if $\bM,\bN$ are
two positive definite matrices with eigenvalues at least $1/2$ then
\begin{equation}
    \label{operator-norm-inverse}
    \|\bM^{-1/2} - \bN^{-1/2}\|_{op}
    \le 2 \|\bM^{1/2} - \bN^{1/2}\|_{op}
    \le \sqrt 2 \|\bM - \bN\|_{op} 
\end{equation}
using the resolvent identity 
$\bM^{-1/2} - \bN^{-1/2}=\bN^{-1/2}(\bN^{1/2} - \bM^{1/2})\bM^{-1/2}$ for the first inequality
and \cite{stackexchange3968118} for the second.
To prove (iii), we apply  \eqref{operator-norm-inverse}
to $\bM =(\sigma^2n)^{-1}[\bE^\top\bE - 2 \|\bE^\top\bX\bH\|_{op} \bI_{T\times T}]$
and $\bN = (\sigma^2n)^{-1}\bE^\top\bP_{\bz}^\perp\bE$, both matrices
having eigenvalues at least $1/2$ on $\Omega_9\cap\Omega_8$.
Rewriting \eqref{lower-decomposition-hbGamma} as $\hbGamma{}^{-1/2 } \preceq (\sigma^2n)^{-1/2 } \bM^{-1/2 }$,
applying the triangle inequality and \eqref{operator-norm-inverse}, we have 
on $\Omega_8\cap\Omega_9$
\begin{align}
    \label{eq:Delta}
\Delta 
& \coloneqq \sqrt{n-T}\|\hbGamma{}^{-1/2}\bE^\top\bz\|_2 -\sqrt{TF_{T,n-T}}
    \\
    \nonumber
& =\sqrt{n-T}(\|\hbGamma{}^{-1/2}\bE^\top\bz\|_2 -\|(\bE^\top\bP_{\bz}^\perp\bE)^{-1/2} \bE^\top\bz\|_2) 
    \\
    \nonumber
& \leq 
    \sqrt{(n-T)/(\sigma^2 n)} \| \bM^{-1/2} \bE^\top \bz\|_2 
        - \sqrt{(n-T)/(\sigma^2 n)} \| \bN^{-1/2} \bE^\top \bz \|_2
    \\
    \nonumber
& \leq
    \sqrt{(n-T)/(\sigma^2 n)}\|
         (
         \bM^{-1/2}
         -
         \bN^{-1/2}
       ) \bE^\top\bz\|_2
     \\
    \nonumber
&\le
     \sqrt{1-T/n}
     \sqrt 2
     \Big\|(\sigma^2 n)^{-1}[\bE^\top\bz\bz^\top\bE - 2 \|\bE^\top\bX\bH\|_{op}\bI_{T\times T}] \Big\|_{op}
 \|\bE^\top\bz\|_2\sigma^{-1}.
\end{align}
The bounds used in \eqref{eq:upper-bound-TZnorm} yield
$I\{\Omega_*\}  \|\bE^\top\bX\bH-\sigma^2\hbA\|_{op}(\sigma^2n)^{-1} = O_\P(\sqrt{T/n})$
and $I\{\Omega_*\}\|\hbA\|_{op} = O_\P(\bar s)$ ,
hence $ \frac{\|\bE^\top\bX\bH\|_{op}}{\sigma^2n} 
        \leq 
         \frac{\|\bE^\top\bX\bH-\sigma^2\hbA\|_{op}}{\sigma^2n}
            + \frac{\|\hbA\|_{op}}n
        =O_\P(\frac{\sqrt T}{\sqrt n}) + O_\P(\frac {\bar s}n).
        $
Furthermore $\|\bE^\top\bz\|_2^2/\sigma^2$ has $\chi^2_T$ distribution, thus 
$\|\bE^\top\bz\|_2^2/\sigma^2 = O_\P(T)$
and
we obtain
$$ \Delta 
        \leq
         \sqrt{1-T/n} \Big( O_\P(\tfrac Tn) + O_\P(\tfrac{\sqrt T}{\sqrt n}) + O_\P(\tfrac {\bar s}n) \Big) O_\P(\sqrt T).$$
Since $\frac Tn \to 0$, the right-hand side of the equality is 
$O_\P(\tfrac {T}{\sqrt n}) + O_\P(\tfrac {s\sqrt T}{n}) = O_\P(\tfrac {T}{\sqrt n}) + o_\P(1)$.

For claim (iv), with $\Delta$ defined in \eqref{eq:Delta} a similar argument yields
\begin{align*}
|\Delta|
     & \le
     \sqrt{n-T}\|
     \bigl(
         \hbGamma {}^{-1/2}
     -
     (\bE^\top\bP_{\bz}^\perp\bE)^{-1/2}
 \bigr) \bE^\top\bz\|_2
     \\&\le
\sqrt 2 \sqrt{1-T/n}
(\sigma^2 n)^{-1}
     \bigl[
     \|\bE^\top\bz\|_{op}^2
     + 2 \|\bE^\top\bX\bH\|_{op}
     + \|\bX\bH\|_{op}^2
     \bigr]
 \|\bE^\top \bz\|_2/\sigma
\end{align*}
on $\Omega_8\cap\Omega_9$, 
thus $|\Delta| \leq 
\sqrt{1-T/n} \bigl( O_\P(\tfrac Tn) + O_\P(\tfrac{\sqrt T}{\sqrt n}) + O_\P(\tfrac {\bar s}n) + O_\P(\bar R^2) \bigr) O_\P(\sqrt T)$
thanks to \Cref{lemma:risk}(ii) for the term $\|\bX\bH\|_{op}/(\sigma^2n)$.
This proves (iv).

It remains to prove (v), for which we need a more subtle argument.
The important remark is that on the one hand
$\bE^\top\bP_{\bz}^\top$ is independent
of $\bE^\top\bz$ because $\bE$ has iid $\mathcal N(0,\sigma^2)$ entries,
while on the other hand $\hbGamma$ is not independent of $\bE^\top\bz$.
To overcome this lack of independence, we bound $\hbGamma$ from below
by a positive definite matrix independent of $\bE^\top\bz$, as follows.
For a fixed subset $J\subset [p]$,
let $\bP_J$ be the orthogonal projection matrix
onto the linear span of
$\{\bz\}\cup \{ \bX\be_j, j\in J\}$
so that the rank of $\bP_J$ is at most $|J|+1$.
Set $\bP_J^\perp =\bI_{n\times n} - \bP_J$. Then
in the event
\begin{equation}
    \hat S \cup \supp(\bB^*) \subset J,
    \label{event}
\end{equation}
we have $\bP_J^\perp\bX(\hbB-\bB^*) = \mathbf{0}$, hence 
$\hbGamma 
\succeq (\bY - \bX\hbB)^\top\bP_J^\perp(\bY - \bX\hbB)
=
\bE^\top\bP_J^\perp \bE$,
thus
\begin{align*}
\sqrt{n-T}
\|\hbGamma^{-\frac12}\bE^\top \bz \|_2
&\leq 
\sqrt{n-T}
\|(\bE^\top\bP_J^\perp\bE)^{-\frac12} \bE^\top \bz \|_2.
\end{align*}
For a fixed $J$ and in the event $\hat S\cup \supp \bB^*\subset J$, we can bound from above $\Delta$ in \eqref{eq:Delta} as
\begin{align}
     \Delta
     &\le
\sqrt{n-T}
\bigl[
    \|({\bE^\top \bP_J^\perp} \bE)^{-\frac12} \bE^\top \bz \|_2
-
\|({\bE^\top \bP_{\bz}^\perp \bE  })^{-\frac12}\bE^\top \bz \|_2
\bigr]
\nonumber
   \\&\le
   \frac{\sqrt{n-T}}{\|({\bE^\top \bP_{\bz}^\perp }\bE)^{-\frac12}\bE^\top \bz \|_2}
\Bigl[
    \|  ({\bE^\top \bP_J^\perp}\bE)^{-\frac12}\bE^\top \bz \|_2^2
-
\|(\bE^\top \bP_{\bz}^\perp \bE)^{-\frac12} \bE^\top\bz\|_2^2
\Bigr]_{+}
\nonumber
   \\&=
   \frac{\sqrt{n-T}}{\|({\bE^\top \bP_{\bz}^\perp}\bE)^{-\frac12}\bE^\top \bz \|_2}
\Bigl[
    \bg^\top
    \Bigl\{
    (\bE^\top\bP_J^\perp\bE)^{-1}
    -
    (\bE^\top\bP_{\bz}^\perp\bE)^{-1}
    \Bigr\}
    \bg
\Bigr]_{+}
\label{eq:thing-to-bound-Woodberry}
\end{align}
where $\bg=\bE^\top \bz \sim \mathcal N_T(\mathbf{0},\sigma^2\bI_{T\times T})$
as before,
the first inequality follows from 
$\hbGamma{}^{-\frac 12} \preceq (\bE^\top\bP_J^\perp \bE)^{-\frac 12}$ and the second from
$\sqrt a - \sqrt b \le (a-b)_+/ \sqrt b$
.
For any $J\subset[p]$, the null space inclusion
$\ker \bP_J\subset \ker \bz\bz^\top$ holds and
the matrix $\bP_{\bz}^\perp - \bP_J^\perp$ is 
an orthogonal projection matrix with rank $r\leq |J|$ so that
$\bP_{\bz}^\perp - \bP_J^\perp = \bQ_J\bQ_J^\top$
for the matrix $\bQ_J\in\R^{n\times {r}}$ with orthonormal columns
given by $\bQ_J = \sum_{k=1}^{r} \bu_k \be_k^\top$ where
    $\bu_k\in\R^n$ are orthonormal eigenvectors of $\bP_{\bz}^\perp - \bP_J^\perp$
corresponding to the non-zero eigenvalues and $\be_k$ are
canonical basis vectors in $\R^{r}$.
By the Sherman-Morrison-Woodbury identity, the matrix in curly brackets is equal to
$$
\bM_J
\coloneqq
(\bE^\top\bP_{\bz}^\perp\bE)^{-1}
\bE^\top\bQ_J
\Bigl(
\bI_{r\times r}
-
\bQ_J^\top\bE
(\bE^\top\bP_{\bz}^\perp\bE)^{-1}
\bE^\top\bQ_J
\Bigr)^{-1}
\bQ_J^\top\bE
(\bE^\top\bP_{\bz}^\perp\bE)^{-1}.
$$
Applying \cite[Theorem II.13]{DavidsonS01} to the Gaussian
matrices $\bQ_J^\top\bE$ and $\bP_{\bz}^\perp\bE$, we find
\begin{equation}
    \begin{split}
\P\bigl(
\|\bE^\top\bQ_J\|_{op} \ge \sigma(\sqrt{T} + \sqrt{|J|} + t)
\bigr) &\le e^{-t^2/2},
\\\P\bigl(
\phi_{\min}(
\bE^\top \bP_{\bz}^\perp \bE
) \le\sigma^2(\sqrt{n-1} - \sqrt{T} - t)_+^2
\bigr)&\le e^{-t^2/2}
\end{split}
\label{eq:bounds-eigenvalues-E_T_Q_J}
\end{equation}
for all $t>0$. As long as $\frac12 \ge(\frac{\sqrt{T}+\sqrt{|J|}+t}{\sqrt{n-1}-\sqrt{T}-t})^{2}$ we have
$$\bI_{r \times r}
-
\bQ_J^\top\bE
(\bE^\top\bP_{\bz}^\perp\bE)^{-1}
\bE^\top\bQ_J
\succeq
\bI_{r \times r}/2$$
and thus $\bg^\top\bM_j\bg \le 2 \|\bQ_j^\top\bE(\bE^\top\bP_{\bz}^\perp\bE)^{-1}\bg\|_2^2$.
Applying Theorem 6.3.2 in \cite{vershynin2018high} and because
$\bg$ is independent of $(\bQ_J^\top\bE, \bP_{\bz}^\perp\bE)$, we find
$$
\P(
\|\bQ_j^\top\bE(\bE^\top\bP_{\bz}^\perp\bE)^{-1}\bg {/\sigma} \|_2  
\ge
\|\bQ_j^\top\bE(\bE^\top\bP_{\bz}^\perp\bE)^{-1}\|_F
+ C t 
\|\bQ_j^\top\bE(\bE^\top\bP_{\bz}^\perp\bE)^{-1}\|_{op}
)\le 2e^{-t^2/2}
$$
for some absolute constant $C>0$.
Combined with \eqref{eq:bounds-eigenvalues-E_T_Q_J} and the union bound,
$$\P\Bigl[
\bg^\top\bM_J\bg
\ge
2
\Bigl(
    \frac{\sigma^{-1}\|\bQ_J^\top\bE\|_F}{(\sqrt{n-1} - \sqrt{T} -t)_+^2}
    +
    Ct
    \frac{\sqrt{T}+\sqrt{|J|} + t}{(\sqrt{n-1} - \sqrt{T} -t)_+^2}
\Bigr)^2
\Bigr]
\le 4 e^{-t^2/2}.
$$
By concentration of 
chi-square distributed random variables with $Tr$ degrees of freedom
(e.g., Theorem 5.6 in \cite{boucheron2013concentration}), we also have
$\P(
\sigma^{-1}\|\bQ_J^\top\bE\|_F \ge \sqrt{T|J|} + t
)\le e^{-t^2/2}$
since $Tr\le T|J|$.
Let $s_* = \bar s + s$ and
note that for $t\geq 0$, 
\begin{align*}
&\phantom{=}\P
\Big(
\Bigl\{
\Delta \geq 
\frac{n-T}{\sqrt{TF_{T,n-T}}}
\Bigl[
    \frac{
        \sqrt{Ts_*}
        +t
        +
        Ct
        (
    \sqrt{T}+\sqrt{s_*} + t
        )
}
    {(\sqrt{n-1} - \sqrt{T} - t)_+^2}
\Bigr]^2
\Bigr\}
\cap \Omega_* 
\Big)
\\ &\leq 
\P\Big( 
\bigcup_{\substack{J\subset [p]\\ |J|=s_*}}
\Big\{
\Delta \geq 
\frac{n-T}{\sqrt{TF_{T,n-T}}}
\Bigl[
    \frac{
        \sqrt{T|J|}
        +t
        +
        Ct
        (
    \sqrt{T}+\sqrt{|J|} + t
        )
}
    {(\sqrt{n-1} - \sqrt{T} - t)_+^2}
\Bigr]^2
\Big\}
\cap
\{\hat S \cup \supp(\bB^*) \subset J\}
\Big)
\\ &\leq 
\sum_{\substack{J\subset [p]\\ |J|=s_*}}
\P\Big( 
\Big\{
\Delta \geq 
\frac{n-T}{\sqrt{TF_{T,n-T}}}
\Bigl[
    \frac{
        \sqrt{T|J|}
        +t
        +
        Ct
        (
    \sqrt{T}+\sqrt{|J|} + t
        )
}
    {(\sqrt{n-1} - \sqrt{T} - t)_+^2}
\Bigr]^2
\Big\}
\cap\{\hat S \cup \supp(\bB^*) \subset J\}
\Big)
\\ &\leq 
5\binom{p}{s_*} e^{-t^2/2},
\end{align*}
where the first inequality holds because $|\hat S \cup \supp(\bB^*)|\leq s_*$
on $\Omega_*$ by \Cref{lemma:sparsity}.
and the last one is obtained by putting together the previous concentration bounds.
Setting $t = x + (2\log\binom{p}{s_*})^{1/2}$,
we find that $\Delta$ is smaller than
$$
\frac{n-T}{\sqrt{TF_{T,n-T}}}
\Bigl[
    \frac{
        \sqrt{Ts_*}
        +\sqrt{2\log\binom{p}{s_*}}+x
        +
        C(\sqrt{2\log\binom{p}{s_*}}+x)
        (
    \sqrt{T}+\sqrt{s_*} + \sqrt{2\log\binom{p}{s_*}} + x
        )
}
    {(\sqrt{n-1} - \sqrt{T} - \sqrt{2\log\binom{p}{s_*}}- x)_+^2}
\Bigr]^2
$$
with probability at least $1-5e^{-x^2/2} -\P(\Omega_*^c)$.
Since $\E[F_{T,n-T}^{-1}] = T/(T-2)$, we have the estimate $F_{T,n-T}^{-1} = O_\P(1)$.
Under \Cref{assum:main}(iv) to control the denominator,
and by the bound $\log\binom{p}{s_*}\leq s_* \log(\tfrac{ep}{s_*})$
the above display is thus
\begin{align*}
&O_\P\Bigl(
\frac{n}{\sqrt T}
\Bigl[
\frac{{Ts} + {s\log(p/s)}
    + {sT\log(p/s)} 
    + s^2\log^2(p/s)
}{n^2}
\Bigr]
\Bigr)
\\&= O_\P\Bigl(
\frac{Ts + s\log(p/s)}{n\sqrt T}
\Bigr)
+ O_\P\Bigl(
\frac{sT \log(p/s)}{n\sqrt T}
\Bigr)
+ O_\P\Bigl(
\frac{s^2 \log^2(p/s)}{n\sqrt T}
\Bigr).
\end{align*}
In the right-hand side, the first term
is $o_\P(1)$ thanks to \Cref{assum:main}(iv).
For $n$ large enough $\log(p/s)\ge 1$ holds, thus the second and third term
are smaller
than $O_\P(\frac{s(s+T)\log^2(p/s)}{n\sqrt T})$.
This proves (v).

In order to deduce the upper bound \eqref{eq:conclusion-lemma-bxi-upper-bound}
from (iii) and (v), it is sufficient to show that
\begin{equation}
    \min\{T/\sqrt n, s(s+T)\log^2(p/s)/(n\sqrt T)\} = o(1)
+
o(\log^2(p/s) n^{-1/4})
\label{eq:to-prove-complete-lemma}
\end{equation}
holds under \Cref{assum:main}.
Let $u_n=sT/n$ and note that $u_n\to0$ by \Cref{assum:main}.
On the one hand, if $T\le \max\{\sqrt n u_n, s\}$
then $T/\sqrt n \le \max \{u_n, \sqrt{sT/n} \} = o(1)$.
On the other hand, if $T > \max\{\sqrt n u_n, s\}$
then
$$
\frac{s(s+T)\log^2(p/s)}{n\sqrt T}
\le
\frac{2s T\log^2(p/s)}{n\sqrt T}
=
\frac{2u_n \log^2(p/s)}{\sqrt T}
\le
\frac{2 u_n^{1/2} \log^2(p/s)}{n^{1/4}}
=o\Bigl(
\frac{ \log^2(p/s)}{n^{1/4}}
\Bigr).
$$
This proves \eqref{eq:to-prove-complete-lemma} and completes the proof.
\end{proof}

\chisquareQuantileProp*
\begin{proof}[Proof of \Cref{prop:4.2}]
We first prove case (i).
Then by definition of $q_{T,\alpha}$ and the union bound, for any constant $\delta>0$
not depending on $n,T$,
\begin{align*}
\P(W_{n}> q_{T,\alpha})
&\le
\P(o_\P(1) > \delta) +
\P((\chi^2_T)^{1/2} > q_{T,\alpha} - \delta)
\\& =
\P(o_\P(1) > \delta)
+ \alpha + \P\big((\chi^2_T)^{1/2} \in [q_{T,\alpha} - \delta, q_{T,\alpha}]\big).
\end{align*}
We now bound the third term. Let $f_T:[0,+\infty)\to[0,\infty)$ 
be the probability density function of $(\chi^2_T)^{1/2}$,
which admits the closed form $f_T(x) = (2^{T/2 - 1}\Gamma(T/2))^{-1} x^{T-1} e^{-x^2/2}$
for $x\ge0$.
Then
$\P((\chi^2_T)^{1/2} \in [q_{T,\alpha} - \delta, q_{T,\alpha}])
\le \delta \sup_{x>0} f_T(x)$.
The supremum $\sup_{x>0} f_T(x)$ is attained at $x= \sqrt{T-1}$, the mode of the chi distribution with $T$ degrees of freedom, so that
$$\textstyle
\sup_{x>0} f_T(x)
= (2^{T/2 - 1}\Gamma(T/2))^{-1} (T-1)^{(T-1)/2} e^{-(T-1)/2} \xrightarrow[T\to+\infty]{} {1/\sqrt{\pi}}$$ by Stirling's formula.
%
Hence there exists an absolute constant $C_0>0$ such that
$$
\P(W_{n}> q_{T,\alpha})
\le \P(o_\P(1)>\delta) + \alpha + \delta C_0.
$$
For any $\eps>0$, let $\delta=\eps/C_0$.
Using by the definition of convergence 
in probability, for $n$ large enough we have
$\P(o_\P(1)>\delta) \le \eps$ so that
$\P(W_{n}> q_{T,\alpha}) - \alpha
\le 2\eps$.
Since this holds for any $\eps>0$, the claim is proved.
The same argument can be applied in case (ii) by reversing the inequalities.

\end{proof}

\begin{proof}[Proof of \eqref{eq:order-q_T-alpha}]
    The convergence in distribution
\begin{equation}
    \label{eq:approximation-q_T-alpha}
\sqrt{2}\bigl((\chi^2_T)^{1/2} - \sqrt{T}\bigr)
= 
\frac{(\sqrt{2T})^{-1}
(\chi^2_T - T)
}{
(\chi^2_T/T)^{1/2}/2+1/2
}
\to^d \mathcal N(0,1)
\end{equation}
holds by the Central Limit Theorem for
$(\sqrt{2T})^{-1}(\chi^2_T - T)\to^d \mathcal N(0,1)$,
the weak law of large numbers for 
$(\chi^2_T/T)^{1/2}\to^\P1$
and Slutsky's theorem. If $\Phi(u)=\P(\mathcal N(0,1)\le u)$ is the standard normal cdf,
for any subsequence $(a_{T'})_{T'}$ of
$a_T = \Phi\bigl(\sqrt{2}(q_{T,\alpha} - \sqrt{T})\bigr)$
converging to an accumulation point $L$,
we have for any $\epsilon>0$ and $T'$ large enough
$$
\P\bigl[\Phi\bigl(\sqrt{2}((\chi^2_{T'})^{1/2} - \sqrt{T'})\bigr)
\le L-\epsilon\bigr] \le 1-\alpha
\le
\P\bigl[\Phi\bigl(\sqrt{2}((\chi^2_{T'})^{1/2} - \sqrt{T'})\bigr)
\le L+\epsilon\bigr]
$$
so that $L-\epsilon \le 1-\alpha + o(1)$ and $1-\alpha \le L+\epsilon + o(1)$
by the weak convergence \eqref{eq:approximation-q_T-alpha}.
It follows that $L=1-\alpha$ is the only accumulation point and 
 $q_{T,\alpha}
- \sqrt T \to z_{\alpha}/{\sqrt 2}$, as desired.
\end{proof}

\section{Proofs for unknown covariance}
\subsection{Asymptotic normality}
\label{subsec:proof-unknown-Sigma-normality}

\begin{proof}[Proof of \Cref{thm:unknown-Sigma-normal} under assumption \eqref{eq:assum:T-log-p-sqrt-n-unknown-Sigma}]
    We will use throughout the proof the notation
    defined after \eqref{eq:linear-model-nodewise}
    for $\tau_j,\bgamma^{(j)}$ and $\bep^{(j)}$.
    Define the direction
    $\tba_j = \be_j (\bSigma^{-1})_{jj}^{-1/2} = \tau_j \be_j$
    normalized such that $\|\bSigma^{-1/2}\tba_j\|_2=1$ by construction,
    as well as
    $\tbz_j = \bX\bSigma^{-1}\tba_j\sim \mathcal N_n(\mathbf{0},\bI_n)$.
    Next, define $\bxi_j,\hbxi_j\in\R^T$ by
    \begin{align}
        \bxi_j 
        &=  (\bY-\bX\hbB)^\top \tbz_j 
        + (n\bI_{T\times T}-\hbA)(\hbB-\bB^*)^\top\tba_j,
        \label{eq:bxi_j-whtout-hat}
        \\
        \hbxi_j 
        &=
        (\bY-\bX\hbB)^\top \hbz_j  \left[n (\hbz_j^\top \bX\be_j)^{-1} \right]
        \tau_j
        + (n\bI_{T\times T}-\hbA)(\hbB-\bB^*)^\top\tba_j
        \label{eq:hbxi_j}
    \end{align}
    so that $\bxi_j$ coincides with \eqref{eq:def-bxi}
    for the normalized direction $\tba_j$.
    Since the second term in $\bxi_j$ is the same as the second term in $\hbxi_j$,
    \begin{equation}
    \|\bxi_j - \hbxi_j\|_2
    =
    \|(\bY-\bX\hbB)^\top\Big\{\hbz_j\tau_j^{-1}  \left[n\tau_j^2 (\hbz_j^\top \bX\be_j)^{-1} \right]
    - \tbz_j
        \Big\}
    \|_2.
    \label{ineq:tildez-j-hat-z-j}
    \end{equation}
    Since
    $\bgamma^{(j)} = - (\bI_p - \be_j\be_j^\top)\bSigma^{-1}\be_j (\bSigma^{-1})_{jj}^{-1}$ in \eqref{eq:linear-model-nodewise},
    or equivalently $\be_j - \bgamma^{(j)} = \tau_j^2 \bSigma^{-1}\be_j$,
    we have
    $$\tbz_j
    = \tau_j \bX\bSigma^{-1}\be_j
    = \tau_j^{-1} \bX(\be_j - \bgamma^{(j)})
    .
    $$
    Next,
    $\hbz_j = \bX\be_j - \bX_{-j} \hbgamma^{(j)}
    = \bX[\be_j - \hbgamma^{(j)}]$
    since by definition of $\hbgamma^{(j)}$, the $j$-th coordinate
    of $\hbgamma^{(j)}$ is zero, so that $\bX_{-j}\hbgamma^{(j)} = \bX\hbgamma^{(j)}$.
    By inserting 
    $\bI_{p\times p} = \sum_{k=1}^p\be_k\be_k{}^\top$
    in \eqref{ineq:tildez-j-hat-z-j}, using 
    that the KKT conditions of $\hbB$ imply
    that $\max_{k\in[p]}\|(\bY-\bX\hbB)^\top\bX\be_k\|_2 \le nT \lambda$
    and the triangle inequality, we find
    \begin{eqnarray}
    \|\bxi_j - \hbxi_j\|_2
    &=&
    \tau_j^{-1}\Big\|(\bY-\bX\hbB)^\top
    \bX
    \sum_{k=1}^p \be_k\be_k^\top
    \Big\{(\be_j - \hbgamma^{(j)})  \left[n\tau_j^2 (\hbz_j^\top \bX\be_j)^{-1} \right]
        - (\be_j - \bgamma^{(j)})
        \Big\}
    \Big\|_2 \notag
  \\ &\le &
    \tau_j^{-1} 
  nT \lambda
  \sum_{k=1}^p
  \Big|
  \be_k^\top
      \Big\{(\be_j - \hbgamma^{(j)}) \left[n\tau_j^2 (\hbz_j^\top \bX\be_j)^{-1} \right]
        - (\be_j - \bgamma^{(j)})
        \Big\}
    \Big|  
    \label{eq:KKTbound}
  \\ &=& 
  \tau_j^{-1} nT \lambda
    \Big\|
    \Big\{(\be_j - \hbgamma^{(j)})  \left[n\tau_j^2 (\hbz_j^\top \bX\be_j)^{-1} \right]
        - (\be_j - \bgamma^{(j)})
        \Big\}
    \Big\|_1.
    \notag
    \end{eqnarray}
    There are two errors at this point: the estimation error
    $\|\hbgamma^{(j)} - \bgamma^{(j)}\|_1$
    and the estimation error $|n \tau_j^2 (\hbz_j^\top\bX\be_j)^{-1} - 1|$,
    which corresponds to the relative error of the estimation
    of the variance $\tau_j^2$ by $n^{-1}\hbz_j^\top\bX\be_j$
    in the linear model \eqref{eq:linear-model-nodewise}.
    Keeping these two errors in mind, by the triangle inequality
    the previous display yields
    \begin{equation}
    \|\bxi_j - \hbxi_j\|_2
    \le
    \frac{
        \tau_j^{-1}  nT \lambda
    }{
        (\hbz_j^\top \bX\be_j)/(n\tau_j^2)
    }
    \Big(
    \Big\|
    \bgamma^{(j)} - \hbgamma^{(j)}
    \Big\|_1
    +
    \Big\|
    \be_j - \bgamma^{(j)}
    \Big\|_1
    \Big|
    1-\frac{\hbz_j^\top\bX\be_j}{n\tau_j^2}
    \Big|
    \Big).
    \label{eq:goal-b1-b2-b3}
    \end{equation}
    For the first term in the parenthesis, inequality
    \eqref{eq:ell_1_rate_hbgamma-j}
    holds: this is the usual $\ell_1$ estimation rate for
    the Lasso estimate $\hbgamma^{(j)}$ for the sparse
    estimation target $\bgamma^{(j)}$ in the linear model
    \eqref{eq:linear-model-nodewise} with
    noise variance $\tau_j^2$.
    For the second term,
    inequality
    \begin{equation}
    \tau_j^{-1}\|\be_j - \bgamma^{(j)}\|_1
    = \tau_j \|\bSigma^{-1}\be_j\|_1
    \le
    \tau_j
    \|\bSigma^{-1}\be_j\|_0^{1/2}
    \|\bSigma^{-1}\be_j\|_2
    \le
    \|\bSigma^{-1/2}\|_{op}
    \|\bSigma^{-1}\be_j\|_0^{1/2}
    \label{eq:b2}
    \end{equation}
    holds thanks to the Cauchy-Schwarz inequality
    and $\tau_j = \|\bSigma^{-1/2}\be_j\|_2^{-1}$.
    Furthermore, by the triangle inequality, we have
    \begin{equation}
    \Big|
    1-\frac{\hbz_j^\top\bX\be_j}{n\tau_j^2}
    \Big|
    \le
    \Big|
    1-\frac{\|\bep^{(j)}\|_2^2}{n\tau_j^2}
    \Big|
    +
    \Big|
    \frac{(\bep^{(j)} - \hbz_j)^\top\bep^{(j)}}{n\tau_j^2}
    -
    \frac{\hbz_j^\top(\bX\be_j-\bep^{(j)} )}{n\tau_j^2}
    \Big|.
    \label{RHS-to-bound-unknown-sigma-gamma-j}
    \end{equation}
    As $\|\bep^{(j)}\|_2^2/\tau_j^2$ has $\chi^2_n$ distribution,
    the first term is $O(n^{-1/2})$ by the Central Limit Theorem.
    For the next term, we use again the triangle inequality.
    To bound the next term, notice that by Hölder's inequality,
    $$
    |(\bep^{(j)} - \hbz_j)^\top\bep^{(j)}|
    =
    |(\hbgamma^{(j)}-\bgamma^{(j)})^\top\bX_{-j}^\top\bep^{(j)}|
    \le
    \|\hbgamma^{(j)} -\bgamma^{(j)}\|_1 \|\bX_{-j}{}^\top\bep^{(j)}\|_\infty
    .
    $$
    Each factor in the right hand side
    is bounded from above as follows:
    $\|\hbgamma^{(j)} -\bgamma^{(j)}\|_1 =
    \tau_j \|\bSigma^{-1}\|_{op} \|\bgamma^{(j)}\|_0 O_\P\big(\sqrt{n^{-1}\log p} \big)
    $ thanks to
    \eqref{eq:ell_1_rate_hbgamma-j}
    and 
    $\|\bX_{-j}{}^\top\bep^{(j)}\|_\infty = \tau_j O_\P(\sqrt{n \log p})$
    because $\bX_{-j}$ is independent of $\bep^{(j)}$
    and $\max_{k\in[p]\setminus \{j\}}\bSigma_{kk}\le 1$.
    This proves that
    $|(\bep^{(j)} - \hbz_j)^\top\bep^{(j)}|/(n\tau_j^2)
    \le
    \|\bSigma^{-1}\|_{op}\|\bgamma^{(j)}\|_0 O_\P(n^{-1}\log p)
    $.
    We also have
    $$
    |\hbz_j^\top(\bX\be_j - \bep^{(j)})|
    =
    |\hbz_j^\top \bX_{-j}\bgamma^{(j)}|
    \le
    \|\hbz_j^\top\bX_{-j}\|_\infty \|\bgamma^{(j)}\|_1
    \le
    O_\P(\tau_j\sqrt{n\log p}) \|\bgamma^{(j)}\|_1
    $$
    thanks to Hölder's inequality and
    the KKT conditions for $\hbgamma^{(j)}$ in \eqref{eq:KKT-hbgamma-j}
    to bound the $\ell_\infty$ norm.
    We have
    $\|\bgamma^{(j)}\|_1 \le 
    \|\bgamma^{(j)}\|_0^{1/2} \|\bgamma^{(j)}\|_2
    $
    and  $\|\bgamma^{(j)}\|_2\le \tau_j^2 \|\bSigma^{-1}\be_j\|_2
    \le \tau_j \|\bSigma^{-1/2}\|_{op}$
    by definition of $\bgamma^{(j)}$
    and the Cauchy-Schwarz inequality.
    Combining these bounds provide an upper bound on the right hand side
    of \eqref{RHS-to-bound-unknown-sigma-gamma-j}, so that
    \begin{align}
        \nonumber
    \Big|
    1-\frac{\hbz_j^\top\bX\be_j}{n\tau_j^2}
    \Big|
    &\le O_\P\Big(\frac{1}{\sqrt n}\Big)
    +
    \|\bSigma^{-1}\|_{op} \|\bgamma^{(j)}\|_0 O_\P\Big(\frac{\log p}{n}\Big)
    +
    \|\bSigma^{-1/2}\|_{op}
    \Big(\frac{\|\bgamma^{(j)}\|_0\log p}{n}\Big)^{1/2}
    \\
    &\le
    \|\bSigma^{-1}\|_{op}
    \big({\|\bgamma^{(j)}\|_0\log(p)}/{n}\big)^{1/2}
    O_\P(1)
    \label{eq:b3}
    \end{align}
    where the second line follows by bounding from above
    the first two terms thanks to assumption
    \eqref{eq:assum:T-log-p-sqrt-n-unknown-Sigma}
    and
    $\|\bSigma^{-1}\|_{op}\ge 1$ (this is a consequence of $\bSigma_{jj}\le 1$ in \Cref{assum:main}).
    The bound \eqref{eq:b3} also provides
    $\hbz_j^\top\bX\be_j/(n\tau_j^2)\smash{\xrightarrow[]{\P}} 1$
    and thus $n\tau_j^2/(\hbz_j^\top\bX\be_j) = O_\P(1)$.
    Using \eqref{eq:ell_1_rate_hbgamma-j}, \eqref{eq:b2} and \eqref{eq:b3} to bound
    from above the right hand side of \eqref{eq:goal-b1-b2-b3} we
    find
    \begin{align*}
        \|\bxi_j - \hbxi_j\|_2
        \le  nT\lambda
        \Big(
            &\|\bSigma^{-1}\|_{op}
            \|\bgamma^{(j)}\|_0 O_\P(\sqrt{n^{-1}\log p})\\
            &+
            \|\bSigma^{-1/2}\|_{op}
            \|\bSigma^{-1}\be_j\|_0^{1/2} \|\bSigma^{-1}\|_{op} \|\bgamma^{(j)}\|_0^{1/2} O_\P(\sqrt{n^{-1}\log p})
        \Big).
    \end{align*}
    Since $\|\bgamma^{(j)}\|_0= \|\bSigma^{-1}\be_j\|_0 - 1$,
    this implies
    $\|\bxi_j - \hbxi_j\|_2
    \le  nT \lambda \|\bSigma^{-1}\|_{op}^{3/2}
    \|\bSigma^{-1}\be_j\|_0
    O_\P(\sqrt{n^{-1}\log p})$.
    Thanks to $\lambda=O\big(\sigma(nT)^{-1/2})(1+\sqrt{\log(p/s)/T})\big)$
    by definition of $\lambda$, we eventually obtain
    \begin{equation}
        \label{eq:hbxi_j-bxi_j-converges-to-0}
    (\sigma^2 n )^{-1/2}\|\hbxi_j - \bxi_j\|_2 
    = O_\P
    \big(
        [\sqrt{T} + \sqrt{\log(p/s)}] \|\bSigma^{-1}\be_j\|_0 \sqrt{\log(p)/n}
    \big)
    \end{equation}
    which converges to 0 in probability
    thanks to assumption \eqref{eq:assum:T-log-p-sqrt-n-unknown-Sigma}.

    To complete the proof of \Cref{thm:unknown-Sigma-normal} and prove asymptotic normality
    for some fixed $\bb\in\R^T$ with $\|\bb\|_2=1$, notice that
    $$
    \zeta_j \coloneqq
    \frac{ 
                n \tba_j^\top (\hbB - \bB^*) \bb
                + \tbz_j^\top(\bY-\bX\hbB)(\bI_{T\times T}-\hbA/n)^{-1}
                \bb
    }{
        \|(\bY-\bX\hbB)
        (\bI_{T\times T}-\hbA/n)^{-1}
        \bb\|_2
}
    $$
    satisfies $\zeta_j \smash{\xrightarrow[]{d}} \mathcal N(0,1)$ by \Cref{thm-normal} 
    applied to the normalized direction $\tba_j$.
    Furthermore,
    \begin{align*}
        &\Big|\zeta_j
        -
        \frac{
            n \be_j^\top (\hbB - \bB^*) \bb
            + n (\hbz_j^\top \bX\be_j)^{-1}
            \hbz_j^\top(\bY-\bX\hbB)(\bI_{T\times T}-\hbA/n)^{-1}
            \bb
        }
        {
            (\tau_j)^{-1}  ~
        \|(\bY-\bX\hbB)
        (\bI_{T\times T}-\hbA/n)^{-1}
        \bb\|_2
        }
        \Big|
        \\&=
        \frac{
            |(\bxi_j-\hbxi_j)^\top(\bI_{T\times T} -\hbA/n)^{-1} \bb|
        }{
        \|(\bY-\bX\hbB)
        (\bI_{T\times T}-\hbA/n)^{-1}
        \bb\|_2
    }
        \\&\le
        (\sigma^2n)^{-1/2}\|\bxi_j-\hbxi_j\|_2
        \|(\bI_{T\times T}-\hbA/n)^{-1}\|_{op}
        \Big(
        \|(\bY-\bX\hbB)
        (\bI_{T\times T}-\hbA/n)^{-1}
        \bb\|_2^{-1}(\sigma^2n)^{1/2}
        \Big).
    \end{align*}
    In the above display,
    $(\sigma^2n)^{-1/2}\|\bxi_j-\hbxi_j\|_2\smash{\xrightarrow[]{\P}} 0$
    when \eqref{eq:assum:T-log-p-sqrt-n-unknown-Sigma} holds,
    $\|(\bI_{T\times T}-\hbA/n)^{-1}\|_{op}\smash{\xrightarrow[]{\P}} 1$
    by \Cref{propMatrixA}(iii) and \Cref{lemma:sparsity},
    and the rightmost factor converges to 1 in probability
    by \Cref{thm:variance}.

    Since $\tau_j=(\bSigma^{-1})_{jj}^{-1/2}$,
        the last claim follows by
    $\hbz_j^\top\bX\be_j/(n\tau_j^2)\smash{\xrightarrow[]{\P}} 1$
    by \eqref{eq:b3} and Slutsky's theorem.
    We also have
    $\|\hbz_j\|_2/(\tau_j \sqrt n )\smash{\xrightarrow[]{\P}} 1$ since,
    using \eqref{eq:KKT-hbgamma-j} and the triangle inequality,
    \begin{align}
        (\tau_j^2n)^{-1} | \|\hbz_j\|_2^2 - \hbz_j^\top\bX\be_j|
       &   =
        (\tau_j^2n)^{-1}
        |\hbz_j^\top\bX_{-j}\hbgamma^{(j)}|
     \cr&\le
        (\tau_j^2n)^{-1}
        O_\P(1) \tau_j\sqrt{n\log p}
        \|\hbgamma^{(j)}\|_1
     \cr&\le O_\P(1) \sqrt{\log(p)/n}\big[\|\hbgamma^{(j)} - \bgamma^{(j)}\|_1 + \|\bgamma^{(j)}\|_1\big] /\tau_j
     \cr&\le O_\P(1) \sqrt{\log(p)/n}\big[\|\bgamma^{(j)}\|_0\sqrt{\log(p)/n} +
     \|\bgamma^{(j)}\|_0^{1/2}
     \big]
     \label{eq:consistency-2}
     \\&= o_\P(1)
     \nonumber
    \end{align}
    thanks to \eqref{eq:ell_1_rate_hbgamma-j} for the first term
    and the Cauchy-Schwarz inequality for the second.
    The convergence to 0 in probability in the last line follows from
    \eqref{eq:assum:T-log-p-sqrt-n-unknown-Sigma}.
\end{proof}

\begin{proof}[Proof of \Cref{thm:unknown-Sigma-normal} under assumption \eqref{eq:assum-unknown-Sigma-second-version}]
    With $\hbxi_j$ in \eqref{eq:hbxi_j} 
    and $\tbxi_j \coloneqq \bE^\top\hbz_j[n(\hbz_j^\top\bX\be_j)^{-1}]\tau_j$
    we have
    \begin{align}
        \nonumber
       \|\hbxi_j - \tbxi_j\|_2
       &=
       \|-\hbA(\hbB-\bB^*)^\top\tba_j
       +
       \tau_j
       (\bB^*-\hbB)^\top\bX_{-j}^\top
       \hbz_j
       [n(\hbz_j^\top\bX\be_j)^{-1}]
       \|_2
     \\&\le
     \|\hbA\|_{op}
     \|\bSigma^{1/2}(\hbB - \bB^*)\|_{op}
     +
     \tau_j
     \|\hbB - \bB^*\|_{2,1}
     \|\bX_{-j}^\top\hbz_j\|_{\infty}
     |[n(\hbz_j^\top\bX\be_j)]^{-1}|
     \label{eq:def-tbxi_j-just-above}
    \end{align}
    thanks to $\|\bSigma^{-1/2}\tba_j\|_2=1$ for the first term
    and
    Hölder's inequality for the second
    term.
    Thanks to \Cref{lemma:risk}(iii), \Cref{lemma:sparsity} and \Cref{propMatrixA} we find
    $
     \|\hbA\|_{op}
     \|\bSigma^{1/2}(\hbB - \bB^*)\|_{op}
     = O_\P(
         \bar s 
         \bar R
         )
    $.
    For the second term, thanks to \eqref{eq:KKT-hbgamma-j}
    and \Cref{lemma:risk}(iv) we have
    $$
     \tau_j
     \|\hbB - \bB^*\|_{2,1}
     \|\bX_{-j}^\top\hbz_j\|_{\infty}
     |[n(\hbz_j^\top\bX\be_j)]^{-1}|
     \le
     O_\P(\sqrt s \bar R)
     \sqrt{n\log p}
     \big|[n\tau_j^2(\hbz_j^\top\bX\be_j)]^{-1}\big|
    $$
    and the bound \eqref{eq:b3} grants
    $
    \hbz_j^\top\bX\be_j/(n\tau_j^2)\smash{\xrightarrow[]{\P}} 1
    $ thanks to the leftmost assumption
    in \eqref{eq:assum-unknown-Sigma-second-version}.
    In summary,
    $(\sigma^2 n)^{-1/2}\|\hbxi_j - \tbxi_j\|_2
    = O_\P(
        n^{-1/2}s \bar R
        +
        \sqrt{s} \bar R \sqrt{\log p}
    )
    =O_\P(
        \sqrt{s} \bar R \sqrt{\log p}
    )
    $
    thanks to $n^{-1/2}\sqrt s \le 1$.
    Hence due to the rightmost assumption
    in \eqref{eq:assum-unknown-Sigma-second-version},
    \begin{equation}
        \label{converge-to-0-hbxi_j-tbxi_j}
    (\sigma^2n)^{-1/2}\|\hbxi_j - \tbxi_j\|_2\smash{\xrightarrow[]{\P}} 0.
    \end{equation}

    Next, assume without loss of generality that $\|\bb\|_2=1$.
    By definition of $\hbxi_j$ in \eqref{eq:hbxi_j},
    \begin{align}
      &  \frac{
            n \be_j^\top (\hbB - \bB^*) \bb
            + n (\hbz_j^\top \bX\be_j)^{-1}
            \hbz_j^\top(\bY-\bX\hbB)(\bI_{T\times T}-\hbA/n)^{-1}
            \bb
        }
        {
            (\bSigma^{-1})_{jj}^{1/2}  ~
            \sigma\sqrt n
        }
        - \frac{\tbxi_j^\top\bb}{\sigma\sqrt n}
    \cr&=
    \frac{
        (\hbxi_j-\tbxi_j)^\top(\bI_{T\times T} - \hbA/n)^{-1}\bb
        }{
            \sigma\sqrt n
        }
        -
        \frac{\tbxi_j^\top(\bI_{T\times T}-(\bI_{T\times T} - \hbA/n)^{-1})\bb}{\sigma\sqrt n}.
        \label{eq:two-terms-unknown-Sigma-2}
    \end{align}
    The first term converges to $0$ in probability thanks to
    the previous paragraph, while the second
    term is $O_\P(s/n) \|\tbxi_j\|_2 (\sigma^2n)^{-1/2}$
    by \Cref{propMatrixA} and \Cref{lemma:sparsity}.
    $\bzeta_j \coloneqq [n\tau_j^2(\hbz_j^\top\bX\be_j)^{-1}]^{-1}\tau_j\|\hbz_j\|_2^{-1}\tbxi_j$ has $\mathcal N_T(\mathbf{0},\sigma^2\bI_{T\times T})$
    distribution
    by independence of $\bE$ and $\bX$.
    Next,
    $\|\tbxi_j\|_2 = \|\hbz_j\|_2 [n\tau_j^2(\hbz_j^\top\bX\be_j)^{-1}]
    \|\bzeta_j\|_2$
    and $\|\bzeta_j\|_2= O_\P(\sqrt T)$ since $\E[\|\bzeta_j\|_2^2]=T$.
    Furthermore,
    $n\tau_j^2(\hbz_j^\top\bX\be_j)^{-1}\smash{\xrightarrow[]{\P}} 1$
    by \eqref{eq:b3}.
    We also have
    $\tau_j \sqrt n \|\hbz_j\|_2^{-1}\smash{\xrightarrow[]{\P}} 1$ by
    \eqref{eq:consistency-2},
    thanks to the leftmost assumption in \eqref{eq:assum-unknown-Sigma-second-version} for the last line in \eqref{eq:consistency-2}.
    This shows that $\|\tbxi_j\|_2/(\sqrt n \|\bzeta_j\|_2) \smash{\xrightarrow[]{\P}} 1$ 
    and that the second term in \eqref{eq:two-terms-unknown-Sigma-2}
    is $O_\P(s/n)$ and converges to $0$ in probability.
    We conclude by observing that $\tbxi_j^\top\bb/(\sigma\sqrt n)\smash{\xrightarrow[]{d}} \mathcal N(0,1)$
    by Slutsky's theorem
    thanks to
    $n\tau_j^2(\hbz_j^\top\bX\be_j)^{-1}\smash{\xrightarrow[]{\P}} 1$
    and
    $\tau_j \sqrt n \|\hbz_j\|_2^{-1}\smash{\xrightarrow[]{\P}} 1$.
    In the denominator,
    $\|(\bY-\bX\hbB)
        (\bI_{T\times T}-\hbA/n)^{-1}
        \bb\|_2$
    and $\sigma\sqrt n$ can be used interchangeably,
    again by Slutsky's theorem,
    since 
    $\|(\bY-\bX\hbB) (\bI_{T\times T}-\hbA/n)^{-1} \bb\|_2
        /
        (\sigma\sqrt n) \smash{\xrightarrow[]{\P}} 1$
    by \Cref{thm:variance}.
\end{proof}

\subsection{Asymptotic \texorpdfstring{$\chi^2_T$}{chi\texttwosuperior} distribution}
\label{subsec:proof-unknown-Sigma-chi2}
\begin{proof}[Proof of \Cref{thm:unknown-Sigma-chi2} under
    assumption \eqref{eq:assum-unknown-Sigma-second-version}
    ]
    Let $\hbxi_j$ and $\tbxi_j$ be defined respectively 
    in \eqref{eq:hbxi_j} and in the sentence preceding
\eqref{eq:def-tbxi_j-just-above}.
    Notice that the quantity in the left hand side of \eqref{eq:conclusion-chi2} is equal to 
        $(1-T/n)^{1/2}\| \hbGamma{}^{-1/2} \bxi \|_2$
        where
        \begin{equation}
            \label{eq:xi-choice-unknown-Sigma}
        \bxi
        =
            \bigl(\tau_j^{-1}
                \frac{\sqrt n}{\|\hbz_j\|_2}
                \bigl[n (\hbz_j^\top \bX\be_j)^{-1} \bigr]^{-1}
            \bigr)
            \hbxi_j
            .
        \end{equation}
        Set $\bz=\hbz_j / \|\hbz_j\|_2$.
        For these values of $\bxi$ and $\bz$,
 we have
    \begin{align*}
        (\sigma^2n)^{-1/2}\|\bxi - \sqrt{n} \bE^\top\bz\|_2
         &
    = (\sigma^2n)^{-1/2}
    \Big\|
    \hbxi_j
     \tau_j^{-1}
    \frac{\sqrt n}{\|\hbz_j\|_2}
    \Bigl[n (\hbz_j^\top \bX\be_j)^{-1} \Bigr]^{-1}
    -
    \bE^\top\hbz_j \frac{\sqrt n}{\|\hbz_j\|_2}
    \Big\|_2
        \\&=
    (\sigma^2n)^{-1/2}
    \Big\|
        \hbxi_j - \tbxi_j 
    \Big\|_2
    \tau_j^{-1} \frac{\sqrt n}{\|\hbz_j\|_2}
    [n (\hbz_j^\top \bX\be_j)^{-1} ]^{-1}.
    \end{align*}
    Hence the above is $o_\P(1)$ by combining
    \eqref{converge-to-0-hbxi_j-tbxi_j}
    with \eqref{eq:b3} and \eqref{eq:consistency-2}.
    An application of \Cref{lemma:bxi-bz} 
    for these values of $\bz$ and $\bxi$
    yields
    \eqref{eq:conclusion-lemma-bxi-upper-bound}
    which completes the proof.
\end{proof}

\begin{proof}[Proof of \Cref{thm:unknown-Sigma-chi2} under
    assumption \eqref{eq:assum:T-log-p-sqrt-n-unknown-Sigma}
    ]
    Let $\bxi_j$ be defined in \eqref{eq:bxi_j-whtout-hat}
    Since $\tba_j,\tbz_j$ defined in the proof of
    \Cref{thm:unknown-Sigma-normal} satisfy
    the assumptions of \Cref{thm:chi2z0}, we have already
    established that
    $(\sigma^2 n )^{-1/2}
        \|
            \bxi_j - \sqrt n \bE^\top  \tbz_j\|\tbz_j\|_2^{-1}
        \|_2
    = o_\P(1)$, cf. \eqref{eq:conclusion-thm-chi2initial} 
    with $\ba=\tba_j$ and $\bz_0 = \tbz_j$.

    We now proceed to show that
    $(\sigma^2n)^{-1/2}\|\bxi - \bxi_j\|_2
    = o_\P(1)$
    for $\bxi$ defined in \eqref{eq:xi-choice-unknown-Sigma}.
    By the triangle inequality
    and since $\hbxi_j$ in \eqref{eq:xi-choice-unknown-Sigma}
    is proportional to $\hbxi_j$, we have
    \begin{align*}
     & (\sigma^2n)^{-1/2}\|\bxi - \bxi_j\|_2
        \\&=
        \frac{1}{\sigma\sqrt n}
    \Big\|_2 
    \hbxi_j
    \frac{\tau_j^{-1}\sqrt n}{\|\hbz_j\|_2}
    \Bigl[n (\hbz_j^\top \bX\be_j)^{-1} \Bigr]^{-1}
    -
    \bxi_j
    \Big\|_2 
       \\&=
        \frac{1}{\sigma\sqrt n}
        \Big\|
        ( \hbxi_j - \bxi_j)
    \frac{\tau_j^{-1}\sqrt n}{\|\hbz_j\|_2}
    \Bigl[n (\hbz_j^\top \bX\be_j)^{-1} \Bigr]^{-1}
            +
            \bxi_j
            \Bigl(
                \frac{\tau_j^{-1}\sqrt n}{\|\hbz_j\|_2}
    \Bigl[n (\hbz_j^\top \bX\be_j)^{-1} \Bigr]^{-1}
    -1 
            \Bigr)
        \Big\|_2
    \\&\le
        \frac{1}{\sigma\sqrt n}
        \Big\|
        ( \hbxi_j - \bxi_j)
        \Big\|_2
    \frac{\tau_j^{-1}\sqrt n}{\|\hbz_j\|_2}
    \Bigl[n (\hbz_j^\top \bX\be_j)^{-1} \Bigr]^{-1}
            +
        \Big\|
            \bxi_j
        \Big\|_2
        \Bigl|
            \frac{\tau_j^{-1}\sqrt n}{\|\hbz_j\|_2}
            \Bigl[n (\hbz_j^\top \bX\be_j)^{-1} \Bigr]^{-1}
            -1 
        \Bigr|.
    \end{align*}
    For the first term, by \eqref{eq:hbxi_j-bxi_j-converges-to-0}
    we already have
    $(\sigma^2 n)^{-1/2}\|\tbxi_j - \bxi_j\|_2
    = o_\P(1)$.
    Combined
    with $\|\hbz_j\|_2/(\tau_j\sqrt n)\smash{\xrightarrow[]{\P}} 1$
    and 
    $n \tau_j^2 (\hbz_j^\top\bX\be_j)^{-1}\smash{\xrightarrow[]{\P}} 1$
    (see \eqref{eq:b3} and \eqref{eq:consistency-2}),
    this proves that the first term above is $o_\P(1)$
    For the remaining terms,
    $(\sigma^2 n)^{-1/2}\|\bxi_j\|_2 = O_\P(\sqrt T)$
    by \eqref{eq:conclusion-thm-chi2initial},
    and the question is whether
    \begin{equation}
    O_\P(\sqrt T)
    \Big|
    \frac{\tau_j^{-1}\sqrt n}{\|\hbz_j\|_2}
    \Bigl[n (\hbz_j^\top \bX\be_j)^{-1} \Bigr]^{-1}
    -
    1\Big|
    \label{eq:how-fast-converge-to-0}
    \end{equation}
    converges to 0 using \eqref{eq:b3}
    and \eqref{eq:consistency-2}.
    With $a_j = \|\hbz_j\|_2^2/(\tau_j^2 n)$
    and $b_j = \hbz_j^\top\bX\be_j/(\tau_j^2n)$ for brevity,
    \begin{align*}
    \big|
    \tau_j^{-1}
    \frac{\sqrt n}{\|\hbz_j\|_2}
    \Bigl[n (\hbz_j^\top \bX\be_j)^{-1} \Bigr]^{-1}
    -
    1\big|
  &=
    a_j^{-1/2}
    | b_j - a^{1/2} |
  \\&\le
    a_j^{-1/2}
    (
    |b_j - 1|
    +
    |1- a_j^{1/2}|)
  \\&=
    a_j^{-1/2}
    (
    |b_j - 1|
    + |1 - a_j|(1+a_j)^{-1}
    ).
    \end{align*}
    We have
    $|a_j - 1| + |b_j-1| = \sqrt{\|\bgamma^{(j)}\|_0\log(p)/n} ~ O_\P(1)$
    thanks to \eqref{eq:b3} and \eqref{eq:consistency-2}.
    Hence thanks to \eqref{eq:assum:T-log-p-sqrt-n-unknown-Sigma},
    quantity \eqref{eq:how-fast-converge-to-0} is $o_\P(1)$.
    Combining all the pieces, we have proved that
    $$
    (\sigma^2n)^{-1/2}\|\bxi
    - \sqrt n\bE^\top\tbz_j\|\tbz_j\|_2^{-1}\|_2
    \le
    (\sigma^2n)^{-1/2}\|\bxi - \bxi_j\|_2
    + o_\P(1)
    \le o_\P(1).
    $$
    Applying 
    \Cref{lemma:bxi-bz} to $\bxi$ in \eqref{eq:xi-choice-unknown-Sigma}
    and $\bz=\tbz_j\|\tbz_j\|_2^{-1}$,
    conclusion
    \eqref{eq:conclusion-lemma-bxi-upper-bound} completes the proof.
\end{proof}

\end{document}

%% file: macros.tex
\def\df{\hat{\mathsf{df}}}

\def\argmin{\mathop{\rm arg\, min}}

\def\bBbar{{\overline \bB}}
\def\bHbar{{\overline \bH}}

\def\R{{\mathbb{R}}}
\def\E{{\mathbb{E}}}

\def\P{{\mathbb{P}}}

\DeclareMathOperator{\trace}{Tr}
\DeclareMathOperator{\range}{range}
\DeclareMathOperator{\rank}{rank}
\let\vec\relax
\DeclareMathOperator{\vec}{\mathbf{vec}}
\DeclareMathOperator{\diag}{\mathbf{diag}}

\DeclareMathOperator{\supp}{supp}

\def\mathbold{\boldsymbol} 

\def\ba{\mathbold{a}}
\def\tba{{\widetilde{\ba}}}
\def\bA{\mathbold{A}}
\def\hbA{\smash{\widehat{\mathbf A}}}

\def\bb{\mathbold{b}}

\def\bB{\mathbold{B}}
\def\hbB{{\widehat{\bB}}}

\def\bC{\mathbold{C}}

\def\bD{\mathbold{D}}

\def\bE{\mathbold{E}}

\def\bF{\mathbold{F}}

\def\bg{\mathbold{g}}

\def\bG{\mathbold{G}}

\def\bH{\mathbold{H}}
\def\tbH{{\tilde{\bH}}}

\def\bI{\mathbold{I}}

\def\bJ{\mathbold{J}}

\def\bM{\mathbold{M}}

\def\bn{\mathbold{n}}

\def\bN{\mathbold{N}}

\def\bP{\mathbold{P}}

\def\bQ{\mathbold{Q}}

\def\br{\mathbold{r}}
\def\tbr{{\widetilde{\br}}}

\def\be{\mathbold{e}}

\def\bu{\mathbold{u}}

\def\bU{\mathbold{U}}

\def\bv{\mathbold{v}}

\def\bV{\mathbold{V}}

\def\bw{\mathbold{w}}

\def\bW{\mathbold{W}}

\def\bx{\mathbold{x}}

\def\bX{\mathbold{X}}\def\tbX{{\tilde{\bX}}}\def\bXbar{{\overline \bX}}

\def\by{\mathbold{y}}

\def\bY{\mathbold{Y}}

\def\bz{\mathbold{z}}
\def\hbz{{\widehat{\bz}}}\def\tbz{{\widetilde{\bz}}}

\def\bbeta{\mathbold{\beta}}

\def\hbbeta{\hat{\bbeta}{}}

\def\bgamma{\mathbold{\gamma}}

\def\hbgamma{{\widehat{\bgamma}}}

\def\bGamma{\mathbold{\Gamma}}\def\hbGamma{{\widehat{\bGamma}}}

\def\ep{\varepsilon}\def\eps{\epsilon}
\def\bep{ {\mathbold{\ep} }}

\def\bzeta{\mathbold{\zeta}}

\def\btheta{\mathbold{\theta}}

\def\bLambda{\mathbold{\Lambda}}
\def\tbLambda{{\widetilde{\bLambda}}}

\def\bxi{\mathbold{\xi}}

\def\hbxi{{\widehat{\bxi}}}\def\tbxi{{\widetilde{\bxi}}}

\def\bPi{\mathbold{\Pi}}

\def\bSigma{\mathbold{\Sigma}}